\documentclass[11pt]{article}

\usepackage[
    left=0.8in,
    right=0.8in,
    top=0.8in,
    bottom=0.9in
]{geometry}

\usepackage{epsfig}
\usepackage{mathrsfs}
\usepackage{amssymb,amsmath,amsthm,amscd}
\usepackage{latexsym}
\usepackage{bbm,dsfont}
\usepackage{enumitem}

\usepackage[colorlinks,
            linkcolor=blue,
            anchorcolor=blue,
            citecolor=blue]{hyperref}

\newtheorem{thm}{Theorem}[section]
\newtheorem{prop}[thm]{Proposition}
\newtheorem{cor}[thm]{Corollary}
\newtheorem{lem}[thm]{Lemma}

\theoremstyle{definition}
\newtheorem{defn}[thm]{Definition}
\newtheorem{rem}[thm]{Remark}
\newtheorem{assumption}{Assumption}[section]

\newcommand{\be}{\begin{equation}}
\newcommand{\ee}{\end{equation}}

\theoremstyle{plain}
\newtheorem{THM}{Theorem}

\newcommand{\1}{\mathds{1}}
\newcommand{\W}{\mathbb{W}}

\def \a{\alpha} \def \b{\beta} \def \g{\gamma} \def \d{\delta}
   
\def \s{\sigma} \def \l{\lambda}  
\def \O{\Omega}   
\def \k{\kappa}  \def \G{\Gamma}
\def \di{\mathrm{dist}} \def \spa{\mathrm{span}}
 
\def \L{\Lambda}

\def\Lip{\mathrm{Lip}}

\def\Eb{\mathbf{E}}

\def\Rb{\mathbb{R}}

\def\Ac{\mathcal{A}}
\def\Fc{\mathcal{F}}
\def\Nc{\mathcal{N}}

\def \eps {{ \varepsilon }}

\begin{document}

\baselineskip=1.2\baselineskip

\pagestyle{plain}

\title{Exponential mixing and Freidlin--Wentzell large deviation principle for Markov cocycles}

\author{
Rongchang Liu \footnote{
  School of Mathematics,
  Sichuan University,
Chengdu, Sichuan 610064, PR China,
Email: rcliu@scu.edu.cn},\;
 Kening Lu \footnote{
    School of Mathematics,
    Sichuan University,
  Chengdu, Sichuan 610064, PR China,
 Email: keninglu@scu.edu.cn},\;
 Lin Shi \footnote{
  School of Mathematical Sciences, 
  University of Electronic Science and Technology of China,
  Chengdu, Sichuan 611731, PR China,
Email: shilinlavender@163.com},\;
  Bixiang Wang \footnote{
Department of Mathematics, New Mexico Institute of Mining and
Technology,  Socorro,  NM~87801, USA,
Email: bwang@nmt.edu}
}

\date{}

\maketitle

\begingroup
\renewcommand{\thefootnote}{}
\footnotetext{
This work was supported by the Fundamental Research Funds for the Central Universities and the National Natural Science Foundation of China (Grant Nos.~12090010, 12090013 and 12471154). 
}
\endgroup

\begin{abstract}
  This paper studies the long-time statistics and small noise asymptotics of Markov cocycles associated with Markov processes in random environments modeled by measure-preserving dynamical systems on a standard Borel probability space. 
  
  Our first result provides an abstract criterion for exponential mixing of stationary measures for such cocycles, formulated toward SPDE applications with assumptions that can be verified directly from a priori estimates. To overcome the nonuniformity from the environment, we combine generalized coupling arguments with ergodic theoretic methods. This allows us to convert nonuniform estimates along the environment into contraction on a positive density set of times, and then upgrade this to all time contraction by introducing a block gap-counting argument.
  
  Our second result establishes a Freidlin--Wentzell large deviation principle (LDP) for the unique stationary measure in the small noise limit with a good rate function. For the upper bound, the noise is allowed to be degenerate, while the deterministic pullback attractor may have nontrivial dynamics. 
  
  The abstract theory applies to nonautonomous SPDEs. We illustrate it with two examples: the two-dimensional Navier--Stokes equations on bounded domains and damped Sine--Gordon equations, where both the deterministic forcing and the degenerate additive noise depend on the random environment. 
  \end{abstract}



  
  \tableofcontents
  
\section{Introduction}

 Markov processes in random environments form an important class of time-inhomogeneous Markov systems in which the time dependence possesses a coherent dynamical structure. In this work, we study the long-time statistics and small noise asymptotics of such systems. We first provide sufficient conditions for the existence of a unique stationary measure that attracts transition probabilities exponentially. We then establish a criterion ensuring that this unique stationary measure satisfies a Freidlin--Wentzell large deviation principle in the small noise limit. The results are formulated under assumptions that, in applications to SPDEs, can be verified directly through a priori estimates. In particular, they provide a unified framework for studying the statistical dynamics of a broad class of stochastic equations with time-dependent coefficients, including the periodic, quasi-periodic, and almost periodic settings that have been extensively studied. 

 Let $(E,\rho)$ be a Polish space with distance $\rho$, and let $\{\beta_t\}$ be an ergodic invertible measure-preserving flow on the standard Borel probability space $(\Sigma, m)$. A Markov process with state space $E$ in the random environment $(\Sigma, \beta_t, m)$ is a family of processes $\{X_t^\sigma\}_{t\ge0,\sigma\in\Sigma}$ such that, for each $\sigma\in\Sigma$, the process $X_t^\sigma$ is Markov on $(\Omega,\mathcal F,\mathbf P)$ and its transition operator
 \[P_t^{\sigma}f:=\mathbf{E}f(X_t^{\sigma}), \text{ for bounded measurable } f: E\to\mathbb{R}, \]
 is Feller and generates a Markov cocycle that satisfies the cocycle Chapman-Kolmogorov relation
  \[P_{t+s}^\sigma=P_s^\sigma \circ  P_t^{\beta_s\sigma},\qquad s,t\geq0.\]
  A typical example is the Markov process generated by a stochastic differential equation with coefficients quasi-periodic in time, where $\beta_t$ is the irrational rotation flow.

  The corresponding invariant object is not a single probability measure on $E$, but a family $\{\mu_\sigma\}_{\sigma\in\Sigma}$ satisfying for $m$-a.e. $\sigma$ that 
  \[
    \mu_\sigma P_t^\sigma=\mu_{\beta_t\sigma},\qquad t\geq0, 
  \]
  and we refer to such a family as a stationary measure. We say that the stationary measure is unique if any two stationary measures coincide for $m$-a.e. $\sigma$. It is called mixing if, for every probability measure $\nu$ on $E$,
  \[
  d(\nu P_t^\sigma,\mu_{\beta_t\sigma})\to0
  \qquad\text{as }t\to\infty,
  \]
  where $d$ is any metric generating the topology of weak convergence on the space of probability measures on $E$.

  \subsection{Unique ergodicity and exponential mixing}

  For time-homogeneous Markov processes, the theory of unique ergodicity and mixing is well developed; see, for instance, the comprehensive monograph of Meyn and Tweedie \cite{meynTweedie}. In infinite dimensions, however, the absence of a canonical reference measure and the possible failure of the strong Feller property create substantial difficulties. Over the past decades, several robust approaches have been developed, see for example \cite{kuk1,HM2006,HM2011,BKS2020} and references therein. 
  
  The time-inhomogeneous case is considerably less understood, especially for infinite-dimensional systems. Kifer \cite{kifer1996} proved a Perron--Frobenius type stability theorem for Markov cocycles on compact state spaces under a randomized Doeblin condition. Feng--Qu--Zhao \cite{FQZ2023} developed a Harris-type theory for general time-inhomogeneous Markov systems under a local Doeblin condition. These Doeblin-type assumptions are powerful in finite dimensions, where they can often be verified through lower bounds on transition densities, but they are generally unavailable for SPDEs. Indeed, transition probabilities in infinite-dimensional spaces rarely admit densities with respect to a common reference measure, and the total variation metric is usually too strong for degenerate-noise SPDEs. Existing results for nonautonomous SPDEs often rely on nondegenerate noise or additional structure in the time dependence. Da Prato--Debussche~\cite{da1} and related works considered time-periodic or stationary settings with nondegenerate noise whose coefficients are independent of time. Liu--Lu~\cite{liur1} studied the two-dimensional Navier--Stokes equation with quasi-periodic deterministic forcing and degenerate noise of Hairer--Mattingly type~\cite{HM2006}, their approach uses the compact torus structure of the underlying irrational rotation. 
  
  In contrast, our framework allows the environment to be merely a standard Borel probability space equipped with an ergodic measurable flow. This unifies  periodic, quasi-periodic, almost periodic, and stationary random settings, but does not rely on compactness or continuity of the symbol space. In particular, it applies to SPDEs with degenerate noise and with noise coefficients depending on time. An informal version of our first main result is the following.
  
  \begin{THM}\label{t.060701}
  Assume that the Markov cocycle satisfies a random Lyapunov structure and that there exists a generalized asymptotic coupling with controlled random Girsanov cost. Then there is a unique stationary measure $\mu_{\sigma}$ which exponentially attracts the transition probabilities both forward and backward in time.  
  \end{THM}

  We defer the technical assumptions $(H1)$--$(H4)$ and more detailed conclusions to Theorem~\ref{t.042501}. The proof is inspired by \cite{BKS2020}, but the random environment setting introduces a new layer of nonuniformity. In contrast with the homogeneous case, the contracting distance, the Lyapunov weight, and the effective coupling constants all depend on the current environment. A direct iteration of the usual weak Harris argument is therefore unavailable.

  To overcome this difficulty, we introduce a random distance evolving along the base flow and use a Lyapunov norm construction from random dynamical systems \cite{Arnold1998,Brandt1986} to absorb the nonuniform factors in time. The Birkhoff ergodic theorem then allows us to obtain contraction over blocks on a set of environment samples with positive measure and uniform control. The remaining blocks may be expanding, but their expansion is controllable relative to the contraction on the good blocks. Therefore, by introducing a block-gap counting procedure, we are able to upgrade the block contraction to a global contraction for almost all environment samples. The existence and exponential stability of the stationary measure then follow from a random fixed point argument. A further advantage of our approach is that it provides quantitative control of the exceptional sets. In particular, these sets can be chosen independently of the small noise parameter, a feature that is essential for deriving the Freidlin--Wentzell large deviation principle for the corresponding stationary measures.

  \subsection{The Freidlin--Wentzell large deviation principle}

  Once unique ergodicity is established, a natural next question is to study the small noise limit and to characterize the asymptotic behavior of the corresponding stationary measures. The study of small noise limits for random perturbations of deterministic systems goes back to a question of Kolmogorov, who proposed the idea that random perturbations may select stable dynamical objects. The closely related Freidlin--Wentzell large deviation principle for invariant measures originates from the classical Freidlin--Wentzell theory. It has been widely studied for autonomous stochastic systems and has also been extended to infinite-dimensional settings; see \cite{fre1,kha1,dem1,dup1,bud1,budhiraja2008large,sow1,cerr0,brz1,cerr2,mar1,mar2,salins2019uniform,sal2}.
  
  For PDEs with vanishing noise, however, most existing works cover the case in which the limiting dynamics is trivial, namely when the global attractor is a singleton. Martirosyan \cite{mar1} established the large deviation principle in the case where the attractor of the limiting dynamics consists of finitely many equilibria together with the connecting orbits between them, and applied this result to damped nonlinear wave equations. In \cite{mar2}, he obtained a large deviation upper bound for the two-dimensional Navier--Stokes equation without assuming trivial limiting dynamics. Much less is known about large deviation principles for stationary measures of genuinely nonautonomous SPDEs.

  Our second main result provides what appears to be the first general SPDE oriented LDP framework for stationary families of Markov cocycles over measurable random environments. The framework allows degenerate noise and, for the upper bound, nontrivial deterministic pullback attractors. Consider a randomly perturbed PDE in the abstract form
  \begin{align*}
    du(t)=F(\beta_t\sigma,u(t))dt
    +\sqrt\eps\,Q(\beta_t\sigma)dW_t, \qquad u(0)=u_0,
  \end{align*}
  which is assumed to be well-posed in a separable Hilbert space $H$. Here $F$ is a deterministic drift, $W_t$ is a cylindrical Wiener process and $Q(\beta_t\sigma)$ is the time-dependent noise coefficient, $\sqrt\eps$ is a noise intensity. We assume that the relevant deterministic system has a pullback attractor $\mathcal{A}(\sigma), \sigma\in\Sigma$. Define the pullback quasipotential at symbol $\sigma$  by
  \begin{equation*}
  \begin{split}
          E_{\Ac(\sigma)}(u_*)
          =\lim_{\eta\downarrow0}\inf\big\{& I^{\beta_{-s}\sigma}_{u_0,s}(u):
          s>0,\ u\in C([0,s];H),\\
          & \quad u_0\in\Ac(\beta_{-s}\sigma),\quad
          u(s)\in B_\eta(u_*)\big\},
  \end{split}
  \end{equation*}
  where $I^{\beta_{-s}\sigma}_{u_0,s}(u)$ is the energy of the controlled path $u$ on $[0,s]$ at symbol $\beta_{-s}\sigma$ and  $B_\eta(u_*)$ is an open ball in $H$. Then our main LDP result can be stated informally as follows. 
  \begin{THM}\label{t.060702}
    Assume that for $\varepsilon\in(0,1]$ the noisy system has a unique stationary measure $\mu^{\varepsilon}_{\s}$ and satisfies the trajectory Freidlin--Wentzell uniform LDP. Suppose further that for $m$-a.e. $\sigma\in\Sigma$, the quasipotential $E_{\Ac(\sigma)}$ has compact level sets, and that the tracking property, the nontrivial escaping energy property, and weak exponential tightness hold. Then for $m$-a.e. $\sigma\in\Sigma$, the upper large deviation bound
    \[
            \limsup_{\eps\to0}\eps\ln\mu^\eps_\sigma(F)
            \leq -\inf_{u\in F}E_{\Ac(\sigma)}(u)
    \]
    holds for every closed set $F\subset H$.  If the deterministic pullback attractor is a random point, then the matching lower bound
    \[
            \liminf_{\eps\to0}\eps\ln\mu^\eps_\sigma(G)
            \geq -\inf_{u\in G}E_{\Ac(\sigma)}(u)
    \]
    holds for every open set $G\subset H$.
  \end{THM}

  We refer the reader to Assumption \ref{a.050201} for the formal technical definitions of these properties. The proof is based on a combination of the Sowers--Martirosyan approach \cite{sow1,mar2} and the ergodicity of the base environment flow. A key difference in the proof of the upper bound is that, in our setting, the escaping energy may vary along the base flow and can even become very small on some blocks. When controlled paths remain in a bounded set but stay away from the pullback attractor, one must ensure that the cumulative energy required to avoid the attractor is still above the quasipotential level under consideration, so that the trajectory LDP can be invoked. This leads to a new nondegeneracy condition on the escaping energy, preventing it from degenerating too rapidly along typical base orbits. This condition is one of the key new ingredients in the Markov cocycle setting.

  Furthermore, we provide sufficient conditions guaranteeing the tracking property and the nontrivial escaping energy property, thereby reducing their verification to a priori estimates for solutions of the controlled system. Our proof does not rely on the contradiction argument used in \cite{mar2}; this makes it possible to quantify the dependence of the relevant constants and exceptional sets on the environment.

 \subsection{Applications to SPDEs} 
 The abstract Theorems~\ref{t.060701} and~\ref{t.060702} apply to a class of dissipative SPDEs. We illustrate the theory with two examples: the two-dimensional Navier--Stokes equation and the damped Sine--Gordon equation. In both cases the environment is a general ergodic invertible measure preserving flow, the deterministic force and the noise coefficient may depend on the environment, and the noise is allowed to be degenerate, provided that it acts on sufficiently many determining modes. The same finite dimensional noise used in the generalized coupling argument also yields the controlled tracking estimates needed for the Freidlin--Wentzell upper bound. 
 
 We first consider the two-dimensional Navier--Stokes equation on a bounded smooth domain $D\subset\mathbb R^2$. Let \[ H=\{u\in L^2(D;\mathbb R^2):\nabla\cdot u=0\text{ in }D,\ u\cdot \mathbf n=0\text{ on }\partial D\} \] and denote by $A$ the Stokes operator and by $B$ the usual bilinear term after the Leray projection. The equation has the abstract form \begin{align}\label{e.intro.ns} du+\big(\nu Au+B(u,u)\big)\,dt =F(\beta_t\sigma)\,dt+\sqrt{\varepsilon}\,Q(\beta_t\sigma)\,dW_t, \qquad u(0)=u_0\in H, \end{align} where $W$ is a finite-dimensional Brownian motion. The coefficient $Q(\sigma)$ is assumed to cover a fixed finite-dimensional determining subspace of $H$, with a measurable right-inverse whose norm is square integrable over the environment. The precise assumptions are stated in Section~\ref{s.060701}. 
 
 Our second example is the damped Sine--Gordon equation on a bounded smooth domain $D\subset \Rb^{ \mathbf{d}}$. Let \[ H:=L^2(D),\qquad V:=H_0^1(D),\qquad A=-\Delta \] with Dirichlet boundary condition. Writing $v=\partial_tu$ and $U=(u,v)$, the equation is written on $X=V\times H$ as \begin{align}\label{e.intro.sg} du&=v\,dt, \nonumber\\ dv+\big(\alpha v+Au+\varsigma\sin u\big)\,dt &=F(\beta_t\sigma)\,dt +\sqrt{\varepsilon}\,Q(\beta_t\sigma)\,dW_t , \end{align} where $\alpha>0$, $\varsigma\in\mathbb R$, and $W$ is again finite dimensional. As in the Navier--Stokes case, the noise is assumed to cover sufficiently many low modes, with a measurable square-integrable right-inverse. The precise assumptions are stated in Section~\ref{s.060702}. The application result can be summarized as follows.

 \begin{THM}\label{t.intro.C} 
 Assume that the coefficients of either \eqref{e.intro.ns} or \eqref{e.intro.sg} satisfy the mild integrability, boundedness, and determining mode controllability conditions stated in Section~\ref{s.060701} or Section~\ref{s.060702}. Then the associated Markov cocycle admits a unique stationary family $\{\mu_\sigma\}_{\sigma\in\Sigma}$, and this family attracts transition probabilities exponentially both in pullback and forward time. Moreover, for the small noise stationary family $\{\mu_\sigma^\varepsilon\}_{\varepsilon\in(0,1]}$, the Freidlin--Wentzell upper bound holds: for $m$-a.e. $\sigma\in\Sigma$ and every closed set $F$ in the corresponding phase space, 
 \[ \limsup_{\varepsilon\to0} \varepsilon\log\mu_\sigma^\varepsilon(F) \leq -\inf_{u\in F}E_{\mathcal A(\sigma)}(u), \] 
 where $E_{\mathcal A(\sigma)}$ is the pullback quasipotential associated with the deterministic limiting equation. If, in addition, the deterministic pullback attractor is a random point, then the matching lower bound holds: for every open set $G$ in the phase space, 
 \[ \liminf_{\varepsilon\to0} \varepsilon\log\mu_\sigma^\varepsilon(G) \geq -\inf_{u\in G}E_{\mathcal A(\sigma)}(u). \] 
 \end{THM} 
 
 The verification for the two equations follows the same general strategy but uses different compactness mechanisms. For the Navier--Stokes equation, parabolic smoothing and determining mode estimates provide the compactness of the quasipotential level sets and the controlled tracking property. For the damped Sine--Gordon equation, where such parabolic smoothing is unavailable, the compactness of the quasipotential level sets is obtained instead from compact-tail estimates for the controlled equation.

\section{Ergodicity and exponential mixing}\label{s.060901}
    In this section, we provide a criterion that can be used to prove ergodicity for time inhomogeneous SPDEs generating Markov processes in a random environment. 
    
    Let $(\Sigma,m)$ be a standard Borel probability space and $(\Sigma,\b_t,m)$ be an invertible ergodic measure preserving dynamical system, i.e., $\beta_0=\mathrm{id}$, $\beta_{t+s}=\beta_{t}\circ \beta_s$ for $s,t\in\mathbb{R}$ and $m$ is an ergodic invariant measure of the flow $\beta_t$, and $(t,\sigma)\to\beta_t\sigma$ is Borel measurable. 
    Let $(E, d)$ be a Polish space,  $P_{t}^{\s}(x,\cdot), x\in E, \s\in \Sigma$ are Feller Markov transition probabilities and $\mathbf P_{\s,x}$ the corresponding Markov family, where the transition operators satisfy the cocycle Chapman-Kolmogorov property
    \[P_{t+s}^\sigma=P_s^{ \sigma} P_t^{\beta_s \sigma}.\]
    For every $t,x$, we assume that $\s\to P_{t}^{\s}(x,\cdot)$ is measurable. Let $p: E\times E\to \mathbb R_+$ be a premetric on $E$, i.e., it is lower semicontinuous and $p(x,y)=0\Leftrightarrow x=y$. We assume that 
    \begin{align}\label{e.042604}
      \rho\leq p^{\varrho}
    \end{align} 
    for some $\varrho>0$. 
    \begin{rem}\label{r.042601}
      It is known that for any ergodic measure preserving flow $(\beta_t)_{t\in\mathbb R}$ on a standard Borel probability space, the discrete dynamical system generated by $\beta_T$ is ergodic except for countably many $T$. See for example \cite[Theorem 1]{PughShub1971}.
    \end{rem}
    \begin{thm}\label{t.042501}
      Assume that there is a lower semicontinuous function $U: E\to[0,\infty)$ and a measurable function $S:E\to[0,\infty]$ such that for any given $x,y\in E$ there exist progressively measurable processes $X_t^{\s,x}$ and $Y_t^{\s,x,y}$
      (denoted $X_t^{\s}$ and $Y_t^{\s}$ for simplicity)
      such that the following conditions hold:
      \begin{enumerate}[label=(H\arabic*)]
        \item There is a lower semicontinuous function $V: E\to \mathbb R_+$ and constant $\gamma>0$ such that
        \begin{align*}
          P_t^{\sigma}V(x) + \gamma \int_0^tP_s^{\sigma}V(x)ds\leq V(x) + \int_0^tK(\beta_s\s)ds
        \end{align*}
        where $K\in L^1(\Sigma,m)$ is nonnegative, and for any $M>0$, $U, p$ are bounded on $\{V\leq M\}$ and $\{V\leq M\}\times \{V\leq M\}$. 
        \item There exist $\zeta,\kappa>0$ such that for all $t\geq 0$,
        \begin{align*}
          p(X_t^{\s},Y_t^{\s})\leq p(x,y)\exp\left(-\zeta t + \kappa\int_0^tS(X_s^{\s})ds\right).
        \end{align*}
        \item There exist $\varpi>0$ and $b\in L^1(\Sigma,m)$ such that 
        \begin{align*}
          U(X_t^{\s}) + \varpi\int_0^tS(X_s^{\s})ds \leq U(x)+ \int_0^tb(\beta_s\s)ds + M_t^{\s},
        \end{align*}
        where $M_t^{\s}$ is a continuous local martingale with $M_0^\sigma=0$ and
        \[d\langle M^{\s}\rangle_t\leq b_1S(X_t^{\s})dt+b_2(\b_t\s)dt \quad \text{ for }t\geq 0,\]
        where  $b_1>0$ is a constant and  $b_2\in L^1(\Sigma,m)$ is nonnegative. Furthermore, 
        \begin{align*}
          \zeta>\frac{\kappa}{\varpi} \int_{\Sigma}b d m.
        \end{align*}
        \item The law of $\left(X_t^{\s}\right)_{t\in\mathbb R_+}$ coincides with $\mathbf P_{\s,x}$ and for any $\d\in(0,1]$, there exist a constant $C_{\d}>0$ and  a measurable function $\bar{\varepsilon}_{\d}:\mathbb{R}_+\to (0,1)$ which is non-increasing, such that for any $t\geq 0$, 
        \begin{align*}
          d_{TV}\left(\mathcal{L}(Y_t^{\s}), P_{t}^{\s}(y,\cdot)\right)\leq C_{\d}\left(M_{\d}(t,\s)\right)^{1/2}, \,\quad d_{TV}\left(\mathcal{L}(Y_t^{\s}), P_{t}^{\s}(y,\cdot)\right)\leq 1-\bar{\varepsilon}_{\d}(M_{\d}(t,\s)),
        \end{align*}
        where 
        \[ M_{\d}(t,\s): =\mathbf E\left(\int_0^tb_3(\beta_s\s)p(X_s^{\s},Y_s^{\s})ds\right)^{\d},\]
        and $b_3\in L^1(\Sigma,m)$ is nonnegative. 
      \end{enumerate}
      Then there is a unique $\beta_t$-invariant stationary measure 
      $(\G_\sigma)_{\s\in \Sigma}$
      of $P_t^{\s}$  such that $\G_\sigma P_t^\sigma=\G_{\beta_t \sigma}$ $m$-a.s. for all $t\geq 0$. Furthermore, there exist $\Sigma_0\subset\Sigma$ with $m(\Sigma_0)=1$, $T>0$ and $\l>0$ such that the following properties hold.
      \begin{itemize}
        \item Pullback exponential stability: for every $\sigma\in \Sigma_0$ and $x\in E$, there is $C_{\lambda}(\sigma,x)>0$ such that, for all $t\geq 0$,
        \begin{align*}
          \W_{\rho\wedge 1}\left(\delta_x P_t^{\beta_{-t} \sigma}, \G_\sigma\right) \leq C_\lambda(\sigma,x)e^{-\lambda t}.
        \end{align*}
        \item Forward exponential stability: for each fixed $r\in[0,T)$, there is a full measure set $\Sigma_r\subset\Sigma_0$ such that, for every $\sigma\in \Sigma_r$ and $x\in E$, there is $C_{\lambda,r}(\sigma,x)>0$ such that, for all $n\geq0$,
        \begin{align*}
          \W_{\rho\wedge 1}\left(\delta_x P_t^\sigma, \G_{\beta_t \sigma}\right) \leq C_{\lambda,r}(\sigma,x)e^{-\lambda t},\qquad t=nT+r.
        \end{align*}
      \end{itemize}
      Here $\W_{\rho\wedge 1}$ is the Wasserstein distance induced by the metric $\rho\wedge 1$. 
    \end{thm}

  \subsection{Local contraction}
  In this subsection, we will construct a random distance, to deduce a local contraction and a random smallness property from assumptions $(H1)$-$(H4)$. We first establish the following random local contraction estimate.  More precisely, we construct a (spatial) Lyapunov weighted premetric $\theta_\alpha$ for which the generalized coupling $(X_t^\sigma,Y_t^\sigma)$ contracts with an environment dependent rate, while the deviation of the second marginal from the true transition law is controlled by an explicit environment dependent defect.  On bounded Lyapunov sublevel sets this also gives the corresponding local smallness estimate.
    \begin{lem}\label{l.042501}
      Under assumptions $(H2)$--$(H4)$, there exist constants $\a_0, v, \g_{*}>0$ depending only on $\varpi,b_1,\zeta,\k$ such that for every $\a\in(0,\a_0]$ and the premetric 
      \begin{align}\label{e.042605}
        \theta_\alpha(x, y):=p(x, y)^\alpha e^{\alpha v U(x)}, \quad x,y\in E,
      \end{align}
      the following conclusions hold:
      \begin{enumerate}
        \item[(1)] There exists $C_{\a}\geq 1$, such that for any $x,y\in E$,
        \[ \mathbf{E} \theta_\alpha\left(X_t^\sigma, Y_t^\sigma\right)\leq A_t(\s)\theta_\alpha(x, y),\]
        and 
      \[d_{TV}\left(\mathcal{L}(Y_t^{\s}), P_{t}^{\s}(y,\cdot)\right)\leq B_t(\s)\theta_\alpha(x, y),\]
        where 
        \begin{align}\label{e.042502}
          A_t(\s) = C_{\a}e^{-\a\chi_{t}(\s)}, \quad B_{t}(\s) = C_{\a}\left(\int_0^{t}b_3(\beta_s\s)e^{-\chi_s(\s)}ds\right)^{\a}
        \end{align}
        with 
        \[\chi_t(\sigma) = \zeta t- \frac{\k}{\varpi-\g_{*} b_1}\int_0^t\left(b(\beta_s\s)+\g_{*} b_2(\beta_s\s)\right)ds.\]
        \item[(2)] For any $B\subset E$ such that $U$ and $p$ are bounded on $B$ and $B\times B$, and any $t\geq0$, there is  $C_{\a,B}>0$ such that for any $x,y\in B$, one has 
        \[d_{TV}\left(\mathcal{L}(Y_t^{\s}), P_{t}^{\s}(y,\cdot)\right)\leq 1-\bar{\varepsilon}_{\a}(C_{\a,B}B_{t}(\s)),\]
        and 
        \[\mathbf{E} \theta_\alpha\left(X_t^\sigma, Y_t^\sigma\right) \leq C_{\alpha,B} A_{t}(\s),\]
        where $\bar{\varepsilon}_{\a}$ is the function from $(H4)$. 
      \end{enumerate}
    \end{lem}
    \begin{proof}
      Fix $\g_{*}>0$ small such that 
      \begin{align}\label{e.042404}
        0<\g_{*}<\varpi / b_1 \text{ and }  \zeta-\frac{\kappa}{\varpi-\g_{*} b_1} \int_{\Sigma}\left(b+\g_{*} b_2\right) d m>0,
      \end{align}
      which is possible since $\zeta>\frac{\kappa}{\varpi} \int_{\Sigma}b d m$ by $(H3)$. Denote $\mathcal{M}_{\g_{*}}^{\sigma} =\sup_{t\geq 0}(M_t^{\s}-\g_{*}\langle M^{\s}\rangle_t)$. By $(H3)$ we have 
      \begin{align*}
        U(X_t^{\s}) + (\varpi-\g_{*} b_1)\int_0^tS(X_s^{\s})ds \leq U(X_0^{\s})+ \int_0^t\left(b(\beta_s\s)+\g_{*} b_2(\beta_s\s)\right)ds + \mathcal{M}_{\g_{*}}^{\sigma},
      \end{align*}
      which combined with $(H2)$, implies 
      \begin{align}\label{e.042401}
        p\left(X_t^\sigma, Y_t^\sigma\right) \leq p(x, y) \exp \Big(-\chi_t(\sigma)+v\left(U(x)-U\left(X_t^\sigma\right)+\mathcal{M}_{\g_{*}}^{\sigma}\right)\Big),
      \end{align}
      where 
      \[v = \frac{\k}{\varpi-\g_{*} b_1}.\]
      It then follows from \eqref{e.042401} and $(H4)$ that 
      \begin{align}\label{e.042402}
        \begin{split}
          d_{TV}\left(\mathcal{L}(Y_t^{\s}), P_{t}^{\s}(y,\cdot)\right)&\leq C_{\d}\left(\mathbf E\left(\int_0^tb_3(\beta_s\s)p(X_s^{\s},Y_s^{\s})ds\right)^{\d}\right)^{1/2}\\
          &\leq C_{\d}p(x,y)^{\frac{\d}{2}}e^{\frac{v\d U(x)}{2}}\left(\mathbf  Ee^{v\d\mathcal{M}_{\g_{*}}^{\s}}\right)^{1/2}\left(\int_0^{t}b_3(\beta_s\s)e^{-\chi_s(\s)}ds\right)^{\d/2}.
        \end{split}
      \end{align}
      Let $\a_0 = \left(\frac{\g_{*}}{2v}\right)\wedge\frac12$. For $\alpha \in (0,\alpha_0]$, consider the premetric 
      \[\theta_\alpha(x, y):=p(x, y)^\alpha e^{\alpha v U(x)}, \quad x,y\in E.\]
      Then by taking $\d=2\a$ in \eqref{e.042402} and using the exponential martingale inequality, one has $\mathbf Ee^{\frac{v\d\mathcal{M}_{\g_{*}}^{\s}}{2}}\leq 2$ and therefore
      \begin{align*}
        d_{TV}\left(\mathcal{L}(Y_t^{\s}), P_{t}^{\s}(y,\cdot)\right)\leq B_{t}(\s)\theta_\alpha(x, y),\quad B_{t}(\s) = C_{\a}\left(\int_0^{t}b_3(\beta_s\s)e^{-\chi_s(\s)}ds\right)^{\a}.
      \end{align*}
      Similarly, it follows from \eqref{e.042401} that 
      \begin{align}\label{e.042403}
        \mathbf{E} \theta_\alpha\left(X_t^\sigma, Y_t^\sigma\right)= \mathbf{E} p\left(X_t^\sigma, Y_t^\sigma\right)^{\a}e^{\a v U(X_t^\sigma)}\leq A_t(\s) \theta_\alpha(x, y), \quad A_t(\s) = C_{\a}e^{-\a\chi_{t}(\s)}
      \end{align}
      Now suppose $B\subset E$ is such that $U$ and $p$ are bounded on $B$ and $B\times B$. Then for $x,y\in B$,  by \eqref{e.042401} and the exponential martingale inequality, one has 
      \begin{align*}
        \mathbf E\left(\int_0^tb_3(\beta_s\s)p(X_s^{\s},Y_s^{\s})ds\right)^{\a} \leq C_{\a,B} B_t(\s)
      \end{align*}
      which by $(H4)$ implies that for $\bar{\varepsilon}_{\a}(C_{\a,B}B_{t}(\s))$ one has 
      \[d_{TV}\left(\mathcal{L}(Y_t^{\s}), P_{t}^{\s}(y,\cdot)\right)\leq 1-\bar{\varepsilon}_{\a}(C_{\a,B}B_{t}(\s)).\]
      Finally, by \eqref{e.042403}, for $x,y\in B$, one has 
      \begin{align*}
        \mathbf{E} \theta_\alpha\left(X_t^\sigma, Y_t^\sigma\right) \leq C_{\alpha,B} A_{t}(\s).
      \end{align*}
    \end{proof}
    As the contraction factor and the total variation defect obtained in Lemma \ref{l.042501} are nonuniform in the environment, they cannot be used directly to obtain a deterministic contraction for the Markov cocycle.  We therefore introduce a temporal Lyapunov weight along the base dynamics, which absorbs the random one-step expansion and defect terms.  The following elementary lemma, essentially a random affine recursion estimate, will be used for this purpose.  For a proof we refer the reader to \cite[Theorem 5.6.5]{Arnold1998} or \cite[Theorem 1]{Brandt1986}.

    \begin{lem}\label{l.042401}
      Let $(\Sigma,\theta,m)$ be an ergodic measure preserving dynamical system. Suppose that $A,F>0$ are measurable, $\log A\in L^1, \log^+F\in L^1$, and 
      \[\int_\Sigma \log A(\s)m(d\s)<\log q\]
      for some $q\in(0,1)$. Then
      \[N(\sigma):=\sum_{n=0}^{\infty}q^{-(n+1)}
      F(\theta^n\sigma)\prod_{j=0}^{n-1}A(\theta^j\sigma)\]
      is finite for \(m\)-a.e. $\sigma$ and satisfies
      \[qN(\sigma)=F(\sigma)+A(\sigma)N(\theta\sigma).\]
    \end{lem}
    
    We now construct an environment dependent distance-like function with respect to which the generalized coupling is strictly contracting with a deterministic constant.  The key point is that, by the ergodic theorem, the exponent $\chi_t(\sigma)$ in \eqref{e.042502} has a positive deterministic average. Thus possible local expansions along exceptional environmental samples can be compensated by the temporal Lyapunov weight.  This yields a fixed contraction factor on the set where the distance is smaller than one.  At this stage the smallness estimate is still environment dependent; it will be converted into a usable smallness condition on good environmental samples in the next subsection.

    \begin{lem}\label{l.042502}
      Under the assumptions of Theorem \ref{t.042501}, there are constants $q\in (0,1)$ and $T>0$ depending only on   $\varpi,b_1,\zeta,\k, b, b_2$  such that $\beta_T$ is ergodic, and the random Lyapunov weight 
      \[N(\sigma):=\sum_{n=0}^{\infty} q^{-(n+1)} (1+B_T\left(\beta_{n T} \sigma\right)) \prod_{j=0}^{n-1} A_T\left(\beta_{j T} \sigma\right)\]
      is finite $m$-a.s., and for the distance-like function 
      \[    d_{\s}(x,y) := 1 \wedge N(\sigma) \theta_\alpha(x, y) \wedge N(\sigma) \theta_\alpha(y, x), \quad x,y\in E,\]
      we have 
      \begin{align*}
        W_{d_{\beta_{T}\sigma}}\left(\delta_x P_T^\sigma, \delta_y P_T^\sigma\right)\leq q d_{\s}(x,y) \text{ for } d_{\s}(x,y)<1
      \end{align*}
      and 
      \begin{align*}
        W_{d_{\beta_{T}\sigma}}\left(\delta_x P_T^\sigma, \delta_y P_T^\sigma\right)\leq  N\left(\beta_T \sigma\right)C_{\alpha,B} A_{T}(\s) + 1-\bar{\varepsilon}_{\a}(C_{\a,B}B_{T}(\s)),
      \end{align*}
      where $C_{\alpha,B},\bar{\varepsilon}_{\a}$ are from Lemma \ref{l.042501}. 
    \end{lem}
    \begin{proof}
      Note that 
      \begin{align*}
        \int_{\Sigma} \log A_t(\sigma) m(d\sigma)&=\log C_\alpha-\alpha \int_{\Sigma} \chi_t(\sigma) d m(\sigma)\\
        & = \log C_\alpha- \a t \left( \zeta-\frac{\kappa}{\varpi-\g_{*} b_1} \int_{\Sigma}\left(b(\s)+\g_{*} b_2(\s)\right) m(d\s)\right).
      \end{align*}
      In view of the choice of $\g_{*}$ in \eqref{e.042404} and Remark \ref{r.042601}, we can find $T>0$ such that $\beta_T$ is ergodic and 
      \begin{align*}
        \int_{\Sigma} \log A_T(\sigma) m(d\sigma)<0. 
      \end{align*}
      Pick any $q\in (0,1)$ such that $\log q>\int_{\Sigma} \log A_T(\s)m(d\s)$ and let 
      \[N(\sigma)=\sum_{n=0}^{\infty} q^{-(n+1)} F_T\left(\beta_{n T} \sigma\right) \prod_{j=0}^{n-1} A_T\left(\beta_{j T} \sigma\right),\]
      where $F_T(\sigma)=1+B_T(\sigma)$. Then $\log ^{+} F_T \in L^1(\Sigma, m)$. Hence, by Lemma \ref{l.042401} (with $\theta=\beta_T$), we know that $N(\sigma)<\infty$ $m$-a.s. and solves the equation 
      \[q N(\sigma)=F_T(\sigma)+A_T(\sigma) N\left(\beta_T \sigma\right).\]
      In particular,
      \begin{align}\label{e.042501}
        A_T(\sigma) N\left(\beta_T \sigma\right)+B_T(\sigma) < q N(\sigma).
      \end{align}
      Now consider the distance-like function 
      \begin{align*} 
        d_{\s}(x,y) := 1 \wedge N(\sigma) \theta_\alpha(x, y) \wedge N(\sigma) \theta_\alpha(y, x), \quad x,y\in E. 
      \end{align*}
      Then by  standard coupling arguments \cite{villani2009optimal}, there is a random variable $Z^{\s}$ such that $\mathcal{L} (Z^{\s}) =  P_T^\sigma(y,\cdot)$ and 
      \[\mathbf P(Y_T^{\s}\neq Z^{\s}) =  d_{TV}\left(\mathcal{L}(Y_T^{\s}), P_{T}^{\s}(y,\cdot)\right)\leq B_{T}(\s)\theta_\alpha(x, y).\] 
      It then follows from Lemma \ref{l.042501} and \eqref{e.042501} that, 
      \begin{align*}
        \W_{d_{\beta_{T}\sigma}}\left(\delta_x P_T^\sigma, \delta_y P_T^\sigma\right)&\leq \mathbf Ed_{\beta_{T}\sigma}(X_T^{\s},Z^{\s})\\
        &\leq \mathbf Ed_{\beta_{T}\sigma}(X_T^{\s},Z^{\s})\1_{\{Y_{T}^{\s}=Z^{\sigma}\}}+\mathbf Ed_{\beta_{T}\sigma}(X_T^{\s},Z^{\s})\1_{\{Y_{T}^{\s}\neq Z^{\sigma}\}}\\
        &\leq \mathbf Ed_{\beta_{T}\sigma}(X_T^{\s},Y_T^{\s}) + \mathbf P(Y_T^{\s}\neq Z^{\s})\\
        &\leq N\left(\beta_T \sigma\right)\mathbf E\theta_{\a}(X_T^{\s},Y_T^{\s})+ B_{T}(\s)\theta_\alpha(x, y)\\
        &\leq \left(A_T(\sigma) N\left(\beta_T \sigma\right)+B_T(\sigma)\right) \theta_\alpha(x, y)\\
        &\leq qN(\s)\theta_{\a}(x,y). 
      \end{align*}
      By symmetry, we see that 
      \[\W_{d_{\beta_{T}\sigma}}\left(\delta_x P_T^\sigma, \delta_y P_T^\sigma\right)\leq qN(\s)\theta_{\a}(x,y)\wedge q N(\s)\theta_{\a}(y,x)=q d_{\s}(x,y)\]
      for $x,y\in E$ with $d_{\s}(x,y)<1$. In particular, as $d_{\s}\leq 1$, we have the non-expansiveness, 
      \begin{align*}
        \W_{d_{\beta_{T}\sigma}}\left(\delta_x P_T^\sigma, \delta_y P_T^\sigma\right)\leq d_{\s}(x,y) \text{ for all } x,y\in E,
      \end{align*}
      Picking any pair $\mu,\nu\in\mathcal{P}(E)$, integrating and taking infimum with respect to all couplings for $\mu,\nu\in\mathcal{P}(E)$, we have 
      \begin{align}\label{e.052602}
      \begin{split}
        \W_{d_{\beta_{T}\sigma}}\left(\mu P_T^\sigma, \nu P_T^\sigma\right)\leq 
        \W_{d_{\sigma}}(\mu,\nu). 
      \end{split}
      \end{align} 
      Similarly, by Lemma \ref{l.042501} and the definition of $d_{\s}$, for $x,y\in B$ where the set $B\subset E$ is from Lemma \ref{l.042501}, one has 
      \begin{align*}
        \W_{d_{\beta_{T}\sigma}}\left(\delta_x P_T^\sigma, \delta_y P_T^\sigma\right)&\leq \mathbf Ed_{\beta_{T}\sigma}(X_T^{\s},Y_T^{\s}) + \mathbf P(Y_T^{\s}\neq Z^{\s})\\
        &\leq N\left(\beta_T \sigma\right)C_{\alpha,B} A_{T}(\s) + 1-\bar{\varepsilon}_{\a}(C_{\a,B}B_{T}(\s)).
      \end{align*}
    \end{proof}
    
  \subsection{Global contraction on a typical set}
    In this subsection we upgrade the local contraction obtained in the previous subsection to a global contraction on a typical set of environmental samples. The difficulty is that both the Lyapunov drift and the smallness estimate are nonuniform in the environment.  We first renormalize the Lyapunov function by a temporal weight along the base dynamics; this produces an adapted Lyapunov function for the skeleton cocycle.  We then combine this renormalized Lyapunov drift with the local contraction of $d_\sigma$ to construct a weighted distance $D_\sigma$ which contracts globally whenever the environment belongs to a suitable positive-measure set.  From now on, $\mathcal P_V(E)$ denotes the set of probability measures on $E$ with finite $V$-moment.

    Fix the skeleton time $T>0$ from Lemma \ref{l.042502} and denote $Q_{\s}:= P_{T}^{\s},\, \Theta := \beta_T$ and for $n\geq 1$, 
    \[Q_\sigma^{(n)}:=Q_\sigma Q_{\Theta \sigma} \cdots Q_{\Theta^{n-1} \sigma}=P_{n T}^\sigma.\]
    By $(H1)$ and Gronwall's inequality, we have 
    \[Q_\sigma V(x) \leq a V(x)+K_T(\sigma), \quad 
    \text{ where } a:=e^{-\gamma T}<1 \; \text{and} \; K_T(\sigma):=\int_0^T e^{-\gamma(T-s)} K\left(\beta_s \sigma\right) d s.\]
    Fix $\eta\in (a,1)$. Since $K\in L^{1}(\Sigma,m)$, Lemma \ref{l.042401} (with $\theta=\Theta^{-1}$) implies that 
    \begin{align}\label{e.042603}
      R(\sigma):=\sum_{j=1}^{\infty} \eta^{-j} a^{j-1} \left(1+K_T\left(\Theta^{-j} \sigma\right)\right)
    \end{align}
    is finite $m$-a.s., and solves the equation 
    \[a R(\sigma)+1+K_T(\sigma)=\eta R(\Theta \sigma).\]
    Therefore, by defining $L_\sigma(x):=\frac{V(x)}{R(\sigma)}$, it follows that 
    \begin{align}\label{e.042607}
      Q_\sigma L_{\Theta \sigma}(x) \leq \eta L_\sigma(x)+\eta, \quad x \in E.
    \end{align}
    By iteration, one has for $n\geq 1$, 
    \begin{align}\label{e.042504}
      Q_\sigma^{(n)} L_{\Theta^n \sigma} \leq \eta^n L_\sigma+\frac{\eta}{1-\eta}
    \end{align}

    The next lemma gives a uniform $L$-moment bound for invariant families. Although the original Lyapunov estimate is nonuniform in the environment, the renormalized function $L_\sigma=V/R(\sigma)$ satisfies the adapted drift inequality \eqref{e.042607}.  This makes it possible to obtain an environment-independent bound on the $L_\sigma$-moment of any invariant family with finite $V$-moment.

    \begin{lem}\label{l.042603}
      If $\G_{\s}$ is an invariant family of $Q_{\s}$ and $\{\G_{\s}\}\subset \mathcal{P}_V(E)$, then 
      \[\int_E L_\sigma(x)  \Gamma_\sigma(dx)\leq \frac{\eta}{1-\eta}\]
      almost surely. 
    \end{lem}
    \begin{proof}
      Note that $R(\s)\geq 1/\eta$. So 
      \[A_L(\s):=\int_E L_\sigma(x)  \Gamma_\sigma(dx)<\infty.\]
      By invariance and \eqref{e.042607}, one has 
      \[A_L(\Theta \sigma) \leq \eta A_L(\s)+\eta.\]
      Iterating backwards, we obtain 
      \[A_L(\s)\leq \eta^n A_L(\Theta^{-n} \sigma)+\frac{\eta}{1-\eta}.\]
      Since $A_L(\s)<\infty$ a.s., there is  $A_0<\infty$ such that the set $\left\{A_L(\s) \leq A_0\right\}$ has positive measure. As $\Theta$ is ergodic, the orbit $\{\Theta^{-n}\s\}$ will visit the set infinitely often, say at $n_k\to\infty$. Then 
      \begin{align*}
        A_L(\s)\leq \eta^{n_k} A_0+\frac{\eta}{1-\eta} \rightarrow \frac{\eta}{1-\eta} \quad \text{as} \quad n_k\to\infty,
      \end{align*}
      giving the desired estimate. 
    \end{proof}

    We now combine the adapted Lyapunov drift with the local contraction of $d_\sigma$.  The weighted distance below is designed to handle three different regimes simultaneously: small $d_\sigma$-distance, where the local contraction applies; large Lyapunov level, where the drift of $L_\sigma$ yields contraction; and bounded Lyapunov level with large $d_\sigma$-distance, where the smallness estimate from the generalized coupling is used.  This gives a genuine global contraction on a positive-measure set of environmental samples.  For $\varkappa>0$, define the weighted distance 
    \begin{align*}
      D_\sigma(x, y):=\left[d_\sigma(x, y)\left(1+\varkappa L_\sigma(x)+\varkappa L_\sigma(y)\right)\right]^{1 / 2}
    \end{align*}
  
    \begin{lem}\label{l.042601}
      Under the assumptions of Theorem \ref{t.042501}, there are measurable $G\subset \Sigma$ with $m(G)>0$, and $\ell\in\mathbb N, 0<r<1$ such that for every $\s\in G$, we have 
      \[\W_{D_{\Theta^{\ell}\sigma}}\left(\mu Q_\sigma^{(\ell)}, \nu Q_\sigma^{(\ell)}\right) \leq r \W_{D_\sigma}(\mu, \nu),\quad \mu, \nu\in \mathcal{P}_V(E) .\]
    \end{lem}
  \begin{proof}
    Let $x,y\in E$ and $\mathcal{C}$ be the set of couplings of $\d_xQ_{\s}^{(\ell)},\d_yQ_{\s}^{(\ell)}$. Then by the Cauchy-Schwarz inequality and \eqref{e.042504}, one has 
    \begin{align}\label{e.042601}
      \begin{split}
        \W_{D_{\Theta^{\ell}\sigma}}&\left(\d_x Q_\sigma^{(\ell)}, \d_yQ_{\s}^{(\ell)}\right) = \inf_{Q\in \mathcal{C}}\int_{E\times E}D_{\Theta^{\ell}\sigma}(u,v)Q(dudv)\\
        &\leq \inf_{Q\in \mathcal{C}}\left(\int_{E\times E}d_{\Theta^{\ell}\sigma}(u,v)Q(dudv)\right)^{1/2}\times\left(\int_{E\times E}(1+\varkappa L_{\Theta^{\ell}\sigma}(u)+\varkappa L_{\Theta^{\ell}\sigma}(v))Q(dudv)\right)^{1/2}\\
        &\leq \W_{d_{\Theta^{\ell}\sigma}}\left(\d_x Q_\sigma^{(\ell)}, \d_yQ_{\s}^{(\ell)}\right)^{1/2}\left(1+\varkappa Q_{\s}^{(\ell)}L_{\Theta^{\ell}\sigma}(x)+\varkappa Q_{\s}^{(\ell)}L_{\Theta^{\ell}\sigma}(y)\right)^{1/2}\\
        &\leq \W_{d_{\Theta^{\ell}\sigma}}\left(\d_x Q_\sigma^{(\ell)}, \d_yQ_{\s}^{(\ell)}\right)^{1/2}\left(1+\varkappa\eta^{\ell}(L_{\s}(x)+L_{\s}(y))+\frac{2\varkappa\eta}{1-\eta}\right)^{1/2}.
      \end{split}
    \end{align}
    Denote $\L:= \frac{2\eta}{1-\eta}$. We split the contraction into three cases.

    {\it Case 1: Small scales.} If $d_{\s}(x,y)<1$, then by Lemma \ref{l.042502}, 
    \begin{align*}
      \W_{D_{\Theta^{\ell}\sigma}}\left(\d_x Q_\sigma^{(\ell)}, \d_yQ_{\s}^{(\ell)}\right)&\leq \left(q d_{\s}(x,y)\left(1+\varkappa\eta^{\ell}(L_{\s}(x)+L_{\s}(y))+\varkappa\L\right)\right)^{1/2}\\
      &\leq r_1 D_{\s}(x,y),
    \end{align*}
    where $r_1=q^{1/2}\left(1+\varkappa\L\right)^{1/2}$.

     {\it Case 2: Large scales.} Since $d_{\s}\leq 1$, it follows from \eqref{e.052602} that 
    \begin{align*}
      \W_{d_{\Theta^{\ell}\sigma}}&\left(\d_x Q_\sigma^{(\ell)}, \d_yQ_{\s}^{(\ell)}\right)\leq d_{\s}(x,y) \text{ for all } x,y\in E. 
    \end{align*}
    Therefore, if $d_{\s}(x,y)=1$ and $L_{\s}(x)+L_{\s}(y)\geq M$, one has 
    \begin{align*}
      \W_{D_{\Theta^{\ell}\sigma}}\left(\d_x Q_\sigma^{(\ell)}, \d_yQ_{\s}^{(\ell)}\right)\leq r_2 D_{\s}(x,y),
    \end{align*}
    where 
    \[r_2=\sup_{l\geq M}\left(\frac{1+\varkappa\eta l+\varkappa\L}{1+\varkappa l}\right)^{1/2}.\]

     {\it Case 3: Intermediate scales.} If $d_{\s}(x,y)=1$ and $L_{\s}(x)+L_{\s}(y)\leq M$, then  Lemma \ref{l.042502} yields 
    \begin{align*}
      \W_{D_{\Theta^{\ell}\sigma}}\left(\d_x Q_\sigma^{(\ell)}, \d_yQ_{\s}^{(\ell)}\right)\leq r_3D_{\s}(x,y),
    \end{align*}
    where 
    \[r_3=\Big(N\left(\beta_{\ell T} \sigma\right)C_{\alpha,B} A_{\ell T}(\s) + 1-\bar{\varepsilon}_{\a}(C_{\a,MR}B_{\ell T}(\s))\Big)^{1/2}\left(1+\varkappa\eta^{\ell} M+\varkappa\L\right)^{1/2},\]
    and $C_{\a,MR}$ denotes $C_{\alpha,B}$ from Lemma \ref{l.042502} with (random) $B=\{x\in X: V(x)\leq MR\}$.
    
    Next we arrange the parameters appropriately so that $r_i<1$ for $i=1,2,3$, to obtain global contraction on a good environmental set. We first choose $M>2\L/(1-\eta)$ so that $r_2<1$ for 
    \[ d_{\s}(x,y)=1\,  \text{ and } \, L_{\s}(x)+L_{\s}(y)\geq M, \quad x,y\in E.\]
    Next, since $b_3\in L^1(\Sigma,m)$, by \eqref{e.042502}, \eqref{e.042404} and the Birkhoff ergodic theorem, we know that $B_{t}(\s)\leq B_{\infty}(\s)<\infty$ $m$-a.s.. Note also that $N(\sigma), R(\sigma)<\infty$ $m$-a.s. by their definitions. Therefore, by taking a large $J>0$, for 
    \[G:=\left\{N(\sigma) \leq J, \quad R(\sigma) \leq J, \quad B_{\infty}(\sigma) \leq J\right\},\]
    we have $m(G)>0$. Then for $\s\in G$, $L_{\s}(x)+L_{\s}(y)\leq M$ implies 
    \[V(x)+V(y) \leq M R(\sigma) \leq M J.\]
    Taking the set $B$ in Lemma \ref{l.042502} to be $\{V \leq M J\}$ for which the boundedness of $U, p$ is guaranteed by $(H1)$, we have $\varepsilon_0:=\bar{\varepsilon}_\alpha\left(C_{\alpha, MJ} J\right)>0$ 
    is deterministic and for $\s\in G$, 
    \begin{align}\label{e.042505}
      \bar{\varepsilon}_\alpha\left(C_{\alpha, MJ} B_{\ell T}(\sigma)\right) \geq \bar{\varepsilon}_\alpha\left(C_{\alpha, MJ} B_{\infty}(\sigma)\right) \geq \varepsilon_0.
    \end{align}
    We now take $\varkappa>0$ small so that $r_1<1$ for $d_{\s}(x,y)<1$ and 
    \[\left(1-\frac{3}{4} \varepsilon_0\right)(1+\varkappa(1+\Lambda))<1,\]
    where the latter will be used to derive $r_3<1$. Now take $\ell$ sufficiently large such that 
    \[C_{\alpha, MJ} q^{\ell} J \leq \frac{\varepsilon_0}{4} \quad \text{and} \quad \eta^{\ell} M\leq 1.\]
    By iterating from \eqref{e.042501}, one has 
    \[N\left(\Theta^{\ell} \sigma\right) \prod_{j=0}^{\ell-1} A_T\left(\Theta^j \sigma\right) \leq q^{\ell} N(\sigma),\]
    and it follows from \eqref{e.042502} that $\prod_{j=0}^{\ell-1} A_T\left(\Theta^j \sigma\right)=C_{\a}^{\ell-1}A_{\ell T}(\s)$ where $C_{\a}>1$. Therefore, for $\s\in G$, we have 
    \[N\left(\Theta^{\ell} \sigma\right) C_{\alpha, MJ} A_{\ell T}(\sigma) \leq C_{\alpha, MJ} q^{\ell} N(\sigma) \leq C_{\alpha, MJ} q^{\ell} J \leq \frac{\varepsilon_0}{4}.\]
    Consequently, in view of \eqref{e.042505}, one has 
    \begin{align*}
      r_3^2\leq \left(1-\frac{3}{4} \varepsilon_0\right)(1+\varkappa(1+\Lambda))<1.
    \end{align*}
    Then by taking $r=r_1 \vee r_2 \vee r_3$, we have 
    \begin{align*}
      W_{D_{\Theta^{\ell}\sigma}}\left(\d_x Q_\sigma^{(\ell)}, \d_yQ_{\s}^{(\ell)}\right)\leq r D_{\s}(x,y)
    \end{align*}
    for all $x,y\in E$. Integrating and taking infimum over all couplings of $\mu,\nu$ completes the proof. 
  
  \end{proof}

    The preceding proof also gives a useful uniform control outside the good environmental set.  Although a general block need not be contractive in the weighted distance $D_\sigma$, it can expand distances by at most a fixed deterministic factor. This bounded-expansion estimate will be used together with the recurrence of the good set $G$ to iterate the contraction along typical environmental trajectories.
  \begin{cor}\label{c.042601}
    For every $k\in\mathbb N$, $\mu,\nu\in \mathcal{P}_{V}$, one has 
    \begin{align*}
      \W_{D_{\Theta^k\sigma}}&\left(\mu Q_\sigma^{(k)}, \nu Q_{\s}^{(k)}\right)\leq h \W_{D_{\s}}(\mu,\nu), \quad m\text{-a.s.}, 
    \end{align*}
    where $h=(1+\varkappa\L)^{1/2}$. 
  \end{cor}
  \begin{proof}
    Observe that for $m$-a.e. $\s\in \Sigma$,
    \begin{align*}
      \W_{d_{\Theta\sigma}}&\left(\mu Q_\sigma, \nu Q_{\s}\right)\leq \W_{d_{\s}}(\mu,\nu) \text{ for all } \mu,\nu\in \mathcal{P}(E),
    \end{align*}
   By iteration, one has for all $k\geq 1$, 
    \begin{align*}
      \W_{d_{\Theta^k\sigma}}&\left(\mu Q_\sigma^{(k)}, \nu Q_{\s}^{(k)}\right)\leq \W_{d_{\s}}(\mu,\nu) \text{ for all } \mu,\nu\in \mathcal{P}(E). 
    \end{align*}
    It follows from \eqref{e.042601} that 
    \begin{align*}
      \W_{D_{\Theta^{k}\sigma}}\left(\d_x Q_\sigma^{(k)}, \d_yQ_{\s}^{(k)}\right)& \leq \W_{d_{\Theta^{k}\sigma}}\left(\d_x Q_\sigma^{(k)}, \d_yQ_{\s}^{(k)}\right)^{1/2}\left(1+\varkappa(L_{\s}(x)+L_{\s}(y))+\varkappa\L\right)^{1/2}\\
      &\leq D_{\s}(x,y)\left(\frac{1+\varkappa(L_{\s}(x)+L_{\s}(y))+\varkappa\L}{1+\varkappa(L_{\s}(x)+L_{\s}(y))}\right)^{1/2}\\
      &\leq (1+\varkappa\L)^{1/2} D_{\s}(x,y).
    \end{align*}
    Taking $h=(1+\varkappa\L)^{1/2}$, integrating and taking infimum over all couplings of $\mu,\nu\in \mathcal{P}_V(E)$ gives the desired estimate. 
  \end{proof}
  \begin{rem}
    For later use, we will fix $\varkappa>0$ small such that $hr<1$. 
  \end{rem}
  
  \subsection{Pullback and forward stability}\label{s.051701}
    In this subsection, we prove the stability of the skeleton cocycle and then construct the stationary family by a pullback fixed point argument.  The contraction obtained in the previous subsection holds only on the good set $G$ and only for blocks of length $\ell$.  Outside $G$, the dynamics may
    fail to be contractive, but Corollary \ref{c.042601} gives a uniform bounded-expansion estimate.  Since the base transformation $\Theta$ is ergodic and $m(G)>0$, typical environmental trajectories visit $G$ with positive frequency.  Combining these visits with the bounded expansion between them yields exponential contraction both in pullback and forward time.

    The next lemma makes this idea quantitative. It shows that the accumulated effect of the contractive visits to $G$ dominates the possible expansions between such visits. As a result, the skeleton cocycle is exponentially contracting in the random distance $D_\sigma$, both for pullback trajectories and for forward trajectories. The prefactor is allowed to depend on the environment, while the exponential rate is deterministic.

    \begin{lem}\label{l.042602}
      For any 
      \[0<\lambda<-\frac{m(G)\log(hr)}{\ell+1},\]
      there is a measurable $C_{\l}(\s)<\infty$ $m$-a.s. such that we have 
      \begin{itemize}
        \item Pullback contraction:
        \begin{align*}
          \W_{D_{\s}}\left(\mu Q_{\Theta^{-n}\s}^{(n)}, \nu Q_{\Theta^{-n}\s}^{(n)}\right)\leq C_{\l}(\s) e^{-\l n} \W_{D_{\Theta^{-n}\s}}(\mu,\nu), \quad n\geq 0;
        \end{align*}
        \item Forward contraction: 
        \begin{align*}
          \W_{D_{\Theta^{n}\s}}\left(\mu Q_{\s}^{(n)}, \nu Q_{\s}^{(n)}\right)\leq C_{\l}(\s) e^{-\l n} \W_{D_{\s}}(\mu,\nu), \quad n\geq 0,
        \end{align*}
      \end{itemize}
      for all $ \mu,\nu\in\mathcal{P}_V(E)$.
    \end{lem}
    Before the proof, we introduce a useful notion for contractive blocks.
    Since the contraction is available only for complete blocks of length $\ell$ starting from points in $G$, overlapping good visits cannot be counted independently.  We therefore count the maximal number of disjoint good blocks that can be placed inside a given integer interval. For $\s\in\Sigma$, and an interval of integers $[a, b)$, a family of starting times $s_1<\cdots<s_k$ is called admissible if 
    \begin{align*}
        \Theta^{s_i}\sigma\in G,\qquad a\leq s_i,\qquad s_i+\ell\leq b, \, \text{ for }1\leq i\leq k,
    \end{align*}
    and $s_{i+1}\geq s_i+\ell$ for $1\leq i\le k-1$. Let $K_{a,b}(\s)$ be the maximal cardinality of such an admissible family. 
    \begin{proof}[Proof of Lemma \ref{l.042602}]
      We provide a proof here for the pullback contraction. The forward contraction is similar. We first claim that 
      \begin{align}\label{e.042602}
        \liminf _{n \rightarrow \infty} \frac{K_{-n, 0}(\sigma)}{n} \geq \frac{m(G)}{\ell+1}, \quad \text{ for } m\text{-a.e. }\s\in \Sigma.
      \end{align}
      Indeed, since $\Theta=\beta_T$ is invertible and ergodic, by Birkhoff's ergodic theorem, we have 
      \begin{align*}
        \lim_{n\to\infty}\frac1n\sum_{j=-n}^{-1}\1_{G}(\Theta^j\s)=m(G)
      \end{align*}
      for $m$-a.e. $\s\in\Sigma$. Since each selected starting removes at most $\ell+1$ possible good starts in an admissible family, one has 
      \begin{align*}
        \sum_{j=-n}^{-1} \1_G\left(\Theta^j \sigma\right)\leq (\ell+1)K_{-n, 0}(\sigma)+\ell+1.
      \end{align*}
      Dividing by $n$ and passing to the limit gives the claim.

      Now given an admissible family with maximal cardinality $K_{-n, 0}(\sigma)$,  we divide the integer interval $[-n,0]$ into $K_{-n, 0}(\sigma)$ good blocks of contraction by $r$ and at most $K_{-n, 0}(\sigma)+1$ gaps corresponding to a factor of $h$. It then follows from Lemma \ref{l.042601} and Corollary \ref{c.042601} that 
      \begin{align*}
        \W_{D_{\s}}\left(\mu Q_{\Theta^{-n}\s}^{(n)}, \nu Q_{\Theta^{-n}\s}^{(n)}\right)\leq h(hr)^{K_{-n, 0}(\sigma)} \W_{D_{\Theta^{-n}\s}}(\mu,\nu).
      \end{align*}
      For 
      \[0<\lambda<-\frac{m(G)\log(hr)}{\ell+1},\]
      the claim \eqref{e.042602} implies that 
      \[C_{\l}(\s):=\sup _{n \geq 0} h (hr)^{K_{-n, 0}(\sigma)} e^{\lambda n}<\infty, \quad m\text{-a.s.},\]
      which gives the desired pullback contraction. 
    \end{proof}

    We now use the pullback contraction to construct the stationary family. The idea is to start the skeleton dynamics at time $-nT$, push the point mass forward to time $0$, and let $n\to\infty$. The exponential pullback contraction makes this sequence Cauchy in the Wasserstein distance induced by $D_\sigma$, while the adapted Lyapunov estimate gives the required moment control. The resulting limit is the unique invariant family for the skeleton. Finally, we pass from the skeleton to continuous time by a perfection argument, obtaining an invariant modification on a flow-invariant full-measure set.
    
    Let $\Sigma_*$ be the full-measure set on which all the estimates and conclusions above hold for every $\sigma\in \Sigma_*$. Since $\Theta=\beta_T$ is invertible and measure preserving, the set
    \[
      \Sigma_0:=\bigcap_{k\in\mathbb Z}\Theta^{-k}\Sigma_*
    \]
    also has full measure and is $\Theta$-invariant. In what follows we work on $\Sigma_0$.

    \begin{thm}\label{t.060201}
      There is a unique family $\{\G_{\sigma}\}_{\sigma\in\Sigma_0}\subset \mathcal{P}_{V}(E)$, measurable in $\sigma$, such that
      \[\G_\sigma Q_\sigma=\G_{\Theta \sigma},\qquad \sigma\in \Sigma_0.\]
      Moreover, there exists $\lambda>0$ such that the following properties hold.
      \begin{itemize}
        \item Pullback exponential stability: for every $\sigma\in \Sigma_0$ and $x\in E$, there is $C_{\lambda}(\sigma,x)>0$ such that, for all $t\geq 0$,
        \begin{align*}
          \W_{D_\sigma}\left(\delta_x P_t^{\beta_{-t} \sigma}, \G_\sigma\right) \leq C_\lambda(\sigma,x)e^{-\lambda t}.
        \end{align*}
        \item Forward exponential stability: for each fixed $r\in[0,T)$, there is a full measure set $\Sigma_r\subset\Sigma_0$ such that, for every $\sigma\in \Sigma_r$ and $x\in E$, there is $C_{\lambda,r}(\sigma,x)>0$ such that, for all $n\geq0$,
        \begin{align*}
          \W_{D_{\beta_t \sigma}}\left(\delta_x P_t^\sigma, \G_{\beta_t \sigma}\right) \leq C_{\lambda,r}(\sigma,x)e^{-\lambda t},\qquad t=nT+r.
        \end{align*}
      \end{itemize}
      Moreover, there exists a universally measurable set $\widehat{\Sigma} \subset \Sigma$ with full measure, invariant under $\left(\beta_t\right)_{t \in \mathbb{R}}$, and a universally measurable modification $\left\{\widehat{\G}_\sigma\right\}_{\sigma \in \widehat{\Sigma}}$ of $\left\{\G_\sigma\right\}_{\sigma\in\Sigma_0}$, such that
      \[
        \widehat{\G}_\sigma P_t^\sigma=\widehat{\G}_{\beta_t \sigma}, \quad \forall \sigma \in \widehat{\Sigma}, \quad \forall t \geq 0.
      \]
    \end{thm}

    \begin{rem}
      For the family $\{\G_\sigma\}_{\sigma\in\Sigma_0}$, the continuous time invariance property holds only outside an exceptional set which may depend on time $t\in\mathbb{R}$. The perfected family $\{\widehat{\G}_\sigma\}_{\sigma\in\widehat{\Sigma}}$ removes this dependence: the same full-measure set $\widehat{\Sigma}$ works for all $t\geq0$. Furthermore, since $\rho\wedge1\leq D_\sigma$ on $\Sigma_0$, the same stability estimates hold for $\W_{\rho\wedge1}$.
    \end{rem}

    \begin{proof}
      The proof is divided into the following four steps. 
  
      {\it Step 1: Existence. }
      For $x\in E$, consider the pullback sequence 
      \[\mu_n^{\sigma, x}:=\delta_x Q_{\Theta^{-n} \sigma}^{(n)}=\delta_x P_{n T}^{\beta_{-n T} \sigma} .\]
      For $m>n$, by the cocycle property, we have $Q_{\Theta^{-m} \sigma}^{(m)}=Q_{\Theta^{-m} \sigma}^{(m-n)} Q_{\Theta^{-n} \sigma}^{(n)} .$ Hence, 
      \[\mu_m^{\sigma, x}=\nu_{m, n}^{\sigma, x} Q_{\Theta^{-n} \sigma}^{(n)}, \quad  \text{ where } \quad \nu_{m, n}^{\sigma, x}:=\delta_x Q_{\Theta^{-m} \sigma}^{(m-n)} .\]
      It follows from Lemma \ref{l.042602} that 
      \begin{align*}
        \W_{D_\sigma}\left(\mu_m^{\sigma, x}, \mu_n^{\sigma, x}\right) \leq C_\lambda(\sigma) e^{-\lambda n} \W_{D_{\Theta^{-n}\s}}\left(\nu_{m, n}^{\sigma, x}, \delta_x\right).
      \end{align*}
      We claim that $ \W_{D_{\Theta^{-n}\s}}\left(\nu_{m, n}^{\sigma, x}, \delta_x\right)$ is uniformly bounded for $m>n$. Indeed, as $d_\sigma \leq 1$, we have 
      \begin{align*}
        \begin{aligned}
          \W_{D_{\Theta^{-n}\sigma}}\left(\nu_{m, n}^{\sigma, x}, \delta_x\right) & \leq \int_E D_{\Theta^{-n} \sigma}(z, x) \nu_{m, n}^{\sigma, x}(d z) \\
          & \leq\left(1+\varkappa \int_E L_{\Theta^{-n} \sigma}(z) \nu_{m, n}^{\sigma, x}(d z)+\varkappa L_{\Theta^{-n} \sigma}(x)\right)^{1 / 2}.
          \end{aligned}
      \end{align*}
      By \eqref{e.042603} and \eqref{e.042504}, one has $L_{\Theta^{-n} \sigma}(x) \leq \eta V(x)$ and 
      \begin{align*}
        \begin{split}
          Q_{\Theta^{-m} \sigma}^{(m-n)} L_{\Theta^{-n} \sigma}(x) &\leq \eta^{m-n} L_{\Theta^{-m} \sigma}(x)+\frac{\eta}{1-\eta}\\
          &\leq \eta^{m-n+1} V(x)+\frac{\eta}{1-\eta} \leq \eta V(x)+\frac{\eta}{1-\eta}.
        \end{split}
      \end{align*}
      Consequently, 
      \begin{align}
        \W_{D_{\Theta^{-n}\sigma}}\left(\nu_{m, n}^{\sigma, x}, \delta_x\right)\leq \left(1+\varkappa Q_{\Theta^{-m} \sigma}^{(m-n)} L_{\Theta^{-n} \sigma}(x)+\varkappa L_{\Theta^{-n} \sigma}(x)\right)^{1 / 2}\leq C_x
      \end{align}
      independent of $m,n$. Hence, 
      \begin{align}\label{e.042606}
        \W_{D_\sigma}\left(\mu_m^{\sigma, x}, \mu_n^{\sigma, x}\right) \leq C_\lambda(\sigma) C_x e^{-\lambda n}.
      \end{align}
      By the definition of $D_{\s}$ and the assumption \eqref{e.042604}, taking $\a=\a_0\wedge\varrho$ in \eqref{e.042605}, one has
      \[\rho(x,y)\wedge1\leq p^{\varrho}(x,y)\wedge1\leq p^{\alpha}(x,y)\wedge1\leq \theta_{\alpha}(x,y), \, \text{ for }x,y\in E,\]
      and hence $\rho\wedge1\leq D_{\s}$. 
      Then \eqref{e.042606} implies that the sequence $\{\mu_n^{\sigma, x}\}_{n\geq 1}$ is Cauchy in $\mathcal{P}(E)$ with respect to the metric $W_{\rho\wedge 1} $. Therefore, there exist $\Gamma_\sigma\in \mathcal{P}(E)$ such that $\mu_n^{\sigma, x}$ converges to $\Gamma_\sigma$ weakly. It follows from the pullback contraction that the limit is independent of $x$ and by lower semicontinuity of $D_{\s}$, 
      \begin{align*}
        \W_{D_\sigma}\left(\mu_n^{\sigma, x}, \Gamma_\sigma\right) \leq \liminf _{m \rightarrow \infty} \W_{D_\sigma}\left(\mu_n^{\sigma, x}, \mu_m^{\sigma, x}\right) \leq C_\lambda(\sigma) C_x e^{-\lambda n}. 
      \end{align*}
      Note that $\s\to\G_{\s}$ is measurable as the transition probabilities $P_{t}^{\s}(x,\cdot)$ are measurable. Moreover, it follows from Fatou's lemma and \eqref{e.042504} that 
      \[\int_E L_\sigma(z) \Gamma_\sigma(d z) \leq \liminf _{n \rightarrow \infty} \int_E L_\sigma(z) \mu_n^{\sigma, x}(d z) \leq \frac{\eta}{1-\eta},\]
      which implies that $\{\G_{\s}\}\subset \mathcal{P}_{V}(E)$.

      {\it Step 2: Invariance and uniqueness.}
      Note that 
      \[\mu_n^{\sigma,x} Q_\sigma=\delta_x Q_{\Theta^{-n} \sigma}^{(n)} Q_\sigma=\delta_x Q_{\Theta^{-n} \sigma}^{(n+1)} \rightarrow \Gamma_{\Theta \sigma},\quad n\rightarrow \infty\]
      and Corollary \ref{c.042601} combined with the Feller property gives 
      \[\W_{D_{\Theta \sigma}}\left(\mu_n^{\sigma,x} Q_\sigma, \Gamma_\sigma Q_\sigma\right) \leq h \W_{D_\sigma}\left(\mu_n^{\sigma,x}, \Gamma_\sigma\right), \quad n\rightarrow \infty.\]
      Therefore,  $\Gamma_\sigma Q_\sigma=\Gamma_{\Theta \sigma}$. 
   
      For uniqueness, if $\widetilde{\Gamma}_\sigma \in \mathcal{P}_V(E)$ is another invariant family, then 
      \[\Gamma_\sigma=\Gamma_{\Theta^{-n} \sigma} Q_{\Theta^{-n} \sigma}^{(n)}, \quad \widetilde{\Gamma}_\sigma=\widetilde{\Gamma}_{\Theta^{-n} \sigma} Q_{\Theta^{-n} \sigma}^{(n)}.\]
      By Lemma \ref{l.042603} and Lemma \ref{l.042602}, 
      \[\W_{D_{\s}}(\Gamma_\sigma,\widetilde{\Gamma}_{\s})\leq C_{\l}(\s) e^{-\l n} \W_{D_{\Theta^{-n}\s}}\left(\Gamma_{\Theta^{-n}},\widetilde{\Gamma}_{\Theta^{-n} \sigma}\right)\to0\]
      as $n\to\infty$. Hence $\Gamma_\sigma = \widetilde{\Gamma}_\sigma$. 
  
      {\it Step 3: Passage to continuous time.} We first establish the exponential stability. Let $t=nT+r$ with $n=\lfloor t/T\rfloor$ and $0\leq r<T$, and set $\sigma_n:=\Theta^{-n}\sigma$. By the cocycle property,
      \begin{align*}
        \delta_x P_t^{\beta_{-t}\sigma}=\nu_{n,r}^{\sigma,x}Q_{\Theta^{-n}\sigma}^{(n)},\qquad \nu_{n,r}^{\sigma,x}:=\delta_xP_r^{\beta_{-t}\sigma}=\delta_xP_r^{\beta_{-r}\sigma_n}.
      \end{align*}
      Since $d_\sigma\leq1$, for any $\mu,\nu\in\mathcal P_V(E)$
      we have 
      \begin{align*}
        \W_{D_\sigma}(\mu,\nu)\leq\left(1+\varkappa\int_E L_\sigma(u)\mu(du)+\varkappa\int_E L_\sigma(u)\nu(du)\right)^{1/2}.
      \end{align*}
      By $(H1)$, for $0\leq r<T$,
      \begin{align*}
        P_r^{\beta_{-r}\sigma_n}V(x)\leq V(x)+\int_{-r}^{0}K(\beta_s\sigma_n)ds\leq V(x)+e^{\gamma T}K_T(\Theta^{-1}\sigma_n).
      \end{align*}
      Since $R(\sigma_n)\geq \eta^{-1}(1+K_T(\Theta^{-1}\sigma_n))$, it follows that
      \begin{align*}
        P_r^{\beta_{-r}\sigma_n}L_{\sigma_n}(x)\leq \eta V(x)+\eta e^{\gamma T}.
      \end{align*}
      Together with Lemma \ref{l.042603}, this gives
      \begin{align*}
        \W_{D_{\Theta^{-n}\sigma}}\left(\nu_{n,r}^{\sigma,x},\Gamma_{\Theta^{-n}\sigma}\right)\leq \left(1+\varkappa\eta V(x)+\varkappa\eta e^{\gamma T}+\frac{\varkappa\eta}{1-\eta}\right)^{1/2}=:C_x .
      \end{align*}
      Therefore the pullback contraction  in Lemma \ref{l.042602} and the invariance $\Gamma_\sigma=\Gamma_{\Theta^{-n}\sigma}Q_{\Theta^{-n}\sigma}^{(n)}$ yield
      \begin{align}\label{e.042701}
        \begin{split}
          \W_{D_\sigma}\left(\delta_xP_t^{\beta_{-t}\sigma},\Gamma_\sigma\right)&=\W_{D_\sigma}\left(\nu_{n,r}^{\sigma,x}Q_{\Theta^{-n}\sigma}^{(n)},\Gamma_{\Theta^{-n}\sigma}Q_{\Theta^{-n}\sigma}^{(n)}\right)\\
          &\leq C_\lambda(\sigma)e^{-\lambda n}\W_{D_{\Theta^{-n}\sigma}}\left(\nu_{n,r}^{\sigma,x},\Gamma_{\Theta^{-n}\sigma}\right)\\
          &\leq C_\lambda(\sigma)C_xe^{-\lambda n}\leq  C_\lambda(\sigma,x)e^{-\frac{\lambda}{T} t}.
        \end{split}
      \end{align}

      We next prove the forward stability estimate. Fix $r\in[0,T)$ and set
      \begin{align*}
        \Sigma_r:=\Sigma_0\cap\beta_{-r}\Sigma_0\cap\left\{\sigma:\int_0^rK(\beta_s\sigma)ds<\infty\right\}.
      \end{align*}
      Then $m(\Sigma_r)=1$. Let $\sigma\in\Sigma_r$, $t=nT+r$ and $\sigma_r:=\beta_r\sigma$. Since $\sigma_r\in\Sigma_0$, the cocycle property and the skeleton invariance give
      \begin{align*}
        \delta_xP_t^\sigma=\delta_xP_r^\sigma Q_{\sigma_r}^{(n)},\qquad \Gamma_{\beta_t\sigma}=\Gamma_{\Theta^n\sigma_r}=\Gamma_{\sigma_r}Q_{\sigma_r}^{(n)}.
      \end{align*}
      Hence the forward contraction in Lemma \ref{l.042602} implies
      \begin{align*}
        \W_{D_{\beta_t\sigma}}\left(\delta_xP_t^\sigma,\Gamma_{\beta_t\sigma}\right)&=\W_{D_{\Theta^n\sigma_r}}\left(\delta_xP_r^\sigma Q_{\sigma_r}^{(n)},\Gamma_{\sigma_r}Q_{\sigma_r}^{(n)}\right)\\
        &\leq C_\lambda(\sigma_r)e^{-\lambda n}\W_{D_{\sigma_r}}\left(\delta_xP_r^\sigma,\Gamma_{\sigma_r}\right).
      \end{align*}
      Moreover, by $(H1)$ and Lemma \ref{l.042603},
      \begin{align*}
        P_r^\sigma L_{\sigma_r}(x)\leq \frac{V(x)+\int_0^rK(\beta_s\sigma)ds}{R(\sigma_r)},\qquad \int_E L_{\sigma_r}(u)\Gamma_{\sigma_r}(du)\leq\frac{\eta}{1-\eta}.
      \end{align*}
      Therefore
      \begin{align*}
        \W_{D_{\sigma_r}}\left(\delta_xP_r^\sigma,\Gamma_{\sigma_r}\right)\leq\left(1+\varkappa\frac{V(x)+\int_0^rK(\beta_s\sigma)ds}{R(\sigma_r)}+\frac{\varkappa\eta}{1-\eta}\right)^{1/2}=:C_r(\sigma,x)<\infty.
      \end{align*}
      Consequently,
      \begin{align*}
        \W_{D_{\beta_t\sigma}}\left(\delta_xP_t^\sigma,\Gamma_{\beta_t\sigma}\right)\leq C_\lambda(\sigma_r)C_r(\sigma,x)e^{-\lambda n}\leq C_{\lambda,r}(\sigma,x)e^{-\frac{\lambda}{T} t},\qquad t=nT+r,\quad n\geq0.
      \end{align*}

      Next we show invariance and uniqueness for continuous time with exceptional set depending on time. Since for any fixed $s\geq 0$, by \eqref{e.042701} and the cocycle property, we have 
      \begin{align*}
        \Gamma_\sigma P_s^\sigma=\lim _{t \rightarrow \infty} \delta_x P_t^{\beta_{-t} \sigma} P_s^\sigma=\lim _{t \rightarrow \infty} \delta_x P_{t+s}^{\beta_{-t} \sigma}=\Gamma_{\beta_s \sigma},
      \end{align*}
      for all $\s\in \Sigma_{0}\cap\beta_{-s}\Sigma_0$. Hence for each fixed $s\geq 0$, we have  $\Gamma_\sigma P_s^\sigma=\Gamma_{\beta_s \sigma},\, m \text {-a.s.}.$
      If $\widetilde{\Gamma}_\sigma \in \mathcal{P}_V(E)$ is another $P_t^{\s}$ invariant family, then it is invariant particularly for $t=T$. So by uniqueness of the $Q_{\s}$-invariant measure, $\widetilde{\Gamma}_\sigma =\Gamma_\sigma$. 
  
      {\it Step 4: Perfection.} Setting 
      \[\widehat{\Sigma}:=\bigcup_{r \in[0, T)} \beta_r \Sigma_0,\]
      it is clear that $\Sigma_0 \subset \widehat{\Sigma}$. Hence $m(\widehat{\Sigma})=1$, where we use the fact that $\widehat{\Sigma}=F([0,T)\times\Sigma_0)$ is universally measurable as $F(r,\sigma):=\beta_r\sigma$ is Borel measurable.  We claim that 
      \[\beta_t\widehat{\Sigma} = \widehat{\Sigma} \quad \text{ for all }t\in\mathbb R.\] Indeed, noting that for any $\eta\in \widehat{\Sigma}$, there exists $r\in[0,T)$ and $\s\in \Sigma_0$ such  that $\eta=\beta_r\s$. For $t\in\mathbb R$ we write $t+r=kT+r_0$ for some $r_0\in[0,T)$. Then 
      \begin{align*}
        \beta_t\eta = \beta_{t+r}\s=\beta_{r_0}\beta_{kT}\sigma\in \widehat{\Sigma}
      \end{align*}
      by the definition of $\widehat{\Sigma}$ and the $\b_T$ invariance of $\Sigma_0$. This shows that $\beta_t\widehat{\Sigma}\subset \widehat{\Sigma}$ and in turn implies the claim as the flow $\b_t$ is invertible. 
  
      We now define a modification $\widehat{\Gamma}$ of $\G$ as follows. For $\eta \in \widehat{\Sigma}$, we take a representative $\s\in\Sigma_0$ such that $\eta=\b_{r}\s$ for some $r\in[0,T)$ and define  $\widehat{\Gamma}_\eta:=\Gamma_\sigma P_r^\sigma$. 
      In particular, if $\eta=\sigma\in \Sigma_0$ then $r=0$ and $\widehat{\Gamma}_\eta=\Gamma_\s$. Hence $\widehat{\Gamma}$ coincides with $\G$ on $\Sigma_0$. This is well-defined since, if $\eta=\beta_r \sigma=\beta_{r^{\prime}} \sigma^{\prime}$ for different $r,r'$ (we may assume $r\geq r'$) and $\s,\s'$, then $\sigma^{\prime}=\beta_{r-r^{\prime}} \sigma$ and as $\s,\s'\in\Sigma_0$, we have $\sigma \in \Sigma_0 \cap \beta_{-\left(r-r^{\prime}\right)} \Sigma_0$. The invariance for fixed time then implies $\Gamma_\sigma P_{r-r^{\prime}}^\sigma=\Gamma_{\sigma^{\prime}}$, 
      which together with the cocycle property yields 
      \[\Gamma_\sigma P_r^\sigma=\Gamma_\sigma P_{r-r^{\prime}}^\sigma P_{r^{\prime}}^{\sigma^{\prime}}=\Gamma_{\sigma^{\prime}} P_{r^{\prime}}^{\sigma^{\prime}}.\]
      Hence $\widehat{\Gamma}_\eta$ is independent of the choice of the representative. 
      
      Now we prove that on $\widehat{\Sigma}$, the invariance property of $\widehat{\Gamma}$ holds for all $s\geq 0$. Taking any $\eta \in \widehat{\Sigma}$, writing $\eta=\beta_r \sigma$ for $\sigma \in \Sigma_0, r \in[0, T)$ and 
      \[r+s=k T+r_0, \quad k \in \mathbb{N}_0, \quad r_0 \in[0, T),\]
      it follows that 
      \[\beta_s \eta=\beta_{r+s} \sigma=\beta_{r_0} \Theta^k \sigma.\]
      By definition $\widehat{\Gamma}_\eta=\Gamma_\sigma P_r^{\s}$ and skeleton invariance $\Gamma_\sigma Q_\sigma^{(k)}=\Gamma_{\Theta^k \sigma}$, we have  
      \begin{align*}
        \begin{aligned}
          \widehat{\Gamma}_\eta P_s^\eta & =\Gamma_\sigma P_r^\sigma P_s^{\beta_r \sigma}  =\Gamma_\sigma P_{r+s}^\sigma =\Gamma_\sigma P_{k T+r_0}^\sigma\\
          &= \Gamma_{\Theta^k \sigma} P_{r_0}^{\Theta^k \sigma}=\widehat{\Gamma}_{\beta_{r_0} \Theta^k \sigma}=\widehat{\Gamma}_{\beta_s \eta}.
          \end{aligned}
      \end{align*}
      Consequently, 
      \[\widehat{\Gamma}_\eta P_s^\eta=\widehat{\Gamma}_{\beta_s \eta}, \quad \forall \eta \in \widehat{\Sigma}, \quad \forall s \geq 0 \]
      as desired. 
  \end{proof}

  \section{An abstract Freidlin--Wentzell principle for stationary measures}
  In this section, we formulate an abstract criterion for the Freidlin--Wentzell large deviation principle of stationary measures in the small noise limit. The deterministic equation is allowed to be non-autonomous through the driving environment base flow, and its long-time behavior is described by a compact pullback attractor. The small random perturbation gives rise to a time-inhomogeneous Markov cocycle with a unique stationary measure. Our goal is to identify the large deviation rate of the family of stationary measures in terms of the minimal controlled energy needed to reach a point from the pullback attractor in the remote past.

  Let $(\Sigma,\beta_t,m)$ be  an ergodic invertible measure preserving flow.  For each $\sigma\in\Sigma$, consider the evolution equation
  \begin{align}\label{e.080301}
      \partial_t u=F(\beta_t\sigma,u), \qquad u(0)=u_0 .
  \end{align}
  We assume that \eqref{e.080301} is well-posed in a separable Hilbert space $H$ with the solution map $\Phi_{0,t,\sigma}$, and that it has a pullback attractor $\Ac(\sigma)$ which is compact in $H$.
  
  We then consider the random perturbation of \eqref{e.080301} with a parameter $\varepsilon\in(0,1]$,
  \begin{align}\label{e.080302}
      \partial_t u=F(\beta_t\sigma,u)+\sqrt{\eps}Q(\beta_t\sigma)\partial_tW, \qquad u(0)=u_0,
  \end{align}
  where $W$ is a two-sided cylindrical Wiener process in $H$ over the canonical Wiener space $(\O,\Fc,\mathbf{P})$. Assume that for each $\s\in \Sigma$, $Q(\sigma):H\to H$ is Hilbert--Schmidt and $\sigma \mapsto\|Q(\sigma)\|_{\mathrm{HS}}^2 \in L^1(\Sigma, m)$. We also assume that \eqref{e.080302} is well-posed in $H$ with the solution map $\Phi_{0,t,\sigma}^{\eps}$. It follows from uniqueness of solutions that
  \begin{align}\label{e.080303}
      \Phi_{s,s+t,\sigma}^{\eps}(\omega)
      =\Phi_{0,t,\beta_s\sigma}^{\eps}(\theta_s\omega),
      \qquad t\geq 0,\ s\in\Rb,
  \end{align}
  where $\theta_s\omega(t)=\omega(t+s)-\omega(s)$ is the Wiener shift. Since the deterministic system is non-autonomous, the Markov process $\Phi_{0,t,\sigma}^{\eps}u_0$ is time inhomogeneous. The corresponding Markov transition operator is
  \begin{align*}
      P_{0,t,\sigma}^{\eps}\phi(u)
      =\mathbf{E}_u\phi(\Phi_{0,t,\sigma}^{\eps}u),
      \qquad \phi\in B_b(H),\ u\in H .
  \end{align*}
  Property \eqref{e.080303} and the $\mathbf{P}$-invariance of the Wiener shift yield the translation identity
  \begin{align*}
      P_{s,s+t,\sigma}^{\eps}=P_{0,t,\beta_s\sigma}^{\eps},
      \qquad s\in\Rb,\ t\geq0 .
  \end{align*}
  We assume that the randomly forced equation \eqref{e.080302} has a unique family of stationary measures $\{\mu_{\sigma}^{\eps}\}_{\sigma\in\Sigma}$ such that
  \begin{align*}
      \mu_{\sigma}^{\eps}P_{0,t,\sigma}^{\eps}
      =\mu_{\beta_t\sigma}^{\eps},
      \qquad \sigma\in\Sigma,\ t\geq0 .
  \end{align*}

  To describe the rate function, we introduce the deterministic control problem associated with the small noise perturbation. The control represents the large deviation cost of replacing the noise by a Cameron--Martin path. Consider the following controlled equation with a control $\varphi\in L_{\mathrm{loc}}^2(\Rb;H)$:
  \begin{align}\label{e.080601}
      \partial_t u=F(\beta_t\sigma,u)+Q(\beta_t\sigma)\varphi .
  \end{align}
  The solution of \eqref{e.080601} with initial condition $u_0\in H$ will be denoted by $\Phi_{0,t,\sigma}^{\varphi}u_0$. For any $T\geq0$ and $u\in C([0,T];H)$, the energy of $u$ with time symbol $\sigma$ is defined by
  \begin{align*}
      I_{u(0),T}^{\sigma}(u)
      =\inf_{\varphi}J_{u(0),T}^{\sigma}(\varphi)\quad \text{with}\quad
      J_{u(0),T}^{\sigma}(\varphi)=\frac12\int_0^T\|\varphi(s)\|^2\,ds,
  \end{align*}
  where the infimum is taken over all $\varphi\in L^2(0,T;H)$ such that $u(t)=\Phi_{0,t,\sigma}^{\varphi}u(0)$ for all $t\in[0,T]$. If no such control exists, then $I_{u(0),T}^{\sigma}(u)=\infty$.
  
  For each $\sigma\in\Sigma$, define the rate function (quasipotential) $E_{\Ac(\sigma)}:H\to[0,\infty]$ by
  \begin{align}\label{e.073003}
      E_{\Ac(\sigma)}(u_*)
      =\lim_{\eta\to0}\inf\Bigl\{
          I_{u(0),s}^{\beta_{-s}\sigma}(u): s>0,\ u\in C([0,s];H),
          \ u(0)\in\Ac(\beta_{-s}\sigma),\ u(s)\in B_{\eta}(u_*)
      \Bigr\},
  \end{align}
  for $u_*\in H$. Since the limiting dynamics is non-autonomous, the quasipotential is defined in a pullback form.  Namely, to reach a point near $u_*$ at the environment $\sigma$, one starts from the attractor $\mathcal A(\beta_{-s}\sigma)$ at time $-s$, evolves under the controlled equation, and then lets the starting time tend to the remote past.
    It follows that
  \begin{align*}
      \Ac(\sigma)\subset\{u\in H:E_{\Ac(\sigma)}(u)=0\}.
  \end{align*}

  The following finite time uniform LDP will be used as one of the assumptions.  It is stated in a form uniform over bounded sets of initial data, which is the natural form needed when applying the Markov property to stationary measures. Let $T>0$ and $\sigma\in\Sigma$. 
  \begin{defn}\label{d.070201}
    The family $\{\Phi_{0,\cdot,\sigma}^{\eps}u_0\}_{\eps\in(0,1]}$ of solutions of \eqref{e.080302} with time symbol $\sigma$ satisfies the Freidlin--Wentzell uniform LDP on $C([0,T];H)$ with rate function $I_{u_0,T}^{\sigma}$ if the following two conditions hold for the law $\nu_{u_0,\sigma}^{\eps}$ of the solution $\Phi_{0,\cdot,\sigma}^{\eps}u_0$:
    \begin{itemize}
        \item[(1)] For every $R>0$, $M\geq0$, and $\d,\d'>0$, there exists $\eps_0>0$ such that
        \begin{align*}
          \inf_{\|u_0\|\leq R}\inf_{\xi\in\{I_{u_0,T}^{\sigma}\leq M\}}
          \left(
          \nu_{u_0,\sigma}^{\eps}\bigl(\Nc_{\d'}(\xi)\bigr)
          -e^{-(I_{u_0,T}^{\sigma}(\xi)+\d)/\eps}
          \right)\geq0,
          \qquad \forall \eps\in(0,\eps_0],
        \end{align*}
        where $\Nc_{\d'}(\xi)$ is the $\d'$-neighborhood of $\xi$ in $C([0,T];H)$.
    
        \item[(2)] For every $R>0$, $M\geq0$, and $\d,\d'>0$, there exists $\eps_0>0$ such that
        \begin{align*}
          \sup_{\|u_0\|\leq R}
          \nu_{u_0,\sigma}^{\eps}
          \left(C([0,T];H)\setminus\Nc_{\d'}(I_{u_0,T}^{\sigma},M)\right)
          \leq e^{-(M-\d)/\eps},
          \qquad \forall \eps\in(0,\eps_0],
        \end{align*}
        where $\Nc_{\d'}(I_{u_0,T}^{\sigma},M)$ is the $\d'$-neighborhood of the level set $\{I_{u_0,T}^{\sigma}\leq M\}$.
    \end{itemize}
  \end{defn}
  We collect the assumptions needed to pass from the finite time trajectory LDP to the large deviation principle for the stationary family.  The first two assumptions provide the deterministic attractor and the stochastic stationary family.  The compactness and tracking assumptions identify the correct quasipotential level sets and allow controlled trajectories starting near the pullback attractor to remain close to those level sets.  The escaping energy and weak exponential tightness assumptions prevent the stationary measures from placing exponentially relevant mass far away from the deterministic pullback attractor.
 
  \begin{assumption}\label{a.050201}
    We assume that
  \begin{enumerate}[label=(H\arabic*)]
      \item The deterministic system \eqref{e.080301} has a compact pullback attractor $\Ac(\sigma)$ which is invariant and pullback attracts every bounded
      set.
  
      \item The stochastic system \eqref{e.080302} has a unique stationary measure and satisfies the trajectory Freidlin--Wentzell uniform LDP.
  
      \item Compactness of level sets: for each $M>0$, the set $K_M(\s):=\{u\in H: E_{\Ac(\sigma)}(u)\leq M\}$ is compact in $H$. 
  
      \item Tracking property: for any $\d,\d',M>0$, there exists $\eta=\eta(\s)\in(0,\d)$ $m$-a.s., such that 
      \begin{align}\label{e.072701}
        \left\{
        u(t):u(0)\in\Ac_{\eta(\beta_{-t}\s)}(\beta_{-t}\sigma),
        I_{u(0),t}^{\beta_{-t}\sigma}(u)\leq M-\d'
        \right\}
        \subset \Nc_{\d}\bigl(K_M(\s)\bigr),
        \qquad \forall \ t>0,
      \end{align}
      where $\Nc_{\d}(K_M(\s))$ is the $\d$-neighborhood of the level set $K_M(\s)$ and $\Ac_{\eta}(\sigma)$ is the $\eta$-neighborhood of  $\Ac(\sigma)$. 
      \item Nontrivial escaping energy: for any $R,T>0$, there is a finite measurable
      $C_{R,T}:\Sigma\to [1,\infty)$ such that, for $m$-a.e. $\sigma\in\Sigma$, for every
      $u_0\in B_R$ and every admissible control $\varphi\in L^2(0,T;H)$,
      \begin{align}\label{e.070101}
        \left\|
        \Phi_{0,T,\beta_{-T}\sigma}^{\varphi}u_0
        -
        \Phi_{0,T,\beta_{-T}\sigma}u_0
        \right\|^2
        \leq
        C_{R,T}(\sigma)J_{u_0,T}^{\beta_{-T}\sigma}(\varphi).
      \end{align}
      \item Weak exponential tightness:
      \begin{align}\label{e.072703}
        \lim_{R\to\infty}\limsup_{\eps\to0}
        \eps\ln\mu_{\sigma}^{\eps}(B_R^c)
        =-\infty, \quad m\text{-a.s.}.
      \end{align}
  \end{enumerate}
  \end{assumption}

  Several comments on these assumptions are in order.  The existence of compact pullback attractors is standard for many dissipative non-autonomous PDEs.  For additively forced stochastic PDEs, finite time Freidlin--Wentzell uniform LDPs are usually obtained either by the weak convergence approach or by a contraction principle.  The uniqueness of the stationary family is supplied in our applications by the generalized-coupling criterion developed in Section \ref{s.060901}, which is designed to cover degenerate noise under weak dissipativity assumptions.  The remaining assumptions are the abstract large-deviation inputs needed to transfer finite time trajectory estimates to stationary measures.
  \begin{rem}
    By $(H5)$, if $u\in C([0,T];H)$ is a controlled path with $u(0)=u_0\in B_R$ and
    $I_{u_0,T}^{\beta_{-T}\sigma}(u)<\infty$, then by taking the infimum over all admissible controls generating the same path $u$, we have 
    \begin{align}\label{e.070304}
      \left\|u(T)-\Phi_{0,T,\beta_{-T}\sigma}u_0\right\|^2\leq
      C_{R,T}(\sigma)I_{u_0,T}^{\beta_{-T}\sigma}(u).
    \end{align}
    This estimate, combined with $(H1)$, implies that for suitable large $T$ and for $\sigma$ in a large measure subset of $\Sigma$,
    \begin{align*}
      a_{R, \eta, T}(\sigma):=\inf \left\{I_{u(0),T}^{\beta_{-T} \sigma}(u): u(0) \in B_R, u(T) \notin \mathcal{A}_{\eta(\sigma)}(\sigma)\right\}
    \end{align*}
    has a positive lower bound, giving a nontrivial escaping energy. See Lemma \ref{l.070101} for related energy control. 
  \end{rem}
  Our main result on the Freidlin--Wentzell LDP is given below, whose proof will be given in the following subsections. 
  \begin{thm}\label{t.080601}
      Under Assumption \ref{a.050201}, for $m$-a.e. $\sigma\in\Sigma$, the family of stationary measures $\{\mu_{\sigma}^{\eps}\}_{\eps\in(0,1]}$ satisfies the large deviation upper bound
      \begin{align*}
          \limsup_{\eps\to0}\eps\ln\mu_{\sigma}^{\eps}(F)
          \leq -\inf_{u\in F}E_{\Ac(\sigma)}(u),
          \qquad \text{for any closed } F\subset H .
      \end{align*}
      If the attractor of the limit system \eqref{e.080301} is a random point, then the large deviation lower bound also holds:
      \begin{align*}
          \liminf_{\eps\to0}\eps\ln\mu_{\sigma}^{\eps}(G)
          \geq -\inf_{u\in G}E_{\Ac(\sigma)}(u),
          \qquad \text{for any open } G\subset H .
      \end{align*}
  \end{thm}

  \begin{rem}[Uniform version on compact bases]
    The $m$-a.s. formulation of the main results is only needed because the base flow is allowed to be merely measurable.  In the compact continuous setting the exceptional sets can be removed.  Indeed, assume that $\Sigma$ is a compact metric space, $\beta_t$ is a continuous invertible flow, and the coefficients entering the equation and the coupling construction depend continuously on $\sigma$.  If the estimates verifying Theorem \ref{t.042501} and Assumption \ref{a.050201} are uniform in $\sigma$, then all good environment sets in the proofs can be chosen to be $\Sigma$. 
    Consequently the exponential mixing estimates and the Freidlin--Wentzell LDP bounds obtained in this paper hold for every environment variable $\sigma\in\Sigma$.
  \end{rem}

  \subsection{Upper bound}
  This subsection proves the large deviation upper bound.  The argument combines three ingredients.  First, weak exponential tightness allows us to restrict the stationary measures to a large ball at an exponentially negligible cost.  Second, the nontrivial escaping energy condition shows that a controlled path which stays in this ball but avoids a small neighborhood of the pullback attractor over sufficiently many selected time blocks must accumulate large action.  Third, if the path enters a small neighborhood of the pullback attractor, the tracking property implies that every controlled trajectory with action below the prescribed level must end near the quasipotential level set $K_M(\sigma)$. Together with the finite time trajectory LDP and the stationarity relation, these estimates yield the desired upper bound.  To deal with the nonuniform dynamics produced by the time inhomogeneity, we introduce good environment sets and work along recurrent good times, thereby passing the LDP from the trajectory level to the stationary measure level.

  For $R,T>0$ and measurable $\eta:\Sigma\to (0,\infty)$, recall that 
  \begin{align}\label{e.070204}
    a_{R, \eta, T}(\sigma):=\inf \left\{I_{u(0),T}^{\beta_{-T} \sigma}(u): u(0) \in B_R, u(T) \notin \mathcal{A}_{\eta(\sigma)}(\sigma)\right\}
  \end{align}
  is the energy for escaping from the random attractor. For $L,c>0$, we also set 
  \begin{align*}
    \mathcal{H}_{R, L}:=\left\{\sigma \in \Sigma: \limsup _{\varepsilon \rightarrow 0} \varepsilon \log \mu_\sigma^{\varepsilon}\left(B_R^c\right)<-L\right\},
  \end{align*}
  and 
  \begin{align*}
    \mathcal{G}:=\left\{\sigma \in \Sigma: a_{R+1, \eta / 2, T}(\sigma) \geq c\right\} \cap \beta_T \mathcal{H}_{R, L}. 
  \end{align*}

  The following lemma provides the good environment sets on which the dynamics is controllable. 
  \begin{lem}\label{l.070101}
    Assume $(H1)$, $(H5)$ and $(H6)$ from Assumption \ref{a.050201}. For any $L>0, \rho\in(0,1)$ and measurable $\eta:\Sigma\to (0,\infty)$, there are $R,T,c>0$ such that $\beta_T$ is ergodic and 
    \[m(\mathcal{G})\geq 1-\rho.\] 
    Moreover, given any $n_*\in\mathbb N$, for $m$-a.e. $\sigma \in \Sigma$, there is $n\in\mathbb N$ such that $\beta_{-nT} \sigma \in \mathcal{H}_{R, L}$, and the selected index set
    \[\mathcal{J}_n:=\left\{0 \leq k \leq n-1: \beta_{-kT} \sigma \in \mathcal{G}\right\}\]
    satisfies $\left|\mathcal{J}_n\right| \geq n_*$. 
  \end{lem}
  \begin{proof}
    Set $\Lambda_R(\s):=\limsup_{\eps\to0}\eps\log\mu_{\s}^{\eps}(B_R^c)$, which tends to $-\infty$ as $R\to\infty$ for $m$-almost surely 
    by \eqref{e.072703} in $(H6)$. Since $\beta_T$ is measure preserving and invertible, choosing $R>0$ large enough such that $m(\mathcal H_{R,L})>1-\frac{\rho}{3}$,  we have 
    \begin{align*}
      m(\beta_T\mathcal H_{R,L})=m(\mathcal H_{R,L})>1-\frac{\rho}{3}
    \end{align*}
    for every $T>0$.

    We next choose $T$. For the above fixed $R$, define
    \begin{align*}
      d_T(\s):=\sup_{u_0\in B_{R+1}}\di\left(\Phi_{0,T,\beta_{-T}\s}u_0,\Ac(\s)\right).
    \end{align*}
    By the pullback attraction property of $\Ac(\s)$ from $(H1)$, for $m$-a.e. $\s\in\Sigma$, one has $d_T(\s)\to0$ as $T\to\infty$. Since $\eta>0$, it follows that
    \begin{align*}
      m\left\{\s\in\Sigma:d_T(\s)<\frac{\eta(\s)}{4}\right\}\to1,\qquad T\to\infty.
    \end{align*}
    Therefore, by Remark \ref{r.042601}, we may choose $T>0$ such that $\beta_T$ is ergodic and
    \begin{align}\label{e.070201}
      m(D_T)>1-\frac{\rho}{3},\qquad \text{ for } D_T:=
      \left\{\s\in\Sigma:d_T(\s)<\frac{\eta(\s)}{4}\right\}.
    \end{align}

    We now choose $c$. Set $ b_T(\s):=\frac{\eta(\s)^2}{16C_{R+1,T}(\s)}$
    where
    $C_{R+1,T}(\s)$ is given by 
    \eqref{e.070304} with $R$ replaced by $R+1$.
    Since $b_T(\s)>0$ for $m$-a.e. $\s\in\Sigma$, we can choose $c>0$ sufficiently small such that
    \begin{align*}
      m\{\s\in\Sigma:b_T(\s)\geq c\}>1-\frac{\rho}{3}.
    \end{align*}
    We claim that 
    \begin{align}\label{e.070203}
      D_T\cap\{b_T\geq c\}\subset\left\{\s\in\Sigma:a_{R+1,\eta/2,T}(\s)\geq c\right\}.
    \end{align}
    Indeed, fix $\s\in D_T\cap\{b_T\geq c\}$. Let $u$ be any controlled path such that 
    \begin{align}\label{e.070202}
      u(0)\in B_{R+1}, \quad u(T)\notin\Ac_{\eta(\s)/2}(\s)\quad
      \text{and } I_{u(0),T}^{\beta_{-T}\s}(u)<\infty.
    \end{align}
    Since $\s\in D_T$ defined in \eqref{e.070201}, we have
    \begin{align*}
      \di\left(
      \Phi_{0,T,\beta_{-T}\s}u(0),\Ac(\s)
      \right)
      <\frac{\eta(\s)}{4},
    \end{align*}
    combining it with \eqref{e.070202}, we infer 
    \begin{align*}
      \left\|u(T)-\Phi_{0,T,\beta_{-T}\s}u(0)\right\|\geq\frac{\eta(\s)}{4}.
    \end{align*}
    It then follows from \eqref{e.070304} implied by $(H5)$ that 
    \begin{align*}
      I_{u(0),T}^{\beta_{-T}\s}(u)\geq\frac{\eta(\s)^2}{16C_{R+1,T}(\s)} =b_T(\s)\geq c.
    \end{align*}
    Taking the infimum over all such controlled paths yields $a_{R+1,\eta/2,T}(\s)\geq c$, proving the claim \eqref{e.070203}.
    Consequently,
    \begin{align*}
      m\left\{\s\in\Sigma:a_{R+1,\eta/2,T}(\s)\geq c
      \right\}\geq m\left(D_T\cap\{b_T\geq c\}\right)>1-\frac{2\rho}{3}.
    \end{align*}
    Combining this with $m(\beta_T\mathcal H_{R,L})>1-\rho/3$, we conclude that
    \begin{align*}
      m(\mathcal G)=m\left(\left\{\s\in\Sigma:a_{R+1,\eta/2,T}(\s)\geq c
      \right\}\cap\beta_T\mathcal H_{R,L}
      \right)>1-\rho.
    \end{align*}

    Since $\beta_T$ is ergodic, by Birkhoff's ergodic theorem, for $m$-a.e. $\sigma\in\Sigma$, as $n\to\infty$ we have 
    \[\frac{1}{n} \sum_{k=0}^{n-1} \1_{\mathcal{G}}\left(\beta_{-kT} \sigma\right) \longrightarrow m(\mathcal{G})>0, \quad \frac{1}{n} \sum_{k=0}^{n-1} \1_{\mathcal{H}_{R, L}}\left(\beta_{-kT} \sigma\right) \longrightarrow m\left(\mathcal{H}_{R, L}\right)>0.\]
    The second convergence implies that the backward orbit visits $\mathcal{H}_{R, L}$ infinitely many times. Let $n_j \rightarrow \infty$ be a sequence such that $\beta_{-n_jT}\sigma \in \mathcal{H}_{R, L}$. Along the same sequence,
    \begin{align*}
    \left|\left\{0 \leq k \leq n_j-1: \beta_{-kT}  \sigma \in \mathcal{G}\right\}\right|=\sum_{k=0}^{n_j-1} \mathbf{1}_{\mathcal{G}}\left(\beta_{-kT}  \sigma\right) \sim n_j m(\mathcal{G}) \rightarrow \infty .
    \end{align*}
    Choosing $j$ large enough gives the desired $n=n_j$ such that $\beta_{-nT} \sigma \in \mathcal{H}_{R, L}$ and $\left|\mathcal{J}_n\right| \geq n_*$. 

  \end{proof}

  We are now ready to prove the upper bound estimate of Theorem \ref{t.080601}.  
  \begin{thm}\label{t.070201}
    Assume Assumption \ref{a.050201}. For $m$-a.e. $\sigma\in\Sigma$,
    \begin{align*}
      \limsup_{\eps\to0}\eps\ln\mu_{\sigma}^{\eps}(F)
      \leq -\inf_{u\in F}E_{\Ac(\sigma)}(u),
      \qquad \text{for any closed } F\subset H .
    \end{align*}
  \end{thm}
  
  \begin{proof}
    Recall that $K_{M}(\s)=\{E_{\Ac(\sigma)}\leq M\}$ is the level set, which is compact by $(H3)$ in Assumption \ref{a.050201} and the lower semicontinuity of $E_{\Ac(\sigma)}$. Therefore, by \cite[Proposition 12.4]{da2014stochastic}, it is sufficient to show that for any positive $\d,\d'$ and $M$, there exists $\eps_0>0$ such that for all $\eps\in(0,\eps_0]$,
    \begin{align}\label{e.070212}
      \mu_{\sigma}^{\eps}
      \left\{u\in H:
      d\left(u,K_{M}(\s)\right)\geq\d
      \right\}
      \leq \exp\left(-(M-\d')/\eps\right).
    \end{align}
    If $M\leq\d'$, then \eqref{e.070212} is trivial. Hence we assume below that $M>\d'$. The proof is divided into the following five steps. 

    {\it Step 1: Preparation. }
    By $(H4)$ in Assumption \ref{a.050201}, there is a radius function $\eta=\eta(\sigma) \in(0, \delta / 2)$ such that for every $t>0$, and $m$-a.e. $\sigma\in\Sigma$, 
    \begin{align}\label{e.070208}
      \left\{u(t): u(0) \in \mathcal{A}_{\eta(\beta_{-t} \sigma)}\left(\beta_{-t} \sigma\right), I_{u(0), t}^{\beta_{-t}\sigma}(u) \leq M-\frac{\delta^{\prime}}{4}\right\} \subset \mathcal{N}_{\delta / 2}\left(K_M(\sigma)\right).
    \end{align}
    Let $M_1, L>M$. For the given $\eta, L$, applying Lemma \ref{l.070101} to choose corresponding $R,T,c>0$ such that $\beta_T$ is ergodic and $m(\mathcal{G})>0$, and for any fixed $n_*>M_1/c$ there is $n\in\mathbb N$ such that $\beta_{-n T} \sigma \in \mathcal{H}_{R, L}$ and $\left|\mathcal{J}_n\right| \geq n_*$ for almost every $\sigma\in\Sigma$. For the rest of  the proof we fix such a symbol $\sigma$ and the integer $n$, put $t=nT$ and denote $\beta_T=\Theta$.

    Define the selected block path set 
    \begin{align*}
      \mathcal{Z}_{\mathcal{J}_n}:=\left\{u \in C([0, nT] ; H): u(0) \in B_R, u\left(s_k\right) \in B_R, u\left(e_k\right) \in \mathcal{A}_{\eta\left(\Theta^{-k} \sigma\right)}^c\left(\Theta^{-k} \sigma\right), \quad k \in \mathcal{J}_n\right\}
    \end{align*}
    and also the enlarged set 
    \begin{align*}
      \widetilde{\mathcal{Z}}_{\mathcal{J}_n}:=\left\{u \in C([0, n T] ; H): u(0) \in B_{R+1}, u\left(s_k\right) \in B_{R+1}, u\left(e_k\right) \in \mathcal{A}_{\frac{1}{2} \eta\left(\Theta^{-k} \sigma\right)}^c\left(\Theta^{-k} \sigma\right), \quad k \in \mathcal{J}_n\right\},
    \end{align*}
    to deal with the separation from the low action level set. Here \begin{align*}
      s_k:=(n-k-1) T, \quad e_k:=(n-k) T, \text{ for } k \in \mathcal{J}_n
    \end{align*}
    are the starting and ending times of the selected block. 

    {\it Step 2: Decomposition and estimate of the tail $I_1$.}
    By the invariance of $\mu_{\sigma}^{\eps}$, namely
    $\mu_{\beta_{-t}\sigma}^{\eps}P_{0,t,\beta_{-t}\sigma}^{\eps}=\mu_{\sigma}^{\eps}$, we have for $t=nT$, 
    \begin{align}\label{e.072707}
      \begin{split}
        \mu_{\sigma}^{\eps}
        \left\{u\in H:
        d\left(u,K_{M}(\s)\right)\geq\d
        \right\} 
        &=\mu_{\sigma}^{\eps}
        \left\{u\in H:u\notin\Nc_{\d}(K_{M}(\s))\right\} \\
        &=\mu_{\beta_{-t}\sigma}^{\eps}P_{0,t,\beta_{-t}\sigma}^{\eps}
        \left\{u\in H:u\notin\Nc_{\d}(K_{M}(\s))\right\} \\
        &=\int_H
        \mathbf{P}\left\{\Phi_{0,t,\beta_{-t}\sigma}^{\eps}v
        \notin\Nc_{\d}(K_{M}(\s))\right\}
        \mu_{\beta_{-t}\sigma}^{\eps}(dv)\\
        &\leq I_1+I_2+I_3, 
      \end{split}
    \end{align}
    where 
    \begin{align}\label{e.070207}
      \begin{split}
        I_3&=\int_{B_R} \mathbf{P}\left\{\Phi_{0, nT, \Theta^{-n} \sigma}^{\varepsilon} v \notin \mathcal{N}_\delta\left(K_M(\sigma)\right) ; \Phi_{0, \cdot, \Theta^{-n} \sigma}^{\varepsilon} v \notin \mathcal{Z}_{\mathcal{J}_n}\right\} \mu_{\Theta^{-n} \sigma}^{\varepsilon}(d v)\\
        I_2&=\int_{B_R} \mathbf{P}\left\{\Phi_{0, nT, \Theta^{-n} \sigma}^{\varepsilon} v \notin \mathcal{N}_\delta\left(K_M(\sigma)\right) ; \Phi_{0, \cdot, \Theta^{-n} \sigma}^{\varepsilon} v \in \mathcal{Z}_{\mathcal{J}_n}\right\} \mu_{\Theta^{-n} \sigma}^{\varepsilon}(d v),
      \end{split}
    \end{align}
    and 
    \begin{align}\label{e.070309}
      I_1 = \int_{B_R^c} \mathbf{P}\left\{\Phi_{0, nT, \Theta^{-n} \sigma}^{\varepsilon} v \notin \mathcal{N}_\delta\left(K_M(\sigma)\right) \right\} \mu_{\Theta^{-n} \sigma}^{\varepsilon}(d v).
    \end{align}
    As $\Theta^{-n} \sigma \in \mathcal{H}_{R, L}$, the definition of $\mathcal{H}_{R, L}$ implies that, for all sufficiently small $\varepsilon>0$,
    \begin{align*}
      I_1\leq \mu_{\Theta^{-n} \sigma}^{\varepsilon}(B_R^c)\leq \exp\left(-\frac{L}{\varepsilon}\right). 
    \end{align*}

    {\it Step 3: Estimate of $I_2$.} We first claim that 
    \begin{align}\label{e.070205}
      \inf \left\{I_{u(0), nT}^{\Theta^{-n}\sigma}(u): u \in \widetilde{\mathcal{Z}}_{\mathcal{J}_n}\right\}>M_1. 
    \end{align}
    Indeed, let $\kappa>0, u \in \widetilde{\mathcal{Z}}_{\mathcal{J}_n}$ and suppose $I_{u(0), nT}^{\Theta^{-n} \sigma}(u)<\infty$. Then there exists an admissible control $\varphi \in L^2(0, nT ; H)$ such that $u(r)=\Phi_{0, r, \Theta^{-n} \sigma}^{\varphi} u(0)$ for $ 0 \leq r \leq n T$,  
    and
    \begin{align*}
    I_{u(0), nT}^{\Theta^{-n} \sigma}(u)\geq \frac{1}{2} \int_0^{nT}\|\varphi(r)\|^2 d r -\kappa.
    \end{align*}
    Since the selected intervals $\left[s_k, e_k\right], k \in \mathcal{J}_n$ are disjoint, one has 
    \begin{align*}
      I_{u(0), nT}^{\Theta^{-n} \sigma}(u) \geq \sum_{k \in \mathcal{J}_n} \frac{1}{2} \int_{s_k}^{e_k}\|\varphi(r)\|^2 d r -\kappa.
    \end{align*}
    By the cocycle property of the controlled flow, the restriction $  w_k(r):=u\left(s_k+r\right)$ for $ 0 \leq r \leq T$ is a controlled path on $[0, T]$ with time symbol $\Theta^{-(k+1)} \sigma$ and with control $\varphi\left(s_k+\cdot\right)$. Moreover, by the definition of $\widetilde{\mathcal{Z}}_{\mathcal{J}_n}$,
    \begin{align*}
    w_k(0)=u\left(s_k\right) \in B_{R+1}, \quad w_k(T)=u\left(e_k\right) \in \mathcal{A}_{\frac{1}{2} \eta\left(\Theta^{-k} \sigma\right)}^c\left(\Theta^{-k} \sigma\right).
    \end{align*}
    Therefore, by definition \eqref{e.070204}, 
    \[\frac{1}{2} \int_{s_k}^{e_k}\|\varphi(r)\|^2 d r \geq a_{R+1, \eta / 2, T}\left(\Theta^{-k} \sigma\right).\]
    Since $k \in \mathcal{J}_n$ implies $\Theta^{-k} \sigma \in \mathcal{G}$, the definition of $\mathcal{G}$ gives $a_{R+1, \eta / 2, T}\left(\Theta^{-k} \sigma\right) \geq c$. Consequently, 
    \begin{align*}
      I_{u(0), nT}^{\Theta^{-n} \sigma}(u) \geq \sum_{k \in \mathcal{J}_n} a_{R+1, \eta / 2, T}\left(\Theta^{-k} \sigma\right)-\kappa \geq c\left|\mathcal{J}_n\right| -\kappa.
    \end{align*}
    By letting $\kappa\to0$ and noting that $\left|\mathcal{J}_n\right|\geq n_*$ and $cn_*>M_1$, we obtain the desired claim \eqref{e.070205}. 

    Next we claim that by setting 
    \begin{align*}
      \rho_\sigma:=\frac{1}{2} \min \left\{1, \min _{k \in \mathcal{J}_n} \eta\left(\Theta^{-k} \sigma\right)\right\}>0, 
    \end{align*}
    for every $v \in B_R$, one has  
    \begin{align}\label{e.070206}
      \mathcal{Z}_{\mathcal{J}_n} \cap\{u(0)=v\} \subset C([0, nT] ; H) \backslash \mathcal{N}_{\rho_\sigma}\left(I_{v, nT}^{\Theta^{-n} \sigma}, M_1\right)
    \end{align}
    where, recall as in Definition \ref{d.070201},  $\mathcal{N}_{\rho_\sigma}\left(I_{v, nT}^{\Theta^{-n} \sigma}, M_1\right)$ denotes the $\rho_\sigma$-neighborhood in $C([0, nT] ; H)$ of the level set
    \begin{align*}
    \left\{\xi \in C([0, nT] ; H): \xi(0)=v, I_{v, nT}^{\Theta^{-n} \sigma}(\xi) \leq M_1\right\}.
    \end{align*}
    Indeed, if \eqref{e.070206} is not true, then for some $v \in B_R$ there exist $u \in \mathcal{Z}_{\mathcal{J}_n}$ and $\xi \in C([0, nT] ; H)$ such that
    \[u(0)=\xi(0)=v, \quad I_{v, nT}^{\Theta^{-n} \sigma}(\xi) \leq M_1, \quad\|u-\xi\|_{C([0, nT] ; H)}<\rho_\sigma .\]
    Since $v \in B_R \subset B_{R+1}$, we have $\xi(0) \in B_{R+1}$. For each $k \in \mathcal{J}_n$, the definition of $\mathcal{Z}_{\mathcal{J}_n}$ gives $u\left(s_k\right) \in B_R$. In view of $\rho_\sigma \leq 1 / 2$, one has 
    \begin{align*}
      \left\|\xi\left(s_k\right)\right\| \leq\left\|u\left(s_k\right)\right\|+\rho_\sigma<R+1 .
    \end{align*}
    Furthermore,
    \begin{align*}
      \begin{aligned}
        \operatorname{dist}\left(\xi\left(e_k\right), \mathcal{A}\left(\Theta^{-k} \sigma\right)\right) & \geq \operatorname{dist}\left(u\left(e_k\right), \mathcal{A}\left(\Theta^{-k} \sigma\right)\right)-\rho_\sigma \\
        & >\eta\left(\Theta^{-k} \sigma\right)-\rho_\sigma \geq \frac{1}{2} \eta\left(\Theta^{-k} \sigma\right).
      \end{aligned}
    \end{align*}
    Thus $\xi \in \widetilde{\mathcal{Z}}_{\mathcal{J}_n}$, contradicting \eqref{e.070205} as $I_{v, nT}^{\Theta^{-n} \sigma}(\xi) \leq M_1$. Therefore, the claim \eqref{e.070206} is true. 

    Now choose $\gamma \in\left(0, M_1-M\right)$. Then \eqref{e.070207}, \eqref{e.070206}  and  the upper bound part of the uniform trajectory LDP on the finite interval $[0, nT]$ with fixed time symbol $\Theta^{-n} \sigma$ gives 
    \begin{align}\label{e.070310}
      \begin{aligned}
        I_2 \leq \sup _{v \in B_R} \mathbf{P}\left\{\Phi_{0, \cdot, \Theta^{-n} \sigma}^{\varepsilon} v \in \mathcal{Z}_{\mathcal{J}_n}\right\}& \leq \sup _{v \in B_R} \mathbf{P}\left\{\Phi_{0, \cdot, \Theta^{-n}\sigma}^{\varepsilon} v \notin \mathcal{N}_{\rho_\sigma}\left(I_{v, nT}^{\Theta^{-n} \sigma}, M_1\right)\right\} \\
       & \leq \exp \left(-\frac{M_1-\gamma}{\varepsilon}\right) \leq \exp \left(-\frac{M}{\varepsilon}\right)
     \end{aligned}
    \end{align}
    for all sufficiently small $\varepsilon>0$.

    {\it Step 4: Estimate of $I_3$.} We first decompose $I_3$ in \eqref{e.070207} into two parts. If a path starts from $B_R$, ends outside $\mathcal{N}_\delta\left(K_M(\sigma)\right)$, and does not belong to $\mathcal{Z}_{\mathcal{J}_n}$, then for at least one selected index $k \in \mathcal{J}_n$ either the starting point of the selected block is outside $B_R$, or the endpoint of the selected block lies inside the $\eta$-neighborhood of the corresponding pullback attractor. Hence
    \begin{align*}
      I_3 \leq I_{31}+I_{32},
    \end{align*}
    where
    \begin{align*}
      I_{32}:=\sum_{k \in \mathcal{J}_n} \int_{B_R} \mathbf{P}\left\{\Phi_{0, s_k, \Theta^{-n} \sigma}^{\varepsilon} v \in B_R^c\right\} \mu_{\Theta^{-n} \sigma}^{\varepsilon}(d v),
    \end{align*}
    and
    \begin{align*}
      \begin{aligned}
        I_{31}:=\sum_{k \in \mathcal{J}_n} \int_{B_R} \mathbf{P}\left\{ \Phi_{0, nT, \Theta^{-n} \sigma}^{\varepsilon} v \notin\mathcal{N}_\delta\left(K_M(\sigma)\right); \Phi_{0, e_k, \Theta^{-n} \sigma}^{\varepsilon} v \in \mathcal{A}_{\eta\left(\Theta^{-k} \sigma\right)}\left(\Theta^{-k} \sigma\right)\right\} \mu_{\Theta^{-n} \sigma}^{\varepsilon}(d v)
     \end{aligned}
    \end{align*}
    We first estimate $I_{32}$. For $k \in \mathcal{J}_n$, stationarity gives
    \begin{align*}
      \begin{aligned}
        \int_H \mathbf{P}\left\{\Phi_{0, s_k, \Theta^{-n} \sigma}^{\varepsilon} v \in B_R^c\right\} \mu_{\Theta^{-n} \sigma}^{\varepsilon}(d v)  =\mu_{\Theta^{-n} \sigma}^{\varepsilon} P_{0, s_k, \Theta^{-n} \sigma}^{\varepsilon}\left(B_R^c\right) =\mu_{\Theta^{-(k+1)} \sigma}^{\varepsilon}\left(B_R^c\right)
      \end{aligned}
    \end{align*}
    because $s_k=(n-k-1) T$. Since $k \in \mathcal{J}_n$ implies $\Theta^{-k} \sigma \in \mathcal{G} \subset \Theta \mathcal{H}_{R, L}$, we have $\Theta^{-(k+1)} \sigma \in \mathcal{H}_{R, L}$. 
    Therefore, for each fixed $k \in \mathcal{J}_n$ and all sufficiently small $\varepsilon>0$,
    \begin{align*}
      \mu_{\Theta^{-(k+1) \sigma}}^{\varepsilon}\left(B_R^c\right) \leq \exp \left(-\frac{L}{\varepsilon}\right)
    \end{align*}
    Since $\mathcal{J}_n$ is finite, we may take the minimum of the corresponding smallness thresholds in $\varepsilon$. Thus
    \begin{align}\label{e.070311}
      I_{32} \leq\left|\mathcal{J}_n\right| \exp \left(-\frac{L}{\varepsilon}\right)
    \end{align}
    for all sufficiently small $\varepsilon>0$.

    We next estimate $I_{31}$. Fix $k \in \mathcal{J}_n$. After time $e_k=(n-k) T$, the remaining time is $k T$ and the time symbol is $\Theta^{-k} \sigma$. By the Markov property,
    \begin{align*}
      I_{31} \leq \sum_{k \in \mathcal{J}_n} \sup _{v \in \mathcal{A}_{\eta\left(\Theta^{-k} \sigma\right)}\left(\Theta^{-k} \sigma\right)} \mathbf{P}\left\{\Phi_{0, k T, \Theta^{-k} \sigma}^{\varepsilon}v \notin \mathcal{N}_\delta\left(K_M(\sigma)\right)\right\}. 
    \end{align*}
    For $k=0$, this event is empty. Indeed, $\eta(\sigma)<\delta / 2$ from {\it Step 1} and $\mathcal{A}(\sigma) \subset K_M(\sigma)$, so
    \begin{align*}
    \mathcal{A}_{\eta(\sigma)}(\sigma) \subset \mathcal{N}_\delta\left(K_M(\sigma)\right). 
    \end{align*}
    Now let $k \geq 1$. By the tracking property \eqref{e.070208}, for every $v \in \mathcal{A}_{\eta\left(\Theta^{-k} \sigma\right)}\left(\Theta^{-k} \sigma\right)$, one has 
    \begin{align*}
      \left\{u(k T): u(0)=v, I_{v, k T}^{\Theta^{-k}\sigma} (u) \leq M-\frac{\delta^{\prime}}{4}\right\} \subset \mathcal{N}_{\delta / 2}\left(K_M(\sigma)\right).
    \end{align*}
    Consequently,
    \begin{align*}
      \left\{\Phi_{0, kT, \Theta^{-k} \sigma}^{\varepsilon}v\notin \mathcal{N}_\delta\left(K_M(\sigma)\right)\right\}\subset\left\{\Phi_{0, \cdot, \Theta^{-k} \sigma}^{\varepsilon} v \notin \mathcal{N}_{\delta / 2}\left(I_{v, k T}^{\Theta^{-k} \sigma}, M-\frac{\delta^{\prime}}{4}\right)\right\}. 
    \end{align*}
    Since there are only finitely many selected indices, and each set $\mathcal{A}_{\eta\left(\Theta^{-k} \sigma\right)}\left(\Theta^{-k} \sigma\right)$ is bounded, the uniform trajectory LDP may be applied to each selected symbol and the minimum of the corresponding $\varepsilon_0$ may be taken. Using the uniform trajectory LDP upper bound with action level $M-\delta^{\prime} / 4$ and error $\delta^{\prime} / 4$, we obtain
    \begin{align}\label{e.070312}
      I_{31} \leq\left|\mathcal{J}_n\right| \exp \left(-\frac{M-\delta^{\prime} / 2}{\varepsilon}\right)
    \end{align}
    for all sufficiently small $\varepsilon>0$.
    
    {\it Step 5: Collecting all estimates together.} By \eqref{e.072707}, \eqref{e.070309}, \eqref{e.070310}, \eqref{e.070311} and \eqref{e.070312}, we have for all sufficiently small $\eps>0$, 
    \begin{align*}
      \begin{aligned}
        &\mu_\sigma^{\varepsilon}\left\{u \in H: \operatorname{dist}\left(u, K_M(\sigma)\right) \geq \delta\right\} \\
        & \quad \leq \exp \left(-\frac{L}{\varepsilon}\right)+\exp \left(-\frac{M}{\varepsilon}\right)+\left|\mathcal{J}_n\right| \exp \left(-\frac{L}{\varepsilon}\right)+\left|\mathcal{J}_n\right| \exp \left(-\frac{M-\delta^{\prime} / 2}{\varepsilon}\right). 
      \end{aligned}
    \end{align*}
    Because $L>M$ and $\left|\mathcal{J}_n\right|<\infty$ is fixed after $\sigma$ is fixed, all finite prefactors can be absorbed into the exponential rate. Hence
    \begin{align*}
      \mu_\sigma^{\varepsilon}\left\{u \in H: \operatorname{dist}\left(u, K_M(\sigma)\right) \geq \delta\right\} \leq \exp \left(-\frac{M-\delta^{\prime}}{\varepsilon}\right)
    \end{align*}
    for all sufficiently small $\varepsilon>0$. This proves \eqref{e.070212} and completes the proof of Theorem \ref{t.070201} on the upper bound. 
  \end{proof}
  
  \subsection{Lower bound}

    We now prove the large deviation lower bound.  In contrast with the upper bound, the lower bound requires a controllable deterministic starting point in the remote past.  For this reason we restrict ourselves in this abstract criterion to the case where the pullback attractor of the limiting system is a single point.  Starting from this deterministic state, the definition of the quasipotential provides a controlled path reaching any prescribed neighborhood of the target point. Stationarity and the trajectory lower bound then transfer this finite time controlled estimate to the stationary measures.

  \begin{prop}\label{p.050601}
    Assume $(H1)$, $(H2)$ and $(H6)$ from Assumption \ref{a.050201}, and suppose that
    $\mathcal A(\sigma)=\{a(\sigma)\}$ for some measurable $a:\Sigma\to H$. Then for $m$-a.e. $\sigma\in\Sigma$ we have 
    \begin{align*}
      \liminf_{\eps\to0}\eps\ln\mu_{\sigma}^{\eps}(G)
      \geq -\inf_{u\in G}E_{\Ac(\sigma)}(u),
      \qquad \text{for any open } G\subset H.
    \end{align*}
  \end{prop}
    We first extract a consequence of weak exponential tightness which will be used to choose good past environmental times.  The point is that the quenched exponential tail bound implies that, for sufficiently large $R$, the set of environmental symbols for which $\mu_\sigma^\varepsilon(B_R)$ carries a uniformly positive amount of mass has a positive \(m\)-measure.  By ergodicity, typical backward orbits visit this set infinitely often.
  \begin{lem}\label{l.050601}
    If
    \[\lim _{R \rightarrow \infty} \limsup _{\varepsilon \rightarrow 0} \varepsilon \log \mu_\sigma^{\varepsilon}\left(B_R^c\right)=-\infty \quad m \text {-a.s. }\]
    then one has 
    \[\lim _{R \rightarrow \infty} \int_{\Sigma} \limsup _{\varepsilon \rightarrow 0} \mu_\sigma^{\varepsilon}\left(B_R^c\right) m(d \sigma)=0.\]
  \end{lem}
  \begin{proof}
    Set 
    \begin{align*}
      \begin{gathered}
        p_{\varepsilon, R}(\sigma):=\mu_\sigma^{\varepsilon}\left(B_R^c\right), \quad  A_R(\sigma):=\underset{\varepsilon \rightarrow 0}{\lim \sup }\, p_{\varepsilon, R}(\sigma), \quad 
        L_R(\sigma):=\underset{\varepsilon \rightarrow 0}{\lim \sup }\, \varepsilon \log p_{\varepsilon, R}(\sigma) .
        \end{gathered}
    \end{align*}
    If $\lim _{R \rightarrow \infty} L_R(\sigma)=-\infty$, then there is $R_0(\sigma)$ such that for $R \geq R_0(\sigma)$, one has $L_R(\sigma)<-1$. Then there is $\varepsilon_0>0$, such that for $0<\varepsilon<\varepsilon_0$, one has $\varepsilon \log p_{\varepsilon, R}(\sigma)<-1/2$, 
    implying 
    \[p_{\varepsilon, R}(\sigma) \leq \exp \left(-\frac{1}{2 \varepsilon}\right).\]
    Hence, $\limsup _{\varepsilon \rightarrow 0} p_{\varepsilon, R}(\sigma)=0$, showing that $A_R(\sigma)=0$ for sufficiently large $R$. Therefore, as $R\rightarrow \infty$, we have $A_R(\sigma) \rightarrow 0$ almost surely. 
    Since $0 \leq A_R(\sigma) \leq 1$, by dominated convergence theorem, one has the desired. 
    \[\lim _{R \rightarrow \infty} \int_{\Sigma} A_R(\sigma) m(d \sigma)=0.\]
  \end{proof}
  \begin{proof}[Proof of Proposition \ref{p.050601}]
    It is sufficient to show that for any $u\in H$ and $\d,\d'>0$, there exists $\eps_0>0$ such that
  \begin{align*}
    \mu_{\sigma}^{\eps}(B_{\d}(u))
    \geq \exp\left(-(E_{\Ac(\sigma)}(u)+\d')/\eps\right),
    \qquad \forall \eps\in(0,\eps_0].
  \end{align*}
  By the definition of \eqref{e.073003}, there exist $s>0$, and a control $\varphi$ such that
  \begin{align}\label{e.050302}
    \Phi_{0,s,\beta_{-s}\sigma}^{\varphi}a(\b_{-s}\s)
    \in B_{\d/4}(u),  \text{ and } J_s^{\beta_{-s}\sigma}(\varphi)\leq E_{\Ac(\sigma)}(u)+\d'/3.
  \end{align}
  Since 
  \[\lim _{R \rightarrow \infty} \int_{\Sigma} \limsup _{\varepsilon \rightarrow 0} \mu_{\tau}^{\varepsilon}\left(B_R^c\right) m(d \tau)=0,\]
  by Lemma \ref{l.050601} and Assumption \ref{a.050201}, we can choose $R$ large such that 
  \[\Sigma_{R}:=\left\{\tau: \limsup _{\varepsilon \rightarrow 0} \mu_{\tau}^{\varepsilon}\left(B_R^c\right)<1/2\right\}\]
  satisfies $m(\Sigma_{R})>0$. Then by Birkhoff's ergodic theorem, there is a sequence $t_n\to\infty$ and $\varepsilon(n)>0$ such that 
  \begin{align}\label{e.050304}
    \mu_{\beta_{-t_n}\sigma}^{\eps}(B_R^c)<1/2, \quad \varepsilon\in(0,\varepsilon(n)]
  \end{align}
  By \eqref{e.050302} and the continuous dependence on the initial condition, there is $r>0$ such that $\Phi_{0,s,\beta_{-s}\sigma}^{\varphi}v
  \in B_{\d/2}(u)$ for any $v\in B_r(a(\b_{-s}\s))$. On the other hand, the uniform pullback attraction implies that 
  \begin{align}\label{e.050303}
    \sup _{\|v\| \leq R}\left\|\Phi_{0, t_n-s, \beta_{-t_n}\sigma} v-a\left(\beta_{-s} \sigma\right)\right\|<r
  \end{align}
  for sufficiently large $n$. Fix such an $n$ from now on. If we set $\widetilde{\varphi}=0$ on $[0,t_n-s]$ and $\widetilde{\varphi}(\cdot)=\varphi(\cdot+s-t_n)$ on $[t_n-s,t_n]$, then for any $v\in B_R(0)$, the corresponding controlled path $\Phi_{0,t,\b_{-t_n}\s}^{\widetilde{\varphi}}v, t\in [0,t_n]$ satisfies $\Phi_{0,t_n-s,\b_{-t_n}\s}^{\widetilde{\varphi}}v\in B_r\left(a\left(\beta_{-s} \sigma\right)\right)$ by \eqref{e.050303}, and therefore 
  \[\Phi_{0,t_n,\b_{-t_n}\s}^{\widetilde{\varphi}}v = \Phi_{t_n-s,t_n,\b_{-t_n}\s}^{\widetilde{\varphi}}\Phi_{0,t_n-s,\b_{-t_n}\s}^{\widetilde{\varphi}}v=\Phi_{0,s,\b_{-s}\s}^{\varphi}\Phi_{0,t_n-s,\b_{-t_n}\s}^{\widetilde{\varphi}}v\in B_{\d/2}(u).\]
  Note that 
  \[J_{t_n}^{\beta_{-t_n}\s}(\widetilde{\varphi})= J_s^{\beta_{-s}\sigma}(\varphi)\leq E_{\Ac(\sigma)}(u)+\d'/3.\]
  It follows from the trajectory LDP lower bound that 
  \begin{align*}
    \inf _{v \in B_R(0)} P_{0, t_n, \b_{-t_n}\s}^{\varepsilon}\left(v, B_\delta\left(u\right)\right)=\inf _{v \in B_R(0)} \mathbf{P}\left\{\Phi_{0, t_n, \beta_{-t_n}\sigma}^{\varepsilon} v \in B_\delta(u)\right\} \geq \exp \left(-\frac{E_{\mathcal{A}(\sigma)}\left(u\right)+2 \delta^{\prime} / 3}{\varepsilon}\right)
  \end{align*}
  for small $\varepsilon$. Consequently, by the invariance of the stationary measure, one has the desired 
  \begin{align*}
    \begin{aligned}
      \mu_\sigma^{\varepsilon}\left(B_\delta\left(u\right)\right) & =\mu_{\b_{-t_n}\s}^{\varepsilon} P_{0, t_n, \b_{-t_n}\s}^{\varepsilon}\left(B_\delta\left(u\right)\right) \\
      & \geq \mu_{\b_{-t_n}\s}^{\varepsilon}\left(B_R(0)\right) \inf _{v \in B_R(0)} P_{0, t_n, \b_{-t_n}\s}^{\varepsilon}\left(v, B_\delta\left(u\right)\right)\geq \exp \left(-\frac{E_{\mathcal{A}(\sigma)}\left(u\right)+\delta^{\prime} }{\varepsilon}\right)
      \end{aligned}
  \end{align*}
  by \eqref{e.050304} and taking $\varepsilon$ small so that $\mu_{\b_{-t_n}\s}^{\varepsilon}\left(B_R\right)>1/2>e^{-\delta^{\prime} /(3 \varepsilon)}$. 
  \end{proof}
  
  The preceding proof does not rely on the weak convergence of $\mu_\sigma^\varepsilon$ as $\varepsilon\to0$.  Nevertheless, in the random-point attractor case this convergence follows from the same ingredients. Indeed, the pullback attraction brings bounded deterministic trajectories close to $a(\sigma)$,
  which along with weak exponential tightness implies the weak convergence of the stationary measures to the deterministic stationary family.

  \begin{prop}
    Assuming $(H1),(H2),(H5),(H6)$ from Assumption \ref{a.050201} and $\mathcal{A}(\sigma)=\{a(\sigma)\}$, then for $m$-a.e. $\sigma\in\Sigma$, the family $\mu_{\sigma}^{\eps}$ converges weakly to 
    $\mu_\sigma=\delta_{a(\sigma)}$
    as $\eps\to0$.
  \end{prop}
  \begin{proof}
     It is enough to prove convergence against bounded Lipschitz functions. Let $\psi\in\Lip_b(H)$ and $\gamma>0$. Set $C_\psi:=1+2\|\psi\|_\infty+\Lip(\psi)$ and $\rho_\gamma:=\frac{\gamma}{16C_\psi}.$ By the weak exponential tightness \eqref{e.072703}, choose $R>0$ such that
    \[
      \Sigma_{R,\gamma}:=\left\{\tau\in\Sigma:\limsup_{\eps\to0}\mu_\tau^\eps(B_R^c)<\rho_\gamma\right\}
    \]
    has positive $m$-measure. By the ergodicity of $\beta_t$, for $m$-a.e. $\sigma\in\Sigma$, there exists a sequence $t_n\to\infty$ such that $\beta_{-t_n}\sigma\in\Sigma_{R,\gamma}$. Fix such a $\sigma$. Since $\Ac(\sigma)=\{a(\sigma)\}$ is the pullback attractor, we have
    \[
      \lim_{t\to\infty}\sup_{u\in B_R}\left\|\Phi_{0,t,\beta_{-t}\sigma}u-a(\sigma)\right\|=0.
    \]
    Hence, for some sufficiently large $n$, with $t:=t_n$ and $\tau:=\beta_{-t}\sigma$, we have $\tau\in\Sigma_{R,\gamma}$ and
    \begin{align}\label{e.070102}
      \sup_{u\in B_R}\left\|\Phi_{0,t,\tau}u-a(\sigma)\right\|<\rho_\gamma .
    \end{align}
    Since $\tau\in\Sigma_{R,\gamma}$, there exists $\eps_1>0$ such that
    \begin{align}\label{e.070103}
      \mu_\tau^\eps(B_R^c)<\rho_\gamma,\qquad 0<\eps\leq\eps_1 .
    \end{align}

    We next derive the terminal time convergence of the noisy trajectories. Set $M_\gamma:=\frac{\rho_\gamma^2}{4C_{R,t}(\sigma)}.$ By \eqref{e.070304}, if $u_0\in B_R$ and $I_{u_0,t}^{\tau}(u)\leq M_\gamma$, then
    \[\left\|u(t)-\Phi_{0,t,\tau}u_0\right\|\leq\frac{\rho_\gamma}{2}.\]
    Therefore, for every $u_0\in B_R$,
    \[\left\{w\in C([0,t];H):\left\|w(t)-\Phi_{0,t,\tau}u_0\right\|>\rho_\gamma\right\}\subset C([0,t];H)\setminus\Nc_{\rho_\gamma/2}\left(I_{u_0,t}^{\tau},M_\gamma\right).\]
    Applying the upper bound part of the trajectory uniform LDP on $C([0,t];H)$, with time symbol $\tau$, radius $\rho_\gamma/2$, and level $M_\gamma$, we get
    \begin{align}\label{e.070104}
      \sup_{u_0\in B_R}\mathbf P\left\{\left\|\Phi_{0,t,\tau}^\eps u_0-\Phi_{0,t,\tau}u_0\right\|>\rho_\gamma\right\}\leq \exp\left(-\frac{M_\gamma}{2\eps}\right)<\rho_\gamma
    \end{align}
    for all sufficiently small $\eps>0$.

    By the invariance of $\mu_\sigma^\eps$,
    \[
      \int_H\psi(u)\mu_\sigma^\eps(du)=\int_H P_{0,t,\tau}^\eps\psi(u)\mu_\tau^\eps(du).
    \]
    Hence, using $\mu_\sigma=\delta_{a(\sigma)}$, \eqref{e.070102}, \eqref{e.070103}, and \eqref{e.070104}, we obtain
    \begin{align*}
      &\left|\int_H\psi(u)\mu_\sigma^\eps(du)-\int_H\psi(u)\mu_\sigma(du)\right|\\
      &\leq 2\|\psi\|_\infty\mu_\tau^\eps(B_R^c)+\int_{B_R}\mathbf E\left|\psi\left(\Phi_{0,t,\tau}^\eps u\right)-\psi(a(\sigma))\right|\mu_\tau^\eps(du)\\
      &\leq C_\psi\mu_\tau^\eps(B_R^c)+C_\psi\rho_\gamma+C_\psi\sup_{v\in B_R}\mathbf P\left\{\left\|\Phi_{0,t,\tau}^\eps v-\Phi_{0,t,\tau}v\right\|>\rho_\gamma\right\}\leq 4C_\psi\rho_\gamma<\gamma,
    \end{align*}
    for all sufficiently small $\eps>0$. Since $\gamma>0$ is arbitrary, it follows that
    \[\lim_{\eps\to0}\int_H\psi(u)\mu_\sigma^\eps(du)=\psi(a(\sigma))=\int_H\psi(u)\mu_\sigma(du),\]
    implying that $\mu_\sigma^\eps\Rightarrow\mu_\sigma$ for $m$-a.e. $\sigma\in\Sigma$.
  \end{proof}

  We finally give a practical criterion for verifying the abstract condition \eqref{e.072701}.  It shows that the tracking property follows from stability of controlled trajectories with respect to nearby initial data, together with a comparable control cost.  Such an estimate is often the natural energy estimate available for dissipative PDEs.
  
  \begin{lem}\label{l.050602}
       Suppose for any $M>0$, there is a finite measurable $C_M:\Sigma\to \mathbb R_+$ such that for any $t>0$, any controlled path $u
       =\Phi^{\varphi_u}_{0,\cdot, \xi}u_0
       $ satisfying 
       \[
      {\frac 12} \int_0^t
      \|\varphi_u (s)\|^2 ds  \leq M, \quad d(u_0, \mathcal{A}(\xi))<1,\]
      and any $v_0 \in \mathcal{A}(\xi)$, there is a controlled path $v
      =\Phi^{\varphi_v}_{0,\cdot, \xi}v_0
      $ such that 
      \begin{align}\label{e.050605}
        \sup _{0 \leq s \leq t}\|u(s)-v(s)\|^2 \leq C_M(\xi)\left\|u_0-v_0\right\|^2
      \end{align}
      and 
      \begin{align}\label{e.050606}
        \int_0^t\left\|\varphi_u(s)-\varphi_v(s)\right\|^2 d s \leq C_M(\xi)\left\|u_0-v_0\right\|^2. 
      \end{align}
      Then \eqref{e.072701} is true.
  \end{lem}
  \begin{proof}
    Fix any $\delta, \delta^{\prime}, M>0$
    such that $M-\delta^{\prime} \geq 0$.
    Then $M^{\prime}:= M- {\frac 34} \delta^{\prime} >0$
    and by assumptions we find that
    there 
     is a finite measurable $C_{M^{\prime}}:\Sigma\to \mathbb R_+$ such that for any $t>0$, any controlled path $u
       =\Phi^{\varphi_u}_{0,\cdot, \xi}u_0
       $ with
       \[
      {\frac 12} \int_0^t
      \|\varphi_u (s)\|^2 ds  \leq M^{\prime}, \quad d(u_0, \mathcal{A}(\xi))<1,\]
      and any $v_0 \in \mathcal{A}(\xi)$, there is a controlled path $v
      =\Phi^{\varphi_v}_{0,\cdot, \xi}v_0
      $ such that 
      \begin{align}\label{jul10a1}
        \sup _{0 \leq s \leq t}\|u(s)-v(s)\|^2 \leq C_
        {M^{\prime}}(\xi)\left\|u_0-v_0\right\|^2
      \end{align}
      and 
      \begin{align}\label{jul10a2}
        \int_0^t\left\|\varphi_u(s)-\varphi_v(s)\right\|^2 d s 
        \leq C_
        {M^\prime} (\xi)\left\|u_0-v_0\right\|^2. 
      \end{align}

    Choose $a>1$ sufficiently close to $1$ such that $a\left(M-
    {\frac 34}
    \delta^{\prime}\right) \leq M-\frac{\delta^{\prime}}{2}$. Then take 
    \begin{align}\label{jul10a3}
    \eta(\xi)<\min \left\{1, \delta, \frac{\delta}{\sqrt{C_
    {M^\prime}
    (\xi)}},\left(\frac{(a-1) \delta^{\prime}}{2 a C_
    {M^\prime}
    (\xi)}\right)^{1 / 2}\right\}.
    \end{align}
    We now  prove 
\eqref{e.072701} is valid for $\eta$ given by \eqref{jul10a3},
for which we first
show that if $u\in C([0,t], H)$
with $t>0$,
$u(0) \in \mathcal{A}_{\eta(\xi)}(\xi)$ and 
$I^\xi_{u(0),t} (u) \leq M-\delta^\prime$, then
  $u(t) \in \mathcal{N}_\delta\left(K_M\left(\beta_t \xi\right)\right)$.

 For  $I^\xi_{u(0),t} (u) \leq M-\delta^\prime$, we infer that
 there exists $\varphi_u$ such that
 $u=\Phi^{\varphi_u}_{0,\cdot, \xi} u(0)$ with
 ${\frac 12} \int_0^t \|\varphi_u (s)\|^2 ds 
 <   M-{\frac 34}
 \delta^\prime = M^\prime$.
 On the other hand,
  for  $u(0) \in \mathcal{A}_{\eta(\xi)}(\xi)$,  
    there exists $v_0 \in \mathcal{A}(\xi)$ such that
    $\left\|u(0)-v_0\right\|<\eta(\xi)$, 
    and thus by \eqref{jul10a1}-\eqref{jul10a2} we 
    find that there exists
    $v=\Phi_{0,\cdot,\xi}^{\varphi_v} v_0$ such that
    \[\|u(t)-v(t)\| \leq \sqrt{C_
    {M^\prime} (\xi)} \eta(\xi)<\delta.\]
    In addition, by \eqref{jul10a2}-\eqref{jul10a3},  
    \begin{align*}
      \begin{aligned}
        I_{v(0),t}^{\xi}(v)  \leq \frac{1}{2} \int_0^t\left\|\varphi_v\right\|^2 d s 
        & \leq 
        {\frac 12}
        a
        \int_0^t
        \|\varphi_u\|^2 ds
        +\frac{a}{2(a-1)} \int_0^t\left\|\varphi_v-\varphi_u\right\|^2 d s \\
        & \leq a\left( M-
        {\frac 34}
        \delta^{\prime} \right)
        +\frac{a}{2(a-1)}\eta(\xi)^2  C
        _{M^\prime}  (\xi) \\ 
        & < M-{\frac 12} \delta^\prime
        +{\frac 14}\delta^\prime < M, 
        \end{aligned}
    \end{align*}  
    implying that $v(t) \in K_M\left(\beta_t \xi\right)$. Therefore, $u(t) \in \mathcal{N}_\delta\left(K_M\left(\beta_t \xi\right)\right)$. In particular, for $\xi=\beta_{-t} \sigma$, we have $u(t) \in \mathcal{N}_\delta\left(K_M(\sigma)\right)$ as desired. 
  \end{proof}

\section{Applications to Navier--Stokes equations}\label{s.060701}

Let $(\Sigma,m)$ be a standard Borel probability space and $(\beta_t)_{t\in\mathbb R}$ be an invertible ergodic measure-preserving dynamical system on $\Sigma$.  Fix $ \mathbf{d}\in\mathbb N$ and let
\[
   W(t)=(W_1(t),\ldots,W_{ \mathbf{d}}(t))
\]
be a two-sided standard Brownian motion in $\Rb^{ \mathbf{d}}$.  Consider a family of Navier--Stokes equations indexed by $\s\in\Sigma$, $\varepsilon\in(0,1]$, 
\begin{align*}
  \partial_t u+u\cdot\nabla u = \nu\Delta u +\nabla p + f(\beta_t\s,x)+\sqrt{\eps}\sum_{k=1}^{ \mathbf{d}}q_k(\beta_t\s,x)\dot W_k(t), \quad \nabla\cdot u =0, 
  \, u|_{\partial D}=0, 
\end{align*}
on a bounded domain $D$ with smooth boundary. Applying the Leray projection $\Pi$ reduces the equation to an evolution equation 
\begin{align}\label{e.072601}
  \partial_t u+B(u,u) + \nu A u  =  F(\beta_t\s)+\sqrt{\eps}Q(\beta_t\s)\dot W_t, \quad \nabla\cdot u =0, 
  \, u|_{\partial D}=0, 
\end{align}
where $B(u,u)=\Pi(u\cdot\nabla u), \, A = -\Pi\Delta, F=\Pi f$ and $Q=\Pi q$. We consider the phase space as 
\[H : = \{u\in L^2(D;\mathbb R^2): \nabla\cdot u=0 \text{ in } D, \, \langle u, \mathbf{n}\rangle = 0 \text{ on } \partial D\},\]
with norm $\|\cdot\|$, where $\mathbf{n}$ is the outward unit normal to $\partial D$. Let $H^m$ denote the divergence-free subspace of the usual Sobolev space of order $m$, with norm $\|\cdot\|_m$. Under conditions on the external forces that will be given below, the above equation is well-posed in $H$. The corresponding solution flow is denoted by $\Phi_{0,t,\s}^{\eps}$. For $\eps=0$, the deterministic solution flow $\Phi_{0,t,\s}$ is a cocycle over $\Sigma$, 
\begin{align*}
  \Phi_{0,s+t,\s}= \Phi_{0,t,\b_s\s}\circ\Phi_{0,s,\s}, \quad \s\in \Sigma. 
\end{align*}

Assume that $\{e_m\}$ is an orthonormal basis of $H$ consisting of eigenfunctions of the Stokes operator $A$ with eigenvalues $\{\lambda_m\}$ and let $P_N$ be the orthogonal projection onto the span of the first $N$ eigenfunctions.  The noise is finite dimensional: for each $\sigma\in\Sigma$ let
\[q_1(\sigma),\ldots,q_{ \mathbf{d}}(\sigma)\in H^1,\qquad Q(\sigma)a:=\sum_{k=1}^{ \mathbf{d}}a_kq_k(\sigma),\quad a=(a_1,\ldots,a_{ \mathbf{d}})\in\Rb^{ \mathbf{d}}.\]
Thus $Q(\sigma)\in\mathcal L(\Rb^{ \mathbf{d}},H)$ and
\[\|Q(\sigma)\|_{\mathcal L_2(\Rb^{ \mathbf{d}},H^1)}^2=\sum_{k=1}^{ \mathbf{d}}\|q_k(\sigma)\|_1^2.\]
Here $F:\Sigma\to H$ is a deterministic force in the random environment $\Sigma$.  

\begin{thm}\label{t.050601}
  Assume
  \[
     \|F\|^2\in L^1(\Sigma,m),\qquad
     \|Q\|_{\mathcal L_2(\Rb^{ \mathbf{d}},H^1)}^2\in L^1(\Sigma,m),\qquad
     Q_\infty:=\operatorname*{sup}_{\sigma\in\Sigma}
     \|Q(\sigma)\|_{\mathcal L(\Rb^{ \mathbf{d}},H)}<\infty .
  \]
  Suppose that there exist an integer
  $N_0\geq 1$ and a measurable family of linear maps
  $R_0(\sigma):P_{N_0}H\to \Rb^{ \mathbf{d}}$ such that
  \[
      Q(\sigma)R_0(\sigma)P_{N_0}=P_{N_0},\qquad
      \|R_0\|_{\mathcal L(P_{N_0}H,\Rb^{ \mathbf{d}})}^2\in L^1(\Sigma,m),
  \]
  and $N_0$ is chosen large enough to satisfy the determining mode gap
  \begin{align}\label{e.060417}
    \nu\lambda_{N_0+1}>
    C_*\int_\Sigma \left(\|F(\sigma)\|^2+\|Q(\sigma)\|_{\mathcal L_2(\Rb^{ \mathbf{d}},H)}^2\right)\,m(d\sigma),
  \end{align}
  where $C_*>0$ 
  is an embedding constant 
  depending  only on $\nu$, $\lambda_1$, and the domain $D$.

  Then the Navier--Stokes system \eqref{e.072601} has a unique stationary measure $\mu_{\s}^{\varepsilon}$ that attracts transition probabilities exponentially under the Wasserstein metric $\mathbb{W}_{\|\cdot\|\wedge 1}$. Moreover, for $m$-a.e. $\s\in \Sigma$, the family $\{\mu_{\s}^{\varepsilon}\}_{\varepsilon\in(0,1]}$ obeys an upper bound Freidlin--Wentzell LDP with a good rate function, which also gives a lower bound, and hence the full Freidlin--Wentzell LDP, in the case when the random attractor of the limiting system consists of a single random point. 
\end{thm}

\subsection{Ergodicity and exponential mixing}
 
The main result of this subsection is the following 
\begin{thm}
  Under the assumptions of Theorem \ref{t.050601}, the Navier--Stokes system \eqref{e.072601} has a unique stationary measure $\mu_{\s}^{\varepsilon}$ that is exponentially mixing under the Wasserstein metric $\mathbb{W}_{\|\cdot\|\wedge 1}$.
\end{thm}

\begin{proof}
  We only need to verify the assumptions of Theorem \ref{t.042501}. Denoting $X_t^{\s}=\Phi_{0,t,\s}^{\varepsilon}u$ and setting 
  \[U(u) = V(u)=\|u\|^2, \, p(u,v)=\|u-v\|^2, \, S(u)=\|u\|_1^2, \]
  then by Ito's formula and Poincare's inequality, we have 
  \begin{align*}
    \mathbf E_{u}V(X_t^{\s}) + \gamma \mathbf E_{u}\int_0^tV(X_s^{\s})ds\leq V(u) + \int_0^tK(\beta_s\s)ds
  \end{align*}
  where $\gamma=\nu\l_1$ and $K(\xi) = C\left(\|F(\xi)\|^2+\|Q(\xi)\|_{\mathcal L_2(\Rb^{ \mathbf{d}},H)}^2\right), \, \xi\in \Sigma$ with a constant $C>0$ that depends only on $\nu,\l_1$. As $K\in L^1(\Sigma,m)$ and $U,p$ are bounded on $\{V\leq M\}$ and $\{V\leq M\}\times \{V\leq M\}$ for any $M>0$, assumption $(H1)$ is verified. 
  
  Let $Y_t^{\s}$ be the solution to 
  \begin{align*}
    \partial_t {Y_t^{\s}} +B(Y_t^{\s}, Y_t^{\s}) + \nu A  Y_t^{\s} =  F(\beta_t\s)+\frac{\nu\l_{N+1}}{2} P_N(X_t^{\s}-Y_t^{\s} )+\sqrt{\eps}Q(\beta_t\s)\dot W_t, 
  \end{align*}
  with initial value $v\in H$,
  where $N=N_0$ is the fixed integer in the theorem.
   Then from the equation of $X_t^{\s}-Y_t^{\s}$, by Ito's formula and Poincare's inequality, we infer 
  \begin{align*}
    p(X_t^{\s}, Y_t^{\s})\leq p(u,v)\exp\left(-\nu\l_{N+1} t + \kappa\int_0^tS(X_s^{\s})ds\right), \, t\geq0,
  \end{align*}
  where $\kappa>0$ is a constant that depends only on $\nu$ and the domain $D$. This verifies assumption $(H2)$.
  
  Again, by Ito's formula and Poincare's inequality, we have 
  \begin{align*}
    U(X_t^{\s}) + \nu\int_0^t\|X_s^{\s}\|_1^2ds\leq U(X_0^{\s})+\int_0^tK(\b_s\s)ds +M_t^{\s},
  \end{align*}
  with  $dM_{t}^{\s} = 2\sqrt{\varepsilon}\langle X_t^{\s},Q(\beta_t\s) dW_t\rangle$. 
  Note that 
  \begin{align*}
    d\langle M^{\s}\rangle_t=4\varepsilon\|Q^*(\beta_t\sigma)X_t^{\s}\|^2dt\leq 4\lambda_1^{-1}\|Q\|_{L^{\infty}(\Sigma;\mathcal L(\Rb^{ \mathbf{d}},H))}^2\|X_t^{\s}\|_1^2dt.
  \end{align*}
  The determining mode gap imposed on $N_0$ gives the strict inequality required in $(H3)$. 
  
  Now note that the solution $\Phi_{0,t,\sigma}^{\varepsilon}v$ has true law  $P_t^{\sigma}(v,\cdot) = \left(\Phi_{0,t,\sigma}^{\varepsilon}\right)_{*}\mathbf{P}_{\sigma,v}$
  as the pushforward of the Wiener measure by the solution map. Denoting 
  \[h(t)= \frac{\nu\lambda_{N+1}}{2\sqrt{\varepsilon}}R_0(\beta_t\sigma)P_N(X_t^{\s}-Y_t^{\s}), \text{ and } \bar{h}(t) =\int_0^th(s)ds\]
  and $T_{\bar{h}}\omega = \omega+\bar{h}$ on the Brownian path space $C([0,\infty);\Rb^{ \mathbf{d}})$, then  $\mathcal{L}(Y_t^{\s}) = \left(\Phi_{0,t,\sigma}^{\varepsilon}\right)_{*}(T_{\bar{h}})_*\mathbf{P}_{\sigma,v}.$
  Therefore, by Theorem A.5 from \cite{BKS2020}, for any $\delta\in(0,1]$, and 
  \begin{align*}
    M_{\delta}(t,\sigma) = \mathbf{E}\left(\int_0^t\frac{\nu^2\lambda_{N+1}^2}{4\varepsilon}\|R_0(\beta_s\sigma)\|_{\mathcal{L}(P_NH,\Rb^{ \mathbf{d}})}^2\|X_s^{\s}-Y_s^{\s}\|^2ds\right)^{\delta}, 
  \end{align*}
  we have 
  \begin{align*}
    d_{TV}\left(\mathcal{L}(Y_t^{\s}),P_t^{\sigma}(v,\cdot) \right)&\leq d_{TV}\left((T_{\bar{h}})_*\mathbf{P}_{\sigma,v},\mathbf{P}_{\sigma,v} \right)\\
    &\leq 2^{\frac{1-\delta}{1+\delta}}(M_{\delta}(t,\sigma))^{\frac{1}{1+\delta}},
  \end{align*}
  and $d_{TV}\left(\mathcal{L}(Y_t^{\s}),P_t^{\sigma}(v,\cdot) \right)\leq 1-\bar{\varepsilon}_{\delta}(M_{\delta}(t,\sigma))$, where 
  \[\bar{\varepsilon}_{\delta}(x) = \frac16\min\left(\frac18,e^{-(2^{2-\delta}x)^{1/\delta}}\right), \quad x\geq 0. \]
  This verifies assumption $(H4)$. 
\end{proof}

\begin{rem}\label{r.060401}
  From the proof we see that for each $j\in\mathbb N$, the relevant estimates can be made uniform for $\varepsilon\in(1/j,1]$. Therefore upon taking a countable intersection, the full measure subset $\widehat{\Sigma}$ in Theorem \ref{t.060201}, and hence the exceptional set in Theorem \ref{t.042501} can be made independent of $\varepsilon\in (0,1]$. 
\end{rem}

\subsection{The Freidlin--Wentzell LDP}
In this subsection we prove the desired Freidlin--Wentzell LDP in Theorem \ref{t.050601} for the family of unique stationary measures obtained in the previous subsection, by verifying the conditions of Theorem \ref{t.080601}. 

The controlled equation for the symbol $\sigma$ with a control $\varphi\in L_{\text{loc}}^2(\Rb;\Rb^{ \mathbf{d}})$ is 
\begin{align}\label{e.072704}
  \partial_t u+B(u,u) + \nu A u  =  F(\beta_t\s)+ Q(\beta_t\s)\varphi,
\end{align}
whose solution with initial condition $v$ is denoted by $\Phi^{\varphi}_{0,t,\s}v$. The energy for any $u\in C([0,T],H)$ is defined as 
\begin{align}\label{e.072901}
  I_T^{\s}(u)=\inf_{\varphi}J_T^{\s}(\varphi),\qquad
  J_T^{\s}(\varphi)=\frac12\int_0^T\|\varphi(s)\|^2ds,
\end{align}
where the infimum is taken over all $\varphi\in L^2(0,T;\Rb^{ \mathbf{d}})$ such that $u(t)=\Phi^{\varphi}_{0,t,\sigma}u(0)$ for all $t\in[0,T]$. If no such control exists, then $I_T^{\s}(u)=\infty$. 

It is known \cite{garcia2014pullback,carvalho2012attractors} that \eqref{e.072601} when $\varepsilon=0$ has a pullback attractor $\Ac(\s)$ such that for $m\text{-a.e.}$ $\sigma$, we have the invariance $\Phi_{0,t,\s}\Ac(\s) = \Ac(\b_t\s)$ for all $t\geq 0$, 
and for any bounded set $B\subset H$,
\begin{align*}
  \lim_{t\to\infty}d\left(\Phi_{0,t,\b_{-t}\s}B, \Ac(\s)\right)=0. 
\end{align*}
Let $P_{0,t,\s}^{\eps}$ be the Markov operators and $\mu_{\s}^{\eps}$ the unique invariant measure of \eqref{e.072601} so that for $m\text{-a.e.}$ $\sigma$, one has $\mu_{\s}^{\eps}P_{0,t,\s}^{\eps} = \mu_{\b_t\s}^{\eps}$ for all $t\geq 0$, 
where we emphasize that the exceptional set does not depend on $\varepsilon\in(0,1]$ by Remark \ref{r.060401}.  

The finite time uniform Freidlin--Wentzell large deviation principle for solution trajectories of the stochastic Navier--Stokes system follows from standard results; see, for instance, \cite{salins2019uniform,budhiraja2008large}. Related trajectory large deviation principles for stochastic Navier--Stokes equations driven by L\'evy or jump noises were established in \cite{xu2009large,zhai2015large,brzezniak2022well}. Therefore, in order to apply Theorem \ref{t.080601}, it remains only to verify the tracking property, the nontrivial escaping energy property, the precompactness of the level sets, and the weak exponential tightness.

We first verify the tracking property \eqref{e.072701} and nontrivial escaping energy property \eqref{e.070101}. 

\begin{lem}
  Under the conditions of Theorem \ref{t.050601}, the estimates \eqref{e.072701} and \eqref{e.070101} in Assumption \ref{a.050201} are true.
\end{lem}
\begin{proof}
  For \eqref{e.072701} we only need to verify the energy estimate in Lemma \ref{l.050602}.  The point is that the
  feedback below uses the fixed projection $P_{N_0}$ from Theorem \ref{t.050601}; in particular, unlike \cite{mar2},
  the number of controlled modes does not depend on the action level $M$.

  {\it Step 1: verification of Lemma \ref{l.050602}.}
  Fix a time symbol $\xi\in\Sigma$ and $M>0$, $t>0$. Let $u$ be a reference controlled path with an admissible control $\varphi_u$ such that
  \[\frac12\int_0^t\|\varphi_u(s)\|^2\,ds\leq M,\qquad d(u(0),\Ac(\xi))<1.\]
  For any $v_0\in\Ac(\xi)$, set $N=N_0$ and define $v$ by
  \begin{align*}
    \partial_t v+B(v,v)+\nu A v
    =F(\beta_t\xi)+Q(\beta_t\xi)\varphi_u+\lambda P_N(u-v),
    \qquad v(0)=v_0,
  \end{align*}
  where $\lambda>0$ will be chosen below.  Since
  $Q(\sigma)R_0(\sigma)P_N=P_N$, the path $v$ is generated by the admissible control
  \begin{align}\label{e.070302}
      \varphi_v(t)=\varphi_u(t)+\lambda R_0(\beta_t\xi)P_N(u(t)-v(t)).
  \end{align}

  We first estimate the reference path $u$ with a long time slope independent of $M$.  Taking the $H$ inner product of the equation for $u$ with $u$ gives
  \begin{align}\label{e.050401}
      \frac12\frac{d}{dt}\|u\|^2+\nu\|u\|_1^2
      =\langle F(\beta_t\xi),u\rangle+
      \langle Q(\beta_t\xi)\varphi_u,u\rangle .
  \end{align}
  By Poincare's inequality and Young's inequality,
  \[\langle F(\beta_t\xi),u\rangle\leq \frac{\nu}{4}\|u\|_1^2+\frac{1}{\nu\lambda_1}\|F(\beta_t\xi)\|^2, \text{ and } \langle Q(\beta_t\xi)\varphi_u,u\rangle\leq \frac{\nu}{4}\|u\|_1^2+\frac{Q_\infty^2}{\nu\lambda_1}\|\varphi_u\|^2 . \]
  Hence, for every $0\leq \tau\leq t$,
  \begin{align}\label{e.050509}
      \int_0^\tau\|u(s)\|_1^2\,ds
      \leq \frac{1}{\nu}\|u(0)\|^2
      +\frac{2}{\nu^2\lambda_1}\int_0^\tau\|F(\beta_s\xi)\|^2\,ds
      +\frac{4Q_\infty^2}{\nu^2\lambda_1}M .
  \end{align}
  We shall also use the standard $H$-bound obtained from \eqref{e.050401} and Gronwall's inequality:
  \begin{align}\label{e.050506}
    \begin{split}
      \|u(\tau)\|\leq e^{-\nu\lambda_1\tau}\|u(0)\|
      +\int_0^{\tau} e^{-\nu\lambda_1(\tau-s)}\|F(\beta_s\xi)\|\,ds +\int_0^{\tau} e^{-\nu\lambda_1(\tau-s)}
      \|Q(\beta_s\xi)\|_{\mathcal L(\Rb^{ \mathbf{d}},H)}\|\varphi_u(s)\|\,ds .
    \end{split}
  \end{align}
  Let
  \[F_0:=2\int_\Sigma \|F(\sigma)\|^2\,m(d\sigma),\qquad \Theta_F(\xi):=\sup_{s\geq0} \left[\int_0^s\|F(\beta_r\xi)\|^2\,dr-sF_0\right]_+,\]
  where $a_+$ denotes the positive part of $a$. By Birkhoff's theorem, $\Theta_F(\xi)<\infty$ for $m$-almost surely.  Since
  $d(u(0),\Ac(\xi))<1$, if $\mathcal R(\xi):=2(\sup_{a\in\Ac(\xi)}\|a\|^2+1)$,  then \eqref{e.050509} yields
  \begin{align}\label{e.060418}
      \int_0^\tau\|u(s)\|_1^2\,ds\leq K_M(\xi)+c_F F_0\tau, \qquad 0\leq\tau\leq t,
  \end{align}
  where
  \[c_F:=\frac{2}{\nu^2\lambda_1},\qquad K_M(\xi):=\frac{\mathcal R(\xi)}{\nu} +c_F\Theta_F(\xi)+\frac{4Q_\infty^2}{\nu^2\lambda_1}M. \]
  Notice that the slope $c_FF_0$ is independent of the action level $M$.

  Now set $w=u-v$.  Subtracting the equations for $u$ and $v$ gives
  \[\partial_t w+\nu Aw+B(w,u)+B(v,w)=-\lambda P_Nw .\]
  Using $\langle B(v,w),w\rangle=0$ and the standard two-dimensional estimate $|\langle B(w,u),w\rangle|\leq \frac{\nu}{2}\|w\|_1^2+\frac{\kappa}{2}\|u\|_1^2\|w\|^2$, 
  we obtain
  \[\frac{d}{dt}\|w\|^2+\nu\|w\|_1^2+2\lambda\|P_Nw\|^2\leq \kappa\|u\|_1^2\|w\|^2.\]
  The spectral inequality gives
  \[\nu\|w\|_1^2+2\lambda\|P_Nw\|^2\geq \gamma_N\|w\|^2,\qquad \gamma_N:=\min\{2\lambda,\nu\lambda_{N+1}\}.\]
  By the determining mode gap in Theorem \ref{t.050601}, choose $\lambda>0$ so that $\alpha_N:=\gamma_N-\kappa c_FF_0>0$. 
  Combining this with \eqref{e.060418} and Gronwall's inequality gives, for all
  $0\leq s\leq t$,
  \begin{align}\label{e.060419}
      \|u(s)-v(s)\|^2 \leq \|u(0)-v_0\|^2\exp\{\kappa K_M(\xi)-\alpha_Ns\} .
  \end{align}
  Therefore
  \[\sup_{0\leq s\leq t}\|u(s)-v(s)\|^2\leq e^{\kappa K_M(\xi)}\|u(0)-v_0\|^2,\]
  which is \eqref{e.050605} in Lemma \ref{l.050602}.

  It remains to estimate the difference of controls.  From \eqref{e.070302}, one has $\varphi_v(s)-\varphi_u(s)=\lambda R_0(\beta_s\xi)P_Nw(s)$. 
  Hence, by \eqref{e.060419},
  \begin{align*}
      \int_0^t\|\varphi_v(s)-\varphi_u(s)\|^2\,ds
      \leq \lambda^2 e^{\kappa K_M(\xi)}\|u(0)-v_0\|^2
      \int_0^t\|R_0(\beta_s\xi)\|_{\mathcal L(P_NH,\Rb^{ \mathbf{d}})}^2 e^{-\alpha_Ns}\,ds \leq C_M(\xi)\|u(0)-v_0\|^2,
  \end{align*}
  where
  \[C_M(\xi):=e^{\kappa K_M(\xi)}\left(1+\lambda^2\int_0^\infty\|R_0(\beta_s\xi)\|_{\mathcal L(P_NH,\Rb^{ \mathbf{d}})}^2e^{-\alpha_Ns}\,ds\right).\]
  Since $\|R_0\|_{\mathcal L(P_NH,\Rb^{ \mathbf{d}})}^2\in L^1(\Sigma,m)$, Fubini's theorem implies that
  $C_M(\xi)<\infty$ for $m$-a.e. $\xi$.  This verifies Lemma \ref{l.050602}, implying \eqref{e.072701}.

  {\it Step 2: verification of \eqref{e.070101}.}
  Consider $u(t)=\Phi_{0,t,\beta_{-T}\sigma}^{\varphi}u_0, \, v(t)=\Phi_{0,t,\beta_{-T}\sigma}u_0$ and $w(t)=u(t)-v(t)$. 
  Then $w$ solves 
  \[\partial_t w+\nu Aw+B(u,u)-B(v,v)=Q(\beta_{t-T}\sigma)\varphi(t),\qquad w(0)=0 .\]
  The standard energy estimate gives
  \begin{align*}
    \frac12\frac{d}{dt}\|w\|^2+\nu\|\nabla w\|^2
    &=-\langle B(w,v),w\rangle+\bigl(Q(\beta_{t-T}\sigma)\varphi(t),w\bigr)\\
    &\leq C\|\nabla v\|\|w\|\|\nabla w\|
      +\|Q(\beta_{t-T}\sigma)\|_{\mathcal L(\Rb^{ \mathbf{d}},H)}\|\varphi(t)\|\|w\|\\
    &\leq \frac{\nu}{2}\|\nabla w\|^2
      +C\|\nabla v\|^2\|w\|^2
      +\|Q(\beta_{t-T}\sigma)\|_{\mathcal L(\Rb^{ \mathbf{d}},H)}\|\varphi(t)\|\|w\| .
  \end{align*}
  Denoting $r(t)=\|w(t)\|$, we have
  \[r'(t)\leq C_\nu\|\nabla v(t)\|^2r(t) +\|Q(\beta_{t-T}\sigma)\|_{\mathcal L(\Rb^{ \mathbf{d}},H)}\|\varphi(t)\| .\]
  Thus Gronwall's inequality yields
  \begin{align*}
      \|w(T)\|^2
      \leq \exp\left(C\int_0^T\|\nabla v(s)\|^2\,ds\right)
      \left(\int_0^T\|Q(\beta_{s-T}\sigma)\|_{\mathcal L(\Rb^{ \mathbf{d}},H)}^2\,ds\right)
      \left(\int_0^T\|\varphi(s)\|^2\,ds\right).
  \end{align*}
  From the deterministic equation for $v$,
  \[\int_0^T\|\nabla v(s)\|^2\,ds\leq \frac{R^2}{\nu}+\frac{1}{\nu^2\lambda_1}\int_{-T}^0\|F(\beta_s\sigma)\|^2\,ds,\qquad u_0\in B_R .
  \]
  Therefore, with
  \begin{align*}
    C_{R,T}(\sigma)
    :=2\left(\int_{-T}^0\|Q(\beta_s\sigma)\|_{\mathcal L(\Rb^{ \mathbf{d}},H)}^2\,ds\right)
    \exp\left[C\left(R^2+\frac{1}{\nu\lambda_1}
    \int_{-T}^0\|F(\beta_s\sigma)\|^2\,ds\right)\right],
  \end{align*}
  we obtain
  \[\|\Phi_{0,T,\beta_{-T}\sigma}^{\varphi}u_0-\Phi_{0,T,\beta_{-T}\sigma}u_0\|^2\leq C_{R,T}(\sigma)J_{u_0,T}^{\beta_{-T}\sigma}(\varphi),\qquad u_0\in B_R .
  \]
  Since $Q_\infty<\infty$ and $F\in L^2(\Sigma,m)$, the random constant $C_{R,T}(\sigma)$ is finite for every fixed $R,T>0$ and $m$-a.e. $\sigma$, which verifies \eqref{e.070101}.
\end{proof}

Next we verify that the level sets of the rate function are precompact. 
\begin{lem}\label{l.070301}
  Under the conditions of Theorem \ref{t.050601}, for $m$-a.e. $\sigma\in\Sigma$ and each $M>0$, the level set 
  \[K_M(\sigma)= \left\{E_{\Ac(\s)}\leq M\right\}\]
  is compact in $H$. 
\end{lem}

\begin{proof}
  For $M>0$, define the reachable set 
  \begin{align*} 
    \mathcal R_M(\sigma):=\left\{
    \Phi_{0,T,\beta_{-T}\sigma}^{\varphi}u_0:
    T>0,\ u_0\in\Ac(\beta_{-T}\sigma),\
    \frac12\int_0^T\|\varphi(s)\|_{ \Rb^{ \mathbf{d}}}^2ds\leq M
    \right\}.
  \end{align*}
  We first prove that $\mathcal R_M(\sigma)$ is precompact in $H$. Since the pullback attractor is invariant, any controlled path with travel time $T<1$ may be extended backward by a zero-control deterministic segment without changing either its endpoint or its action. Hence it is enough to consider $T\geq1$.

  Let
  \begin{align*}
    u(t)=\Phi_{0,t,\beta_{-T}\sigma}^{\varphi}u_0,
    \qquad
    u_0\in\Ac(\beta_{-T}\sigma),
    \qquad
    \frac12\int_0^T\|\varphi(s)\|_{ \Rb^{ \mathbf{d}}}^2ds\leq M.
  \end{align*}
  By the standard energy estimate for the controlled two-dimensional Navier--Stokes equation, for $0\leq t\leq T$,
  \begin{align*}
    \|u(t)\|^2&\leq e^{-\nu\lambda_1t}\|u_0\|^2+C\int_0^t e^{-\nu\lambda_1(t-s)}\|F(\beta_{s-T}\sigma)\|^2ds+C Q_\infty^2\int_0^t e^{-\nu\lambda_1(t-s)}\|\varphi(s)\|_{ \Rb^{ \mathbf{d}}}^2ds.
  \end{align*}
  Taking $t=T-1$ and using the standard pullback absorbing estimate for the deterministic attractor, we obtain a finite random constant $C_M(\sigma)$ such that
  \begin{align}\label{e.070213}
    \|u(T-1)\|^2\leq C_M(\sigma)
  \end{align}
  for every such controlled path. More explicitly, $C_M(\sigma)$ may be chosen in terms of $M$, $Q_\infty$, and
  \begin{align*}
    \sup_{T\geq1} e^{-\nu\lambda_1(T-1)}\sup_{a\in\Ac(\beta_{-T}\sigma)}\|a\|^2+\int_{-\infty}^{-1}e^{\nu\lambda_1(r+1)}\|F(\beta_r\sigma)\|^2dr,
  \end{align*}
  which is finite for $m$-a.e. $\sigma$. In what follows, $C_M(\sigma)$ denotes a finite random constant which may change from line to line.

  We now use the smoothing on the last unit time interval. Set
  \begin{align*}
    z(r):=u(T-1+r),\qquad\psi(r):=\varphi(T-1+r),\qquad 0\leq r\leq1.
  \end{align*}
  Then $z$ solves the equation  
  \begin{align*}
    \partial_r z+\nu Az+B(z,z)=F(\beta_{r-1}\sigma)+Q(\beta_{r-1}\sigma)\psi(r),\qquad z(0)=u(T-1). 
  \end{align*}
  By \eqref{e.070213}, $Q_\infty<\infty$, and $F\in L^2_{\mathrm{loc}}$ along $m$-a.e. environmental orbit, the forcing
  \begin{align}\label{e.070214}
    h(r):=F(\beta_{r-1}\sigma)+Q(\beta_{r-1}\sigma)\psi(r), \quad \text{ satisfies } \int_0^1\|h(r)\|^2dr
    \leq
    C_M(\sigma).
  \end{align}
  The $H$ energy estimate on $[0,1]$ gives
  \begin{align}\label{e.070215}
    \sup_{0\leq r\leq1}\|z(r)\|^2
    +
    \int_0^1\|z(r)\|_1^2dr
    \leq C_M(\sigma).
  \end{align}
  We next estimate the terminal $V$ norm. Put $Y(r):=\|z(r)\|_1^2$. 
  Taking the inner product of the equation with $Az$ and using the standard two-dimensional estimate for the nonlinear term
  \begin{align*}
    |\langle B(z,z), Az\rangle|\leq C\|z\|^{1/2}\|z\|_1\|Az\|^{3/2}
    \leq\frac{\nu}{4}\|Az\|^2+C\|z\|^2\|z\|_1^4,
  \end{align*}
  together with Young's inequality for the forcing term, one obtains
  \begin{align}\label{e.070301}
    \frac{d}{dr}Y(r)+\nu\|Az(r)\|^2&\leq C\|z(r)\|^2Y(r)^2 +C\|h(r)\|^2.
  \end{align}
  By \eqref{e.070215}, there exists $r_0\in(0,1/2)$ such that
  \begin{align*}
    Y(r_0)\leq 2\int_0^{1/2}Y(r)dr
    \leq 2C_M(\sigma).
  \end{align*}
  Using \eqref{e.070215} and \eqref{e.070214}, Gronwall's inequality on $[r_0,1]$ applied to \eqref{e.070301} yields
  \begin{align}\label{e.070216}
    \begin{split}
      \|z(1)\|_1^2 = Y(1) \leq\left(Y\left(r_0\right)+C \int_{r_0}^1\|h(r)\|^2 d r\right) \exp \left(C\sup _{0 \leq r \leq 1}\|z(r)\|^2\int_{r_0}^1Y(r) d r\right)\leq C_M(\sigma). 
    \end{split}
  \end{align}
  Since $z(1)=u(T)$, \eqref{e.070216} shows that
  \begin{align*}
    \mathcal R_M(\sigma)
    \subset
    \{v\in V:\|v\|_1^2\leq C_M(\sigma)\}.
  \end{align*}
  As the embedding $V\subset H$ is compact, $\mathcal R_M(\sigma)$ is precompact in $H$.

  It remains to pass from reachable sets to the quasipotential level set. By the definition of $E_{\Ac(\sigma)}$, if $x\in K_M(\sigma)$, then for every $\rho>0$ and every $\kappa>0$ there exists a controlled path starting from $\Ac(\beta_{-T}\sigma)$ for some $T>0$, with action at most $M+\kappa$, whose endpoint lies in $B_\rho(x)$. Therefore,
  \begin{align*}
    K_M(\sigma)
    \subset
    \overline{\mathcal R_{M+\kappa}(\sigma)}^{\,H},
    \qquad \kappa>0.
  \end{align*}
  Taking, for instance, $\kappa=1$, and using the precompactness of $\mathcal R_{M+1}(\sigma)$ in $H$, we conclude that $K_M(\sigma)$ is precompact in $H$.
  One can also verify $K_M(\sigma)$ is closed in $H$ and thus compact.
\end{proof}

Next we verify the weak exponential tightness. 
\begin{lem}
  Under the conditions of Theorem \ref{t.050601}, we have 
  \begin{align*}
    \lim _{R \rightarrow \infty} \limsup _{\varepsilon \rightarrow 0} \varepsilon \ln \mu_\sigma^{\varepsilon}\left(B_R^c\right)=-\infty, \quad m \text{-a.s.}.
  \end{align*}
\end{lem}

\begin{proof}
  Set $f(\sigma):=\|F(\sigma)\|^2$ and $\alpha:=\nu\lambda_1$. Note $\|Q(\sigma)\|_{\mathcal L_2(\Rb^{\mathbf d},H)}^2\leq \mathbf d Q_\infty^2$. 
  By Itô's formula, Poincare's inequality and Young's inequality, the solution satisfies
  \begin{align}\label{e.070401}
    d\|u(t)\|^2+\alpha\|u(t)\|^2dt\leq C f(\beta_t\sigma)dt+C\eps dt+dM_t,
  \end{align}
  where $M_t$ is the martingale term satisfying 
  \begin{align}\label{e.070402}
    d\langle M\rangle_t\leq C\eps Q_\infty^2\|u(t)\|^2dt.
  \end{align}
  Taking expectation in \eqref{e.070401}, using the stationarity relation and letting the pullback time tend to infinity, we obtain
  \begin{align}\label{e.070403}
    \int_H\|u\|^2\mu_\sigma^\eps(du)\leq K_0(\sigma),\quad \eps\in(0,1], \text{ where } K_0(\sigma):=C\int_{-\infty}^0e^{\alpha r}\left(1+f(\beta_r\sigma)\right)dr, 
  \end{align}
  and $K_0\in L^1(\Sigma,m)$.

  We next prove an exponential pullback estimate. Choose $a>0$ small enough such that $CaQ_\infty^2\leq\alpha/4$, set $\gamma:=\alpha/2$, and define $a_T(t):=ae^{-\gamma(T-t)}$ for $0\leq t\leq T$. Then $a_T(T)=a$, $a_T(0)=ae^{-\gamma T}$, and
  \begin{align}\label{e.070303}
    A_T(t):=a_T'(t)-\alpha a_T(t)+CQ_\infty^2a_T(t)^2\leq0.
  \end{align}
  Let
  \begin{align*}
    X(t):=\|u(t)\|^2,\qquad Z(t):=\exp\left(\frac{a_T(t)X(t)}{\eps}\right). 
  \end{align*}
  By Itô's formula, \eqref{e.070401}--\eqref{e.070402} and \eqref{e.070303}, we have 
  \begin{align*}
    dZ(t)&\leq Z(t)\left[\frac{A_T(t)}{\eps}X(t)+\frac{Ca_T(t)}{\eps}f(\beta_{t-T}\sigma)+Ca_T(t)\right]dt+Z(t)\frac{a_T(t)}{\eps}dM_t\\
    &\leq Z(t)\left[\frac{Ca_T(t)}{\eps}f(\beta_{t-T}\sigma)+Ca_T(t)\right]dt+Z(t)\frac{a_T(t)}{\eps}dM_t\\
    &=Z(t)d\Lambda_T(t)+Z(t)\frac{a_T(t)}{\eps}dM_t,
  \end{align*}
  by setting 
  \begin{align*}
    \Lambda_T(t):=\frac{C}{\eps}\int_0^t a_T(s)f(\beta_{s-T}\sigma)ds+C\int_0^t a_T(s)ds.
  \end{align*}
  The integration factor then gives
  \begin{align*}
    d\left(e^{-\Lambda_T(t)}Z(t)\right)
    \leq
    e^{-\Lambda_T(t)}Z(t)\frac{a_T(t)}{\eps}dM_t .
  \end{align*}
  Let $\tau_N$ be a localizing sequence for the stochastic integral on the right. Taking expectation on $[0,T\wedge\tau_N]$, we obtain
  \begin{align*}
    \mathbf E\left(e^{-\Lambda_T(T\wedge\tau_N)}Z(T\wedge\tau_N)\right)\leq Z(0).
  \end{align*}
  Letting $N\to\infty$ and using Fatou's lemma yield $\mathbf E Z(T)\leq Z(0)e^{\Lambda_T(T)}$, implying 
  \begin{align}\label{e.070404}
    \mathbf E\exp\left(\frac{a\|\Phi_{0,T,\beta_{-T}\sigma}^{\eps}u\|^2}{\eps}\right)\leq \exp\left(\frac{ae^{-\gamma T}\|u\|^2}{\eps}+\frac{L_F(\sigma)}{\eps}+C\right),
  \end{align}
  where
  \begin{align*}
    L_F(\sigma):=C\int_{-\infty}^0e^{\gamma r}f(\beta_r\sigma)dr<\infty, \quad \text{ for } m\text{-a.e. }  \sigma\in\Sigma. 
  \end{align*}
  For $R,T>0$, define $G_{T,R}:=\left\{u\in H:e^{-\gamma T}\|u\|^2\leq R^2/4\right\}$. 
  By Chebyshev's inequality and \eqref{e.070404}, for $u\in G_{T,R}$,
  \begin{align}\label{e.070405}
    \mathbf P\left\{\|\Phi_{0,T,\beta_{-T}\sigma}^{\eps}u\|>R\right\}\leq \exp\left(-\frac{3aR^2/4-L_F(\sigma)}{\eps}+C\right).
  \end{align}
  On the other hand, by \eqref{e.070403},
  \begin{align}\label{e.070406}
    \mu_{\beta_{-T}\sigma}^\eps(G_{T,R}^c)\leq \frac{4e^{-\gamma T}}{R^2}\int_H\|v\|^2\mu_{\beta_{-T}\sigma}^\eps(dv)\leq \frac{4e^{-\gamma T}}{R^2}K_0(\beta_{-T}\sigma).
  \end{align}
  Therefore, by stationarity, \eqref{e.070405} and \eqref{e.070406},
  \begin{align}\label{e.070407}
    \mu_\sigma^\eps(B_R^c)\leq \frac{4e^{-\gamma T}}{R^2}K_0(\beta_{-T}\sigma)+\exp\left(-\frac{3aR^2/4-L_F(\sigma)}{\eps}+C\right).
  \end{align}
  Since $K_0\in L^1(\Sigma,m)$, the Borel--Cantelli lemma implies that, for every $\eta>0$ and $m$-a.e. $\sigma\in\Sigma$, there exists $N_\eta(\sigma)$ such that
  \begin{align}\label{e.070408}
    K_0(\beta_{-n}\sigma)\leq e^{\eta n},\qquad n\geq N_\eta(\sigma).
  \end{align}
  Taking $T=T_{\eps,R}:=\left\lceil\frac{R^2}{\eps}\right\rceil,$ and using \eqref{e.070407}--\eqref{e.070408}, we obtain for fixed $R>0$,
  \begin{align*}
    \limsup_{\eps\to0}\eps\ln\mu_\sigma^\eps(B_R^c)
    \leq
    \max\left\{-(\gamma-\eta)R^2,\ L_F(\sigma)-\frac{3aR^2}{4}\right\}.
  \end{align*}
  Choosing $\eta\in(0,\gamma)$ and letting $R\to\infty$, we conclude that
  \begin{align*}
    \lim_{R\to\infty}\limsup_{\eps\to0}\eps\ln\mu_\sigma^\eps(B_R^c)=-\infty,\qquad m\text{-a.s.}.
  \end{align*}
\end{proof}

\section{Applications to damped Sine--Gordon equations}\label{s.060702}

Let $(\Sigma,m)$ be a standard Borel probability space and $(\beta_t)_{t\in\mathbb R}$ be an invertible ergodic measure-preserving dynamical system on $\Sigma$.  Fix $\mathbf{d}\in\mathbb N$ and let
\begin{align*}
  W(t)=(W_1(t),\ldots,W_{\mathbf{d}}(t))
\end{align*}
be a two-sided standard Brownian motion in $\Rb^{\mathbf{d}}$.  Let $D\subset\Rb^d$ be a bounded domain with smooth boundary and let $H:=L^2(D),\, V:=H_0^1(D),\, A=-\Delta$
with Dirichlet boundary condition.  We write $\|\cdot\|$ for the norm of $H$ and $\|u\|_1:=\|A^{1/2}u\|$ for the norm of $V$.  Let $\{e_j\}_{j\geq1}$ be an orthonormal basis of $H$ consisting of eigenfunctions of $A$, with eigenvalues $0<\lambda_1\leq\lambda_2\leq\cdots$, and let $P_N$ denote the orthogonal projection in $H$ onto $\spa\{e_1,\ldots,e_N\}$.

We consider the damped Sine--Gordon equation in a random environment
\begin{align}\label{e.060401}
  \partial_{tt}u+\alpha\partial_tu+Au+\varsigma\sin u
  =F(\beta_t\s)+\sqrt{\eps}\,Q(\beta_t\s)\dot W_t,
  \qquad u|_{\partial D}=0,
\end{align}
where $\alpha>0$ and $\varsigma\in\Rb$ are fixed constants.  Here
\begin{align*}
  Q(\sigma):\Rb^{\mathbf{d}}\to H, \qquad Q(\sigma)a=\sum_{k=1}^{\mathbf{d}}a_kq_k(\sigma),
\end{align*}
with $q_k(\sigma)\in V$.  Thus
\begin{align*}
  Q(\beta_t\sigma)\dot W_t =\sum_{k=1}^{\mathbf{d}}q_k(\beta_t\sigma)\dot W_k(t)
\end{align*}
is a finite-dimensional additive noise.  Equivalently, setting $v=\partial_tu$ and $U=(u,v)$, we write
\begin{align}\label{e.060402}
  du&=v\,dt,\nonumber\\
  dv+\big(\alpha v+Au+\varsigma\sin u\big)dt
  &=F(\beta_t\s)dt+\sqrt{\eps}\,Q(\beta_t\s)dW_t .
\end{align}
The phase space is $X:=V\times H$ with norm  $\|U\|_X^2:=\|u\|_1^2+\|v\|^2$ for $U=(u,v)$. 
Under the assumptions below, \eqref{e.060402} is well-posed in $X$ and generates a Markov cocycle, whose solution map is denoted by $\Phi_{0,t,\s}^{\eps}$.  When $\eps=0$, the deterministic solution map $\Phi_{0,t,\s}$ satisfies $\Phi_{0,s+t,\s}=\Phi_{0,t,\beta_s\s}\circ\Phi_{0,s,\s}$ for $s,t\geq0$. 

We assume throughout this section that
\begin{align}\label{e.060403}
   \|F\|_1^2\in L^1(\Sigma,m),\qquad
   Q_{1,\infty}:=\operatorname*{sup}_{\sigma\in\Sigma}
        \|Q(\sigma)\|_{\mathcal L(\Rb^{\mathbf{d}},V)}<\infty .
\end{align}
Note that the following boundedness properties hold from $Q_{1,\infty}<\infty$ as we are considering finite dimensional noise:
\begin{align}\label{e.060404}
  \|Q\|_{\mathcal L_2(\Rb^{\mathbf{d}},V)}^2\in L^1(\Sigma,m),    \qquad
   Q_{0,\infty}:=\operatorname*{sup}_{\sigma\in\Sigma}
        \|Q(\sigma)\|_{\mathcal L(\Rb^{\mathbf{d}},H)}<\infty.
\end{align}

The noise is assumed to be non-degenerate only on a fixed determining subspace.  Choose once and for all
$\rho_*>0$ sufficiently small so that
\begin{align}\label{e.060405}
   0<\rho_*<\frac{\alpha}{2},
   \qquad
   \rho_*(\alpha-\rho_*)\leq \frac{\lambda_1}{4}.
\end{align}
Let $N_0\geq1$ be so large that
\begin{align}\label{e.060406}
   \frac{|\varsigma|}{\sqrt{\lambda_{N_0+1}}}
   \leq \frac{\rho_*}{4}.
\end{align}
This explicit determining mode gap will be used to absorb the high-mode part of the Lipschitz nonlinearity.  We assume that there exists a measurable family of linear maps $R_0(\sigma):P_{N_0}H\to \Rb^{\mathbf{d}}$ 
such that
\begin{align}\label{e.060407}
   Q(\sigma)R_0(\sigma)P_{N_0}=P_{N_0},
   \qquad
   \|R_0\|_{\mathcal L(P_{N_0}H,\Rb^{\mathbf{d}})}^2\in L^1(\Sigma,m).
\end{align}
No non-degeneracy is imposed on modes above $P_{N_0}H$.

\begin{thm}\label{t.060401}
  Assume \eqref{e.060403}, \eqref{e.060406}, and \eqref{e.060407}. Then the damped Sine--Gordon Markov cocycle \eqref{e.060402} has a unique stationary measure $\mu_\s^{\eps}$ that attracts transition probabilities exponentially under the Wasserstein metric $\mathbb{W}_{\|\cdot\|_X\wedge1}$.  
  
  Moreover, for $m$-a.e. $\s\in\Sigma$, the family $\{\mu_\s^{\eps}\}_{\eps\in(0,1]}$ obeys the upper bound Freidlin--Wentzell LDP with a good rate function.  If the pullback attractor of the limiting deterministic Sine--Gordon cocycle consists of a single random point, then the full LDP holds with the same rate function.
\end{thm}

\begin{rem}\label{r.070301}
  The trajectory Freidlin--Wentzell LDP in Theorem \ref{t.060401} is standard for the additive noise damped wave equation with globally Lipschitz nonlinearity.  In the present case the map $u\mapsto\sin u$ is globally Lipschitz from $H$ to $H$, and the skeleton equation below is globally well-posed in $X$.  Thus the usual weak-convergence or contraction-principle \cite{salins2019uniform,budhiraja2008large} gives the required uniform finite horizon LDP on bounded subsets of $X$.
\end{rem}

\subsection{Ergodicity and exponential mixing}

For the ergodicity argument we use the equivalent energy
\begin{align*}
   \mathcal E(U):=\|v+\rho_*u\|^2+\|u\|_1^2,
   \qquad U=(u,v)\in X.
\end{align*}
By \eqref{e.060405} and Poincar\'{e}'s inequality, there exist constants $0<c_E<C_E<\infty$ such that
\begin{align}\label{e.060408}
   c_E\|U\|_X^2\leq \mathcal E(U)\leq C_E\|U\|_X^2,
   \qquad U\in X.
\end{align}

\begin{thm}\label{t.060402}
  Under the assumptions of Theorem \ref{t.060401}, the Markov cocycle generated by \eqref{e.060402} has a unique stationary measure that is exponentially mixing under the Wasserstein metric $W_{\|\cdot\|_X\wedge1}$.
\end{thm}

\begin{proof}
  We verify the assumptions of Theorem \ref{t.042501}.  Given two initial states $U_0=(u_0,v_0)$ and $\widetilde U_0=(\widetilde u_0,\widetilde v_0)$, let $X_t^\s=(u(t),v(t))=\Phi_{0,t,\s}^{\eps}U_0$ and set $U(U_0)=V(U_0)=\mathcal E(U_0)$ and  $S(U_0)=\mathcal E(U_0).$
  Define the premetric
  \begin{align}\label{e.060409}
      p(U_1,U_2):=\|w_1\|_1^2+\|w_2+\rho_*w_1\|^2,
      \qquad U_i=(u_i,v_i),
  \end{align}
  where $w_1=u_1-u_2$ and $w_2=v_1-v_2$.  By \eqref{e.060408}, $p$ is equivalent to the square of the $X$-distance.

  We first verify the Lyapunov estimate.  Let $r(t):=v(t)+\rho_*u(t)$. 
  From \eqref{e.060402},
  \begin{align*}
    dr+ (\alpha-\rho_*)r\,dt =\big(\rho_*(\alpha-\rho_*)u-Au-\varsigma\sin u+F(\beta_t\s)\big)dt +\sqrt\eps Q(\beta_t\s)dW_t .
  \end{align*}
  Applying Ito's formula to $\mathcal E(X_t^\s)=\|r(t)\|^2+\|u(t)\|_1^2$ and using Poincare's inequality, \eqref{e.060405}, and Young's inequality, we obtain
  \begin{align}\label{e.060410}
     d\mathcal E(X_t^\s)+c_0\mathcal E(X_t^\s)dt
     \leq C_0\left(1+\|F(\beta_t\s)\|^2+\eps\|Q(\beta_t\s)\|_{\mathcal L_2(\Rb^{\mathbf{d}},H)}^2\right)dt+dM_t^\s,
  \end{align}
  where $c_0,C_0>0$ are constants and $dM_t^\s=2\sqrt\eps\langle r(t),Q(\beta_t\s)dW_t\rangle$. 
  Consequently,
  \begin{align*}
    d\langle M^\s\rangle_t =4\eps\|Q(\beta_t\s)^*r(t)\|_{\Rb^{\mathbf{d}}}^2dt \leq C Q_{0,\infty}^2\mathcal E(X_t^\s)dt .
  \end{align*}

  Taking expectations in \eqref{e.060410} gives assumption $(H1)$, with
  \begin{align*}
    K(\sigma)=C\left(1+\|F(\sigma)\|^2+\|Q(\sigma)\|_{\mathcal L_2(\Rb^{\mathbf{d}},H)}^2\right), \qquad K\in L^1(\Sigma,m).
  \end{align*}
  Moreover, since $\eps\in(0,1]$, integrating \eqref{e.060410} gives
  \begin{align*}
    \mathcal E(X_t^\s)+c_0\int_0^t\mathcal E(X_s^\s)\,ds \leq \mathcal E(X_0^\s)+\int_0^t b(\beta_s\s)\,ds+M_t^\s,
  \end{align*}
  where
  \begin{align*}
    b(\sigma):=C_0\left(1+\|F(\sigma)\|^2+\|Q(\sigma)\|_{\mathcal L_2(\Rb^{\mathbf{d}},H)}^2\right)\in L^1(\Sigma,m).
  \end{align*}
  Thus the structural part of $(H3)$ holds with $U=S=\mathcal E,\, \varpi=c_0, \, b_1=CQ_{0,\infty}^2$ and $b_2=0.$
  The remaining strict gap in $(H3)$ will be verified after the contraction estimate below.  Indeed, the Sine--Gordon coupling gives a pure exponential contraction, and hence $(H2)$ holds with any sufficiently small $\kappa>0$.

  We next verify the contraction estimate.  Let $Y_t^\s=(\widetilde u(t),\widetilde v(t))$ solve the controlled equation
  \begin{align}\label{e.060430}
    d\widetilde u&=\widetilde v\,dt,\nonumber\\
    d\widetilde v+\big(\alpha\widetilde v+A\widetilde u+\varsigma\sin\widetilde u
        +\varsigma P_{N_0}(\sin u-\sin\widetilde u)\big)dt
    &=F(\beta_t\s)dt+\sqrt\eps Q(\beta_t\s)dW_t,
  \end{align}
  with initial data $\widetilde U_0$. For simplicity, write $N=N_0$ and $\Pi_N:=I-P_N$. 
  Set $w=u-\widetilde u$ and $z=v-\widetilde v=\partial_tw$.  Then
  \begin{align}\label{e.060412}
      \partial_{tt}w+\alpha\partial_tw+Aw
      +\varsigma\Pi_{N}(\sin u-\sin\widetilde u)=0 .
  \end{align}
  Set $y=z+\rho_*w$.  Then $\partial_tw=z=y-\rho_*w$, and from \eqref{e.060412} we infer
  \begin{align}\label{e.060411}
      \partial_t y+(\alpha-\rho_*)y+Aw
      -\rho_*(\alpha-\rho_*)w
      =-\varsigma\Pi_{N}(\sin u-\sin\widetilde u).
  \end{align}
  Define
  \begin{align*}
      L_N(t):=\|P_{N}y(t)\|^2+\|P_{N}w(t)\|_1^2, \, \text{ and } \, H_N(t):=\|\Pi_{N}y(t)\|^2+\|\Pi_{N}w(t)\|_1^2 .
  \end{align*}

  We first estimate the low modes.  Projecting \eqref{e.060411} onto $P_{N}H$, the right-hand side vanishes.  Hence
  \begin{align*}
    \partial_t P_{N}y+(\alpha-\rho_*)P_{N}y+AP_{N}w -\rho_*(\alpha-\rho_*)P_{N}w=0,
  \end{align*}
  and
  \begin{align*}
    \partial_t P_{N}w=P_{N}y-\rho_*P_{N}w .
  \end{align*}
  Taking the scalar product of the first equation with $2P_{N}y$ and using the second equation gives
  \begin{align*}
      \frac d{dt}L_N =-2(\alpha-\rho_*)\|P_{N}y\|^2-2\rho_*\|P_{N}w\|_1^2+2\rho_*(\alpha-\rho_*)\langle P_{N}w,P_{N}y\rangle .
  \end{align*}
  By Poincare's inequality and Young's inequality,
  \begin{align*}
      2\rho_*(\alpha-\rho_*)|\langle P_{N}w,P_{N}y\rangle|
      &\leq
      2\rho_*(\alpha-\rho_*)\lambda_1^{-1/2}
      \|P_{N}w\|_1\|P_{N}y\| \\
      &\leq
      \rho_*\|P_{N}w\|_1^2
      +\frac{\rho_*(\alpha-\rho_*)^2}{\lambda_1}\|P_{N}y\|^2 .
  \end{align*}
  Therefore, by \eqref{e.060405},
  \begin{align}\label{e.060413}
    \frac d{dt}L_N+\rho_*L_N\leq0 .
  \end{align}

  We now estimate the high modes.  Projecting \eqref{e.060411} onto $\Pi_{N}H$ and using $\partial_t\Pi_{N}w=\Pi_{N}y-\rho_*\Pi_{N}w$, we obtain
  \begin{align*}
      \frac d{dt}H_N
      &+2(\alpha-\rho_*)\|\Pi_{N}y\|^2+2\rho_*\|\Pi_{N}w\|_1^2\\
      &\leq2\rho_*(\alpha-\rho_*)|\langle\Pi_{N}w,\Pi_{N}y\rangle|
      +2|\varsigma|\,\|\Pi_{N}(\sin u-\sin\widetilde u)\|\|\Pi_{N}y\| .
  \end{align*}
  The first term on the right is estimated by the inverse Poincare inequality on high modes:
  \begin{align*}
      2\rho_*(\alpha-\rho_*)|\langle\Pi_{N}w,\Pi_{N}y\rangle|
      &\leq
      2\rho_*(\alpha-\rho_*)\lambda_{N_0+1}^{-1/2}
      \|\Pi_{N}w\|_1\|\Pi_{N}y\|                                      \\
      &\leq
      \frac{\rho_*}{4}\|\Pi_{N}w\|_1^2
      +\frac{4\rho_*(\alpha-\rho_*)^2}{\lambda_{N_0+1}}\|\Pi_{N}y\|^2  \leq
      \frac{\rho_*}{4}\|\Pi_{N}w\|_1^2
      +(\alpha-\rho_*)\|\Pi_{N}y\|^2 .
  \end{align*}
  In the last inequality we used $4\rho_*(\alpha-\rho_*)\leq\lambda_1\leq\lambda_{N_0+1}$, which follows from \eqref{e.060405}.

  For the nonlinear term, we use the global Lipschitz property of the nonlinear term:
  \begin{align*}
    \|\Pi_{N}(\sin u-\sin\widetilde u)\| \leq \|\sin u-\sin\widetilde u\| \leq \|w\| \leq \|P_{N}w\|+\|\Pi_{N}w\|.
  \end{align*}
  Hence
  \begin{align*}
      2|\varsigma|\,\|\Pi_{N}(\sin u-\sin\widetilde u)\|\|\Pi_{N}y\|
      &\leq2|\varsigma|\|P_{N}w\|\|\Pi_{N}y\|
       +2|\varsigma|\|\Pi_{N}w\|\|\Pi_{N}y\|\\
      &\leq\frac{\alpha-\rho_*}{4}\|\Pi_{N}y\|^2
      +\frac{4|\varsigma|^2}{(\alpha-\rho_*)\lambda_1}\|P_{N}w\|_1^2\\
      &\quad+\frac{\alpha-\rho_*}{4}\|\Pi_{N}y\|^2
      +\frac{4|\varsigma|^2}{(\alpha-\rho_*)\lambda_{N_0+1}}\|\Pi_{N}w\|_1^2 .
  \end{align*}
  As $\rho_*\leq\alpha-\rho_*$, the determining mode gap \eqref{e.060406} implies $  \frac{4|\varsigma|^2}{(\alpha-\rho_*)\lambda_{N_0+1}} \leq\frac{\rho_*}{4}.$
  Therefore, with $C_{\rm sg}:=\frac{4|\varsigma|^2}{(\alpha-\rho_*)\lambda_1}$, 
  we have
  \begin{align*}
    2|\varsigma|\,\|\Pi_{N}(\sin u-\sin\widetilde u)\|\|\Pi_{N}y\| \leq\frac{\alpha-\rho_*}{2}\|\Pi_{N}y\|^2 +\frac{\rho_*}{4}\|\Pi_{N}w\|_1^2+C_{\rm sg}L_N .
  \end{align*}
  Combining the preceding estimates gives
  \begin{align}\label{e.060414}
      \frac d{dt}H_N+\frac{\rho_*}{2}H_N
      \leq C_{\rm sg}L_N .
  \end{align}

  Now let $A_0\geq \max\left\{1,\frac{2C_{\rm sg}}{\rho_*}\right\}$
  and define $p_*(U,\widetilde U):=A_0L_N+H_N$. 
  Multiplying \eqref{e.060413} by $A_0$ and adding \eqref{e.060414}, we obtain
  \begin{align*}
      \frac d{dt}p_*(X_t^\s,Y_t^\s)\leq-(A_0\rho_*-C_{\rm sg})L_N(t)-\frac{\rho_*}{2}H_N(t)\leq-\frac{\rho_*}{2}\big(A_0L_N(t)+H_N(t)\big).
  \end{align*}
  Thus
  \begin{align}\label{e.060415}
      p_*(X_t^\s,Y_t^\s)
      \leq e^{-\rho_*t/2}p_*(U_0,\widetilde U_0),
      \qquad t\geq0 .
  \end{align}
  Since $L_N+H_N=\|y\|^2+\|w\|_1^2$ and $A_0\geq1$, the premetric $p_*$ is equivalent to the original premetric \eqref{e.060409}.  We relabel $p_*$ as $p$ and set $\zeta:=\frac{\rho_*}{2}$. 
  This verifies $(H2)$ in the form
  \begin{align*}
    p(X_t^\s,Y_t^\s)\leq e^{-\zeta t}p(U_0,\widetilde U_0).
  \end{align*}
  Since $S=\mathcal E\geq0$, the stated form of $(H2)$ also holds with the same $\zeta$ and with any fixed $\kappa>0$:
  \begin{align*}
    p(X_t^\s,Y_t^\s) \leq p(U_0,\widetilde U_0)\exp\left(-\zeta t+\kappa\int_0^tS(X_s^\s)\,ds\right).
  \end{align*}
  We now choose $\kappa>0$ sufficiently small so that $\zeta>\frac{\kappa}{c_0}\int_\Sigma b(\sigma)\,m(d\sigma).$
  Since $\varpi=c_0$ in $(H3)$, this is exactly the strict gap condition required in $(H3)$.

  It remains to verify the Girsanov condition $(H4)$.  Since \eqref{e.060407} holds, the process $Y_t^\s$ in \eqref{e.060430} is the solution of the original Sine--Gordon equation starting from $\widetilde U_0$ driven by the shifted Wiener path
  \begin{align*}
  W(t)+\int_0^t h(s)ds, \qquad h(s)=-\frac{\varsigma}{\sqrt\eps} R_0(\beta_s\s)P_{N_0}(\sin u(s)-\sin\widetilde u(s)).
\end{align*}
  Therefore, using $|\sin a-\sin b|\leq |a-b|$ and \eqref{e.060415},
  \begin{align*}
      \int_0^t\|h(s)\|_{\Rb^{\mathbf{d}}}^2ds
      \leq \frac{C\varsigma^2}{\eps}p(U_0,\widetilde U_0)
      \int_0^t\|R_0(\beta_s\s)\|_{\mathcal L(P_{N_0}H,\Rb^{\mathbf{d}})}^2e^{-\zeta s}ds .
  \end{align*}
  Since $\|R_0\|_{\mathcal L(P_{N_0}H,\Rb^{\mathbf{d}})}^2\in L^1(\Sigma,m)$, the integral on the right is finite for $m$-a.e. $\s$.  The same Girsanov estimate used in the Navier--Stokes example verifies $(H4)$ for each fixed $\eps\in(0,1]$, with
  \begin{align*}
    b_3^\eps(\sigma)=\frac{C\varsigma^2}{\eps}\|R_0(\sigma)\|_{\mathcal L(P_{N_0}H,\Rb^{\mathbf{d}})}^2 .
  \end{align*}
  The conclusions now follow from Theorem \ref{t.042501}.
\end{proof}

\subsection{The Freidlin--Wentzell LDP}

The controlled Sine--Gordon equation for the symbol $\sigma$ and a control $\varphi\in L^2_{\rm loc}(\Rb;\Rb^{\mathbf{d}})$ is
\begin{align*}
  \partial_{tt}u+\alpha\partial_tu+Au+\varsigma\sin u
  =F(\beta_t\s)+Q(\beta_t\s)\varphi,
  \qquad u|_{\partial D}=0 .
\end{align*}
The solution with initial condition $U_0=(u_0,v_0)\in X$ is denoted by
$\Phi_{0,t,\s}^{\varphi}U_0$.  For $U\in C([0,T];X)$, define
\begin{align*}
   I_T^\s(U)=\inf_\varphi J_T^\s(\varphi),
   \qquad
   J_T^\s(\varphi)=\frac12\int_0^T\|\varphi(s)\|_{\Rb^{\mathbf{d}}}^2ds,
\end{align*}
where the infimum is taken over all $\varphi\in L^2(0,T;\Rb^{\mathbf{d}})$ such that
$U(t)=\Phi_{0,t,\s}^{\varphi}U(0)$ for all $t\in[0,T]$.  If no such control exists, then $I_T^\s(U)=\infty$.

It is known \cite{wang2007pullback,chepyzhov2005sinegordon} that the deterministic Sine--Gordon cocycle has a compact pullback attractor
$\Ac(\s)\subset X$ satisfying $\Phi_{0,t,\s}\Ac(\s)=\Ac(\beta_t\s)$ for all $t\geq0$  and for every bounded set $B\subset X$,
\begin{align*}
  \lim_{t\to\infty}d_X\big(\Phi_{0,t,\beta_{-t}\s}B,\Ac(\s)\big)=0, \qquad m\text{-a.s.}
\end{align*}
Let $P_{0,t,\s}^{\eps}$ be the Markov operators and let $\mu_\s^{\eps}$ be the stationary measure of \eqref{e.060402}, so that
\begin{align*}
  \mu_\s^{\eps}P_{0,t,\s}^{\eps}=\mu_{\beta_t\s}^{\eps}.
\end{align*}
The quasipotential $E_{\Ac(\sigma)}$ and its level sets $K_M(\sigma)$
are defined by the same formula as \eqref{e.073003}, with the Hilbert space there replaced by the present phase space $X$.

In view of Remark \ref{r.070301}, to prove the Freidlin--Wentzell LDP for the stationary measure by Theorem \ref{t.080601}, it remains to verify the tracking property, the nontrivial escaping
energy property, the precompactness of the level sets and the weak exponential tightness. 

We first verify the tracking property and the nontrivial escaping energy property in Assumption \ref{a.050201}.

\begin{lem}
  Under the assumptions of Theorem \ref{t.060401}, the tracking property \eqref{e.072701} and the nontrivial escaping energy property \eqref{e.070101} in Assumption \ref{a.050201} are true for the Sine--Gordon cocycle.
\end{lem}

\begin{proof}
  The proof is divided into two steps as in the case of Navier-Stokes equations. 

  \emph{Step 1: verification of Lemma \ref{l.050602}.}
  Fix $\xi\in\Sigma$, $M>0$, and $t>0$.  Let $U=(u,u_t)$ be a controlled path
  with control $\varphi_u$ satisfying  $d_X(U(0),\Ac(\xi))<1$
  and 
  \[\frac12\int_0^t\|\varphi_u(s)\|_{\Rb^{\mathbf{d}}}^2ds\leq M.\]  
  For any $V_0=(\widetilde u_0,\widetilde v_0)\in\Ac(\xi)$, define $V=(\widetilde u,\widetilde u_t)$ by
  \begin{align}\label{e.060433}
      \widetilde u_{tt}+\alpha\widetilde u_t+A\widetilde u+\varsigma\sin\widetilde u
      +\varsigma P_{N_0}(\sin u-\sin\widetilde u)
      =F(\beta_t\xi)+Q(\beta_t\xi)\varphi_u,
      \qquad V(0)=V_0.
  \end{align}
  Since $Q(\sigma)R_0(\sigma)P_{N_0}=P_{N_0}$, \eqref{e.060433} is the original controlled equation with the control
  \begin{align}\label{e.060431}
      \varphi_v(t)=\varphi_u(t)-\varsigma R_0(\beta_t\xi)P_{N_0}(\sin u(t)-\sin\widetilde u(t)).
  \end{align}
  Let $w=u-\widetilde u$.  Then $w$ solves \eqref{e.060412} with $\s$ replaced by $\xi$.  The deterministic estimate \eqref{e.060415} gives
  \begin{align}\label{e.060432}
      \|U(s)-V(s)\|_X^2
      \leq C e^{-\zeta s}\|U(0)-V_0\|_X^2,
      \qquad 0\leq s\leq t.
  \end{align}
  In particular,
  \begin{align*}
    \sup_{0\leq s\leq t}\|U(s)-V(s)\|_X^2 \leq C\|U(0)-V_0\|_X^2 .
  \end{align*}
  Moreover, by \eqref{e.060431}, the Lipschitz property of $\sin$, and \eqref{e.060432},
  \begin{align*}
      \int_0^t\|\varphi_v(s)-\varphi_u(s)\|_{\Rb^{\mathbf{d}}}^2ds
      &\leq C\varsigma^2\|U(0)-V_0\|_X^2
      \int_0^t\|R_0(\beta_s\xi)\|_{\mathcal L(P_{N_0}H,\Rb^{\mathbf{d}})}^2e^{-\zeta s}ds  \\
      &\leq C(\xi)\|U(0)-V_0\|_X^2,
  \end{align*}
  where
  \begin{align*}
    C(\xi):=C\varsigma^2\int_0^\infty \|R_0(\beta_s\xi)\|_{\mathcal L(P_{N_0}H,\Rb^{\mathbf{d}})}^2e^{-\zeta s}ds<\infty
  \end{align*}
  for $m$-a.e. $\xi$, by Fubini's theorem and \eqref{e.060407}.  This proves the conditions of Lemma \ref{l.050602}, implying \eqref{e.072701}. 

  \emph{Step 2: verification of \eqref{e.070101}.}
  Let $U(t)=\Phi_{0,t,\beta_{-T}\s}^{\varphi}U_0$, $\bar U(t)=\Phi_{0,t,\beta_{-T}\s}U_0$ and $\mathcal{U}(t)=U(t)-\bar U(t)$. 
  Then $\mathcal{U}=(w,w_t)$ satisfies
  \begin{align*}
     w_{tt}+\alpha w_t+Aw+\varsigma (\sin u-\sin\bar u)
     =Q(\beta_{t-T}\s)\varphi(t),
     \qquad \mathcal{U}(0)=0.
  \end{align*}
  Applying the energy estimate based on $w_t+\rho_*w$ , we get
  \begin{align*}
    \frac d{dt}\mathcal E(\mathcal{U}(t)) \leq C\mathcal E(\mathcal{U}(t))+C\|Q(\beta_{t-T}\s)\varphi(t)\|^2 .
  \end{align*}
  Hence
  \begin{align}\label{e.060434}
      \|\mathcal{U}(T)\|_X^2
      \leq C e^{CT}\int_0^T\|Q(\beta_{t-T}\s)\|_{\mathcal L(\Rb^{\mathbf{d}},H)}^2\|\varphi(t)\|_{\Rb^{\mathbf{d}}}^2dt
      \leq C_{R,T}(\s)J_{U_0,T}^{\beta_{-T}\s}(\varphi),
  \end{align}
  where one may take $C_{R,T}(\s)=2Ce^{CT}Q_{0,\infty}^2.$ Thus \eqref{e.060434}  implies \eqref{e.070101}.
\end{proof}

Next we verify the compactness of level sets.  Because the damped wave equation has no parabolic smoothing, this step uses the additional compact-tail information encoded in $Q_{1,\infty}<\infty$ and $\|F\|_1^2\in L^1(\Sigma,m)$.

\begin{lem}
  For $m$-a.e. $\sigma\in\Sigma$ and every $M>0$, the level set
  \begin{align*}
  K_M(\sigma)=\{E_{\Ac(\sigma)}\leq M\}
\end{align*}
  is compact in $X$.
\end{lem}

\begin{proof}
  Let
  \begin{align*}
    \mathcal R_M(\sigma):= \left\{U(T):T>0,\ U(0)\in\Ac(\beta_{-T}\sigma),\ I_{U_0,T}^{\beta_{-T}\sigma}(U)\leq M\right\}.
  \end{align*}
  It suffices to show that $\mathcal R_M(\sigma)$ is precompact in $X$.  Since the attractor is invariant, any path in the definition of $\mathcal R_M(\sigma)$ may be extended backward by a zero-control deterministic segment of arbitrary length.  Thus, when estimating a given endpoint, we may assume that the travel time $T$ is as large as needed.

  We first derive a uniform weighted energy bound.  From the deterministic version of \eqref{e.060410} applied to the controlled equation and from \eqref{e.060404}, one obtains
  \begin{align*}
    \begin{split}
      \frac{d}{d s} \mathcal{E}(U(s))+c_0 \mathcal{E}(U(s)) &\leq C\left(1+\left\|F\left(\beta_{s-T} \sigma\right)\right\|^2+\left\|Q\left(\beta_{s-T} \sigma\right) \varphi(s)\right\|^2\right)\\
      &\leq  C\left(1+\left\|F\left(\beta_{s-T} \sigma\right)\right\|^2+Q_{0, \infty}^2\|\varphi(s)\|_{\mathbb{R}^{\mathbf{d}}}^2\right),
    \end{split}
  \end{align*}
  from which we infer 
  \begin{align*}
    \begin{aligned}
      & \mathcal{E}(U(T))+c_0 \int_0^T e^{-c_0(T-s)} \mathcal{E}(U(s)) d s \\
      & \quad \leq e^{-c_0 T} \mathcal{E}(U(0))+C \int_{-T}^0 e^{c_0 r}\left(1+\left\|F\left(\beta_r \sigma\right)\right\|^2\right) d r+2 C M Q_{0, \infty}^2.
      \end{aligned}
  \end{align*}
  As $U(0) \in \mathcal{A}\left(\beta_{-T} \sigma\right)$, the deterministic pullback absorbing estimate gives
  \begin{align*}
    \sup _{a \in \mathcal{A}(\tau)} \mathcal{E}(a) \leq C\left(1+\int_{-\infty}^0 e^{c_0 r}\left(1+\left\|F\left(\beta_r \tau\right)\right\|^2\right) d r\right).
  \end{align*}
  Taking $\tau=\beta_{-T} \sigma$, we have 
  \begin{align*}
    \begin{aligned}
    e^{-c_0 T} \mathcal{E}(U(0)) & \leq C e^{-c_0 T}+C e^{-c_0 T} \int_{-\infty}^0 e^{c_0 r}\left(1+\left\|F\left(\beta_{r-T} \sigma\right)\right\|^2\right) d r \\
    & \leq C\left(1+\int_{-\infty}^0 e^{c_0 q}\left(1+\left\|F\left(\beta_q \sigma\right)\right\|^2\right) d q\right).
    \end{aligned}
  \end{align*}
  Therefore, one has 
  \begin{align}\label{e.060435}
      \mathcal E(U(T))
      +\int_0^T e^{-c_0(T-s)}\mathcal E(U(s))ds
      \leq B_M(\sigma),
  \end{align}
  where 
  \begin{align*}
    B_M(\sigma)=C\left(1+M Q_{0, \infty}^2+\int_{-\infty}^0 e^{c_0 r}\left[1+\left\|F\left(\beta_r \sigma\right)\right\|^2+\left\|Q\left(\beta_r \sigma\right)\right\|_{\mathcal{L}_2\left(\mathbb{R}^{\mathbf{d}}, H\right)}^2\right] d r\right)<\infty
  \end{align*}
  and the constant $C>0$ does not depend on $\sigma, T, M$ and the controlled path. 

  Let $P_n^\perp:=I-P_n$ and define the phase-space projection $\Pi_{n}^X(u,v):=(P_nu,P_nv)$. 
  We prove that
  \begin{align}\label{e.060436}
      \sup_{Z\in\mathcal R_M(\sigma)}\|(I-\Pi_{n}^X)Z\|_X\longrightarrow0
      \qquad\text{as }n\to\infty .
  \end{align}
  Fix a controlled path $U=(u,u_t)$ and write $U_n^\perp=(P_n^\perp u,P_n^\perp u_t)$.  Applying the same damped-wave energy estimate to the projected equation gives
  \begin{align}\label{e.060437}
      \|U_n^\perp(T)\|_X^2
      &\leq C e^{-c_0T}\|U_n^\perp(0)\|_X^2\nonumber\\
      &\quad +C\int_0^T e^{-c_0(T-s)}
      \Big(\|P_n^\perp F(\beta_{s-T}\sigma)\|^2
      +\|P_n^\perp Q(\beta_{s-T}\sigma)\varphi(s)\|^2
      +\|P_n^\perp\sin u(s)\|^2\Big)ds .
  \end{align}
  Since we may extend the path backward by an arbitrary zero-control segment, the first term in \eqref{e.060437} can be made arbitrarily small uniformly over the reachable endpoint, using the tempered pullback absorbing bound for the attractor.

  For the force term, since $\|F\|_1^2\in L^1(\Sigma,m)$,
  \begin{align*}
      \int_0^T e^{-c_0(T-s)}\|P_n^\perp F(\beta_{s-T}\sigma)\|^2ds
      \leq \frac1{\lambda_{n+1}}
      \int_{-\infty}^0 e^{c_0r}\|F(\beta_r\sigma)\|_1^2dr
      \longrightarrow0 .
  \end{align*}
  For the controlled noise term, \eqref{e.060404} gives
  \begin{align*}
      \int_0^T e^{-c_0(T-s)}\|P_n^\perp Q(\beta_{s-T}\sigma)\varphi(s)\|^2ds
      \leq \frac{Q_{1,\infty}^2}{\lambda_{n+1}}
      \int_0^T\|\varphi(s)\|_{\Rb^{\mathbf{d}}}^2ds
      \leq \frac{2MQ_{1,\infty}^2}{\lambda_{n+1}} .
  \end{align*}
  Finally, the chain rule gives $\sin u\in V$ and $\|\sin u\|_1\leq \|u\|_1$, 
  hence by \eqref{e.060435},
  \begin{align*}
      \int_0^T e^{-c_0(T-s)}\|P_n^\perp\sin u(s)\|^2ds
      \leq \frac1{\lambda_{n+1}}
      \int_0^T e^{-c_0(T-s)}\|u(s)\|_1^2ds
      \leq \frac{B_M(\sigma)}{\lambda_{n+1}} .
  \end{align*}
  Combining these estimates proves \eqref{e.060436}.  Since the projected set $\Pi_{n}^X\mathcal R_M(\sigma)$ is bounded in the finite-dimensional space $P_nV\times P_nH$, it is precompact.  The uniform tail estimate \eqref{e.060436} implies that $\mathcal R_M(\sigma)$ is precompact in $X$.  Therefore $K_M(\sigma)$ is compact by a similar argument as in the proof of Lemma \ref{l.070301}.

\end{proof}

We now verify weak exponential tightness.

\begin{lem}
  Assume \eqref{e.060403} and \eqref{e.060404}.  If $\{\mu_\sigma^\eps\}$ is the stationary family obtained in Theorem \ref{t.060402}, then
  \begin{align*}
  \lim_{R\to\infty}\limsup_{\eps\to0} \eps\ln \mu_\sigma^\eps\big(\{U\in X:\|U\|_X>R\}\big)=-\infty, \qquad m\text{-a.s.}
\end{align*}
\end{lem}

\begin{proof}
  The proof is parallel to the Navier--Stokes case, with $\mathcal E$ replacing $\|u\|^2$. Put $f(\sigma):=1+\|F(\sigma)\|^2$ and $\gamma:=\frac{c_0}{2}.$
  By \eqref{e.060404}, one has
  \begin{align*}
    \|Q(\sigma)\|_{\mathcal L_2(\Rb^{\mathbf d},H)}^2\leq \mathbf d Q_{0,\infty}^2,\qquad m\text{-a.s.}.
  \end{align*}
  From \eqref{e.060410},
  \begin{align}\label{e.060438}
    d\mathcal E(U(t))+c_0\mathcal E(U(t))dt\leq C f(\beta_t\sigma)dt+C\eps dt+dM_t,
  \end{align}
  and
  \begin{align}\label{e.060439}
    d\langle M\rangle_t\leq C\eps Q_{0,\infty}^2\mathcal E(U(t))dt.
  \end{align}
  Taking expectations in the pullback form of \eqref{e.060438}, using the stationarity relation and letting the pullback time tend to infinity, we obtain
  \begin{align}\label{e.060440}
    \int_X\mathcal E(U)\mu_\sigma^\eps(dU)\leq K_0(\sigma):=C\int_{-\infty}^0e^{\gamma r}f(\beta_r\sigma)dr,\qquad \eps\in(0,1],
  \end{align}
  where $K_0\in L^1(\Sigma,m)$.

  We next prove an exponential pullback estimate. Choose $a>0$ small enough such that $CaQ_{0,\infty}^2\leq c_0/4$, and define $a_T(t):=ae^{-\gamma(T-t)}$ for $0\leq t\leq T$. Then $a_T(T)=a$, $a_T(0)=ae^{-\gamma T}$, and
  \begin{align}\label{e.060441}
    a_T'(t)-c_0a_T(t)+CQ_{0,\infty}^2a_T(t)^2\leq0.
  \end{align}
  Let
  \begin{align*}
    X(t):=\mathcal E(U(t)),\qquad Z(t):=\exp\left(\frac{a_T(t)X(t)}{\eps}\right).
  \end{align*}
  By Itô's formula, \eqref{e.060438}--\eqref{e.060439} and \eqref{e.060441}, we have
  \begin{align*}
    dZ(t)&\leq Z(t)\left[\frac{Ca_T(t)}{\eps}f(\beta_{t-T}\sigma)+Ca_T(t)\right]dt+Z(t)\frac{a_T(t)}{\eps}dM_t\\
    &=Z(t)d\Lambda_T(t)+Z(t)\frac{a_T(t)}{\eps}dM_t,
  \end{align*}
  where
  \begin{align*}
    \Lambda_T(t):=\frac{C}{\eps}\int_0^t a_T(s)f(\beta_{s-T}\sigma)ds+C\int_0^t a_T(s)ds.
  \end{align*}
  The integration factor gives
  \begin{align*}
    d\left(e^{-\Lambda_T(t)}Z(t)\right)\leq e^{-\Lambda_T(t)}Z(t)\frac{a_T(t)}{\eps}dM_t.
  \end{align*}
  Let $\tau_N$ be a localizing sequence for the stochastic integral on the right. Taking expectation on $[0,T\wedge\tau_N]$, letting $N\to\infty$ and using Fatou's lemma yield
  \begin{align*}
    \mathbf E\left(e^{-\Lambda_T(T)}Z(T)\right)\leq Z(0).
  \end{align*}
  Hence
  \begin{align}\label{e.060442}
    \Eb\exp\left(\frac{a\mathcal E(\Phi_{0,T,\beta_{-T}\sigma}^{\eps}U)}{\eps}\right)\leq \exp\left(\frac{ae^{-\gamma T}\mathcal E(U)}{\eps}+\frac{K_0(\sigma)}{\eps}+C\right).
  \end{align}

  For $R,T>0$, define
  \begin{align*}
    G_{T,R}:=\left\{U\in X:e^{-\gamma T}\mathcal E(U)\leq \frac{R^2}{4}\right\}.
  \end{align*}
  By Chebyshev's inequality and \eqref{e.060442}, for $U\in G_{T,R}$,
  \begin{align}\label{e.060443}
    \mathbf P\left\{\mathcal E(\Phi_{0,T,\beta_{-T}\sigma}^{\eps}U)>R^2\right\}\leq \exp\left(-\frac{3aR^2/4-K_0(\sigma)}{\eps}+C\right).
  \end{align}
  On the other hand, by \eqref{e.060440},
  \begin{align}\label{e.060444}
    \mu_{\beta_{-T}\sigma}^\eps(G_{T,R}^c)\leq \frac{4e^{-\gamma T}}{R^2}\int_X\mathcal E(U)\mu_{\beta_{-T}\sigma}^\eps(dU)\leq \frac{4e^{-\gamma T}}{R^2}K_0(\beta_{-T}\sigma).
  \end{align}
  Therefore, by stationarity, \eqref{e.060443} and \eqref{e.060444},
  \begin{align}\label{e.060445}
    \mu_\sigma^\eps(\mathcal E>R^2)\leq \frac{4e^{-\gamma T}}{R^2}K_0(\beta_{-T}\sigma)+\exp\left(-\frac{3aR^2/4-K_0(\sigma)}{\eps}+C\right).
  \end{align}

  Since $K_0\in L^1(\Sigma,m)$, the Borel--Cantelli lemma implies that, for every $\eta>0$ and $m$-a.e. $\sigma\in\Sigma$, there exists $N_\eta(\sigma)$ such that
  \begin{align}\label{e.060446}
    K_0(\beta_{-n}\sigma)\leq e^{\eta n},\qquad n\geq N_\eta(\sigma).
  \end{align}
  Taking $T=T_{\eps,R}:=\left\lceil R^2/\eps\right\rceil$ and using \eqref{e.060445}--\eqref{e.060446}, we obtain for fixed $R>0$,
  \begin{align*}
    \limsup_{\eps\to0}\eps\ln \mu_\sigma^\eps(\mathcal E>R^2)\leq \max\left\{-(\gamma-\eta)R^2,\ K_0(\sigma)-\frac{3aR^2}{4}\right\}.
  \end{align*}
  Choosing $\eta\in(0,\gamma)$ and letting $R\to\infty$, we get
  \begin{align*}
    \lim_{R\to\infty}\limsup_{\eps\to0}\eps\ln \mu_\sigma^\eps(\mathcal E>R^2)=-\infty,\qquad m\text{-a.s.}.
  \end{align*}
  Finally, by \eqref{e.060408}, there is $c_E>0$ such that $c_E\|U\|_X^2\leq\mathcal E(U)$. Therefore
  \begin{align*}
    \{U\in X:\|U\|_X>R\}\subset\left\{U\in X:\mathcal E(U)>c_E R^2\right\},
  \end{align*}
  and the desired weak exponential tightness follows by replacing $R$ by $\sqrt{c_E}R$ in the preceding estimate.
\end{proof}

\bibliographystyle{amsplain}
\bibliography{ref}

\end{document}